\documentclass[english,a4paper,11pt]{amsart}
\RequirePackage[T1]{fontenc}
\RequirePackage{amsthm,amsmath}
\usepackage{amsthm,amsmath}
\usepackage{hyperref}
\usepackage[a4paper]{geometry}
\usepackage{cite}
\allowdisplaybreaks[4]
\usepackage{orcidlink}
\theoremstyle{plain}
\newtheorem{theorem}{Theorem}[section]
\newtheorem{lemma}[theorem]{Lemma}
\newtheorem{corollary}[theorem]{Corollary}
\newtheorem{proposition}[theorem]{Proposition}

\newtheorem{claim}[theorem]{Claim}

\theoremstyle{definition}
\newtheorem{definition}[theorem]{Definition}

\newtheorem{hypo}[theorem]{Hypothesis}

\theoremstyle{remark}
\newtheorem{remark}[theorem]{Remark}

\renewcommand{\emph}[1]{{\sl #1}}
\numberwithin{equation}{section}

\newcounter{lil1}
\newenvironment{step}
{\begin{list} { \bf Step (\Roman{lil1})}
		{ \usecounter{lil1}
			\setlength{\leftmargin}{0.0cm}
			\setlength{\topsep}{0.2cm}
			\setlength{\itemsep}{0.0cm}
			\setlength{\parsep}{0.1cm}
			\setlength{\itemindent}{0.8cm}
			\setlength{\parskip}{0.0cm}}}
	{\end{list}}

\newcounter{lil33}

\newcounter{lil1q}

\newcommand{\MA}{\mathfrak{A}}
\newcommand{\wh}{\widehat}

\usepackage{dsfont}
\usepackage{stmaryrd}
\usepackage{color}

\newcommand{\wi}[1]{{\widetilde{#1}}}

\newcommand{\baray}{\begin{array}{rcl}}
	\newcommand{\earay}{\end{array}}
\newcommand{\barray}{\begin{array}{rcl}}
	\newcommand{\earray}{\end{array}}

\newcommand\dela[1]{}

\newcommand{\bcase}{\begin{cases}}
	\newcommand{\ecase}{\end{cases}}

\def\wi{\widetilde}

\newcommand{\DeltaA}{A}
\newcommand\del[1]{}

\del{

	\newcommand\del[1]{}
	}

\newcommand{\Law}{\mbox{Law}}

\newcommand{\lk}{\left}
\newcommand{\lqq}{\lefteqn}
\newcommand{\rk}{\right}
\newcommand{\la}{\langle}
\newcommand{\ra}{\rangle}

\newcommand{\LL}{{\rm I \kern -0.2em L}}

\newcommand{\ep} {\varepsilon }

\newcommand{\be} {\begin{enumerate} }
	\newcommand{\ee} {\end{enumerate} }

\newcommand{\CO}{{{ \mathcal O }}}

\newcommand{\CA}{{{ \mathcal A }}}

\newcommand{\CG}{{{ \mathcal G }}}

\newcommand{\CB}{{{ \mathcal B }}}

\newcommand{\CF}{{{ \mathcal F }}}

\newcommand{\CN}{{{ \mathcal N }}}

\newcommand{\RR}{{\mathbb{R}}}

\newcommand{\NN}{\mathbb{N}}

\newcommand{\PP}{{\mathbb{P}}}

\newcommand{\EE}{ \mathbb{E} }

\newcommand{\DEQS}{\begin{eqnarray*} }
	\newcommand{\EEQS}{\end{eqnarray*} }
\newcommand{\DEQSZ}{\begin{eqnarray} }
	\newcommand{\EEQSZ}{\end{eqnarray} }
\newcommand{\DEQ}{\begin{eqnarray}}
	\newcommand{\EEQ}{\end{eqnarray}}

\newcommand{\Fcal} {{\mathcal F}}
\newcommand{\Gcal} {{\mathcal G}}

\newcommand{\Lcal} {{\mathcal L}}
\newcommand{\Mcal} {{\mathcal M}}

\newcommand{\Ocal} {{\mathcal O}}

\newcommand{\Vcal} {{\mathcal V}}

\newcommand{\Xcal} {{\mathcal X}}

\newcommand{\Afrak} {{\mathfrak A}}

\newcommand{\R}{\mathbb{R}}
\newcommand{\N}{\mathbb{N}}

\renewcommand{\P}{\mathbb{P}}
\newcommand{\Eb}{\mathbb{E}}

\newcommand{\X}{\mathbb{X}}
\newcommand{\D}{\mathbb{D}}
\renewcommand{\emph}[1]{{\sl #1}}

\usepackage{amsmath}
\usepackage{amssymb, color}
\usepackage{latexsym}

\usepackage{ulem}
\usepackage{mathrsfs}

\usepackage{enumerate}

\usepackage{cancel}

\usepackage[T1]{fontenc}
\usepackage{lipsum}
\usepackage{amsfonts}
\usepackage{graphicx}
\usepackage{epstopdf}
\usepackage{algorithmic}
\usepackage{bbm}

\usepackage{todonotes}
\usepackage{url}

\definecolor{amethyst}{rgb}{0.6, 0.4, 0.8}
\newcommand{\Pro}{Pro}

\AtBeginDocument{  }

\usepackage{graphicx}
\usepackage{tikz}
\usepackage{atbegshi}

\AtBeginShipoutFirst{%
    \begin{tikzpicture}[remember picture, overlay]
        \node[anchor=north west, xshift=1in, yshift=-1.5cm] at (current page.north west) {
            \includegraphics[width=2cm]{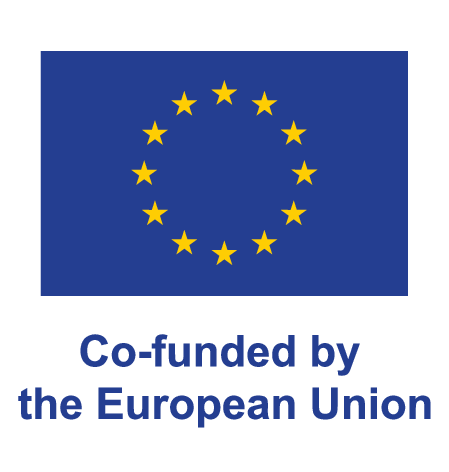} 
        };
    \end{tikzpicture}%
}

\begin{document}

	\title[Stochastic Schauder-Tychonoff type theorem]{A stochastic Schauder-Tychonoff type theorem and its applications}
	\author[E. Hausenblas]{Erika Hausenblas}
	\address{Montanuniversit\"{a}t Leoben\\
		Department Mathematik und Informationstechnologie\\
		Franz Josef Stra{\ss}e 18\\
		8700 Leoben\\
		Austria}
	\email{erika.hausenblas@unileoben.ac.at}

		\author[A. Kumar]{Ankit Kumar$^\ast$}
	\address{Montanuniversit\"{a}t Leoben\\
		Department Mathematik und Informationstechnologie\\
		Franz Josef Stra{\ss}e 18\\
		8700 Leoben\\
		Austria}
	\email{ankit.kumar@unileoben.ac.at, ankitkumar.2608@gmail.com}

		\author[J. M. T\"olle]{Jonas M. T\"{o}lle}
	\address{Aalto University\\
		Department of Mathematics and Systems Analysis\\
		PO Box 11100 (Otakaari 1, Espoo)\\
		00076 Aalto\\
		Finland}
	\email{jonas.tolle@aalto.fi}
	\date{\today}
	\begin{abstract}
		One standard way to prove existence for deterministic, highly nonlinear PDEs is to use the Schauder–Tychonoff fixed-point theorem. In what follows, we introduce and verify a stochastic variant of the Schauder–Tychonoff theorem.  We apply our existence result to nonlinear stochastic diffusion equations with non-Lipschitz perturbations.
	\end{abstract}
	

	\thanks{E.H. gratefully acknowledges the support by the Austrian Science Fund, Project
		number: P34681. A.K. acknowledges partial funding by the Austrian Science Fund (FWF), Project
		Number: P34681, P32295 and \href{https://www.fwf.ac.at/forschungsradar/10.55776/ESP4373225}{10.55776/ESP4373225}.
		The research of J.M.T. was partially supported by the European Union's Horizon Europe research and innovation programme under the Marie Sk\l{}odowska-Curie Actions Staff Exchanges (Grant agreement no.~101183168 -- LiBERA, Call: HORIZON-MSCA-2023-SE-01).\\
		This work is licensed under the Creative Commons Attribution 4.0 International License. To view a copy of this license, visit \url{http://creativecommons.org/licenses/by/4.0/} or send a letter to Creative Commons, PO Box 1866, Mountain View, CA 94042, USA.  \includegraphics[height=1em]{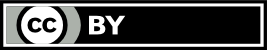}
Original work to appear in Analysis and Mathematical Physics (2026) \url{https://doi.org/10.1007/s13324-026-01262-y}.\\$^\ast$Corresponding author.}

	\maketitle
	
	{\small \tableofcontents}

	\section{Introduction}
	
	Nonlinear stochastic coupled systems arise naturally in 
	a broad spectrum of applications in biology, chemistry, physics, and engineering; for more details, we refer to \cite{FKEHMH,EHMAHTAT,EHBJM2,EHBJMPAR,EHDMTT2,EHDMTT1,EHJMT,DMEHAZ}, and references therein. 
	In mathematical biology, these systems describe phenomena such as chemotaxis, 
	population dynamics with cross-diffusion, intracellular signalling, or tissue 
	growth under random fluctuations, see  \cite{MBFHAJ,boundednessbyentropymethod,HausenblasPanda1,EHBJMPAR,EHBJM2,levin2,segel}. In biochemistry \cite{DebusscheHoegeleImkeller2013,Dillon,kepper,FKEHMH,grayscott,klaus1,klausmeier,Maini,murray1,murray2,Perthame,Sherratt:2005,Sherratt:2010,Sherratt:2011,sherratt,skt,turing,nadia,vanderStelt:2012gs,Woolley}, reaction--diffusion networks 
	with stochastic forcing are used to model molecular transport, enzyme kinetics, 
	and pattern formation at low copy numbers. Building on the recent work \cite{AAMV1}, which established global well-posedness for reaction--diffusion equations with transport noise and critical superlinear diffusion using stochastic maximal regularity, as well as the work \cite{AAMV3}, which introduces a new local Lipschitz condition for the critical variational framework of SPDEs, Agresti and Veraar further expanded on these developments in \cite{AAMV2} with a comprehensive survey of nonlinear SPDEs and maximal regularity.
	
	In fluid mechanics and porous media 
	theory, stochastic perturbations appear in models for transport in random 
	environments, flow through heterogeneous materials, or chemically active 
	fluids, see \cite{DGG2020,Gess2017,DMEHAZ,vazquez2007porous}. Related classes of equations also occur in neuroscience (FitzHugh--Nagumo and Hodgkin--Huxley models), epidemiology, climate modelling, 
	and in the study of soft matter and active particle systems, see \cite{RF,AGHM,GrayScott_original,ALH,EFK,JNSA,YY} for more details. These examples 
	highlight the fundamental role of nonlinear stochastic parabolic systems as 
	accurate and versatile models for complex multiscale processes influenced by 
	randomness. There is a huge literature on the solutions of related equations, for instance see \cite{BDPR3,BDPR2016,BM2013,Cao,Capinski1993,HE,DGG2020,evans2012introduction,FehrmanGess2021,Flandoli,franco,Gess2017,kotelenez,TessitoreZabczyk1998}, and references therein. 
	
	There are several different kinds of solutions for stochastic systems, including classical solutions \cite{GDPMR}, mild solutions \cite{JBW}, Dirichlet solutions \cite{SAMR}, variational solutions \cite{GDPMR}, or solutions via regularity structures and rough paths \cite{MH2014}, \cite{JLL,MGCPLL,ethier} built on a given probability space, and stochastically weak solutions \cite{MV}. 
	Our main focus is to establish the existence of a solution defined in the sense of Definition \ref{def2.2}.

	We consider a class of nonlinear, possibly strongly coupled, stochastic parabolic 
	systems of the form
	\begin{equation}\label{e:SPDE}
		\begin{cases}
			du(t) +A_1(u(t),v(t))\,dt 
			= G_1(u(t),v(t))\,dt + \sigma_1(u(t),v(t))\,dW_1(t), \\[6pt]
			dv(t) + A_2(u(t),v(t))\,dt 
			= G_2(u(t),v(t))\,dt + \sigma_2(u(t),v(t))\,dW_2(t),
		\end{cases}
	\end{equation} for $t\in(0,T)$,
	subject to suitable initial and boundary conditions.
	Here
	\begin{itemize}
		\item $u$ and $v$ are unknown processes taking values in separable Hilbert spaces 
		$H_1$ and $H_2$, respectively;
		\item $A_1:D(A_1)\subset V_1 \to V_1^\ast$ and 
		$A_2:D(A_2)\subset V_2 \to V_2^\ast$ are (possibly nonlinear) 
		operators defined on Gelfand triplets
		\[
		V_1 \hookrightarrow H_1 \equiv H_1^\ast \hookrightarrow V_1^\ast,
		\qquad
		V_2 \hookrightarrow H_2 \equiv H_2^\ast \hookrightarrow V_2^\ast;
		\]
		\item $G_1:V_1\times V_2 \to V_1^\ast$ and 
		$G_2:V_1\times V_2 \to V_2^\ast$ denote lower-order (possibly nonlinear) 
		drift terms;
		\item $\sigma_1:V_1\times V_2\to \mathcal{L}_2(U_1,H_1)$ and 
		$\sigma_2:V_1\times V_2 \to \mathcal{L}_2(U_2,H_2)$ are noise coefficients 
		with values in the spaces of Hilbert--Schmidt operators from separable 
		Hilbert spaces $U_1$ and $U_2$ into $H_1$ and $H_2$, respectively;
		\item $\{W_1(t)\}_{t\in[0,T]}$ and $\{W_2(t)\}_{t\in[0,T]}$ are cylindrical Wiener processes 
		on $U_1$ and $U_2$, respectively, defined on a filtered probability space 
		$\big(\Omega,\mathcal{F},\{\mathcal{F}_t\}_{t\in[0,T]},\mathbb{P}\big)$ satisfying the 
		usual conditions.
	\end{itemize}

When passing from deterministic to stochastic partial differential equations (SPDEs), one quickly encounters a number of difficulties that are not present in the 
classical setting. Fundamental identities that hold deterministically may fail in the stochastic framework, for instance, expectations do not in general 
factorise, so that $\mathbb{E}[uv] \neq \mathbb{E}[u]\mathbb{E}[v]$, and 
integrability assumptions such as $\mathbb{E}\big[\|v\|\big] < \infty$ do not automatically 
imply higher moments. Moreover, many structural tools commonly used in 
deterministic analysis such as local monotonicity methods or maximum principles may 
no longer apply in the presence of randomness. Uniform a priori bounds, while 
still obtainable in many situations, often require more delicate arguments, and 
non-Lipschitz nonlinearities typically need to be handled via suitable 
interpolation or cut-off procedures.

Despite these additional challenges, a general strategy for establishing the 
existence of probabilistic weak solutions has emerged and is now widely used (see \cite{CP1997,weiroeckner}). 
One usually begins by constructing a sequence of approximate solutions, either 
through a Galerkin scheme or a time-discretisation method. These approximations 
must satisfy uniform bounds that are strong enough to guarantee the tightness of the 
induced probability laws. Once tightness is established, the Skorokhod 
representation theorem plays a central role: it provides a new probability space 
on which a subsequence of the approximations converges almost surely, with 
respect to appropriately ``reconstructed’’ Wiener processes. The final and often 
most delicate step consists of analysing the limiting objects and showing that 
the limit indeed satisfies the original SPDE in the desired weak or 
martingale sense. For this particular part, we combine a standard identification procedure with the approach given in  \cite{weiroeckner}.

In the present work, we take a different route. Instead of relying on a 
discretisation-based compactness argument, we formulate the existence theory as a 
fixed-point problem on the level of probability laws. By applying the 
Schauder--Tychonoff theorem directly to the laws of appropriately parametrised 
solution candidates, we obtain a modular and conceptually transparent framework that is independent of the particular approximation scheme. This approach 
isolates the essential probabilistic and topological ingredients and clarifies 
the exact analytical requirements needed for the fixed-point argument to close. 
As a consequence, the method applies to a broad class of nonlinear stochastic 
systems of type~\eqref{e:SPDE} without the need to verify discretisation-specific 
properties, thereby simplifying the existence theory while retaining full 
generality.

The article is structured as follows. In Section~\ref{schauder} we present the main result, and its proof is given in Section~\ref{sec:proof_of_schauder}.
To demonstrate its applicability, we present in Section  \ref{sec:app}  several examples for which we can establish the existence of probabilistic weak solutions.	As a first illustration, we consider an SPDE with a nonlinear Nemytskii operator (see \cite{runst}) as the drift. To be more precise,  in Subsection \ref{Example1} we verify the assumption of our main theorem for the following SPDEs
\begin{equation*}
	\left\{
	\begin{aligned}
		d u(t)&= \big(\Delta u(t)+u^{[\frac{1}{2}]}(t)\big)\,d t+\Sigma (u(t))\,dW(t), \quad \text{in } (0,T)\times \Ocal,\\
		u(0)&=u_0,\ \text{in }\ \Ocal,
	\end{aligned}
	\right.
\end{equation*}where $u^{[\gamma]}:=|u|^{\gamma-1}u$ with $\gamma \in (0,\infty)$, and $\mathcal{O}=(0,1)$. 

Next, in Subsection~\ref{Example2} we show that the following porous-media type model fits into the framework of Theorem~\ref{ther_main}:

\begin{equation*}
	\left\{
	\begin{aligned}
		du(t)&= \big(\Delta u^{[\gamma]}(t)+u^{[\frac{1}{2}]}(t)\big)\,d t+ \Sigma(u(t))\,dW(t),\quad \text{in } (0,T)\times \Ocal,\\
		u(0)&=u_0,\ \text{in }\ \Ocal, 
	\end{aligned}
	\right.
\end{equation*}
where \(u^{[\gamma]}:=|u|^{\gamma-1}u\) with \(\gamma\in(0,\infty)\), and \(\Ocal=(0,1)\).
Finally,  in Subsection~\ref{Example3} 
 we consider the case in which the nonlinearity appears in front of the noise:
\begin{equation*} 
	\left\{
	\begin{aligned}
		du(t)&= \Delta u^{[\gamma]}(t)\,d t+ \nabla u^{[\frac{1}{2}]}(t)\,dW(t),\quad \text{in } (0,T)\times \Ocal,\\
		u(0)&=u_0,\quad \text{in } \Ocal, 
	\end{aligned}
	\right.
\end{equation*}
where again \(u^{[\gamma]}:=|u|^{\gamma-1}u\) with \(\gamma\in(0,\infty)\), and \(\Ocal=(0,1)\).
Here the stochastic integral is understood in the appropriate Hilbert-space sense, and the coefficient in front of the noise is nonlinear in \(u\). {One should note that $\mathbb{X}$ is a generic Banach space, which may change from one example to another.}
In the appendix, we collect several auxiliary results that are essential for the proof of the main theorem. In particular, we define extensions of probability spaces in Appendix~\ref{extension}, introduce the shifted Haar projection in Appendix~\ref{app:haar-system}, and prove the infinite-dimensional L\'evy–Ciesielski theorem for the Brownian motion in Appendix~\ref{LC:sec}.

	\section{The stochastic Schauder-Tychonoff type theorem}\label{schauder}

		Let us fix some notation in the sense of \cite[Chapter 5]{weiroeckner}. 
	Let $\Afrak=(\Omega,\Fcal,\mathbb{F},\P)$ be a filtered probability space
	with filtration $\mathbb{F}=\{\Fcal_{t}\}_{t\in [0,T]}$ satisfying the usual conditions.
	Let us consider a Gelfand triple $V\subset H\subset V^*$, where $H$ denotes a separable Hilbert space, and $V$ denotes a reflexive Banach space, and $V^*$ represents the topological dual of $V$. For a given $m$\footnote{We will see later which $m$ to choose.}, we fix $\X=L^m(0,T;V)$. Let $\X_1\subset \X$ be a reflexive Banach space embedded compactly and densely in $\X$.
	Now, we define
		\begin{equation*}
		\begin{aligned}  \Mcal_{\Afrak}^{m}(\X)
			:= & \Big\{ \xi:\Omega\times[0,T]\to V\;\colon\;\\
			&\qquad\text{\ensuremath{\xi} is \ensuremath{\mathbb{F}}-progressively measurable}\;\text{and}\;\EE\big[\|\xi\|_{\X}^{m}\big]<\infty\Big\},
		\end{aligned}
	\end{equation*}
	equipped with the norm
	\[
	\|\xi\|_{\Mcal_{\Afrak}^{m}(\X)}:=\Big\{\EE\big[\|\xi\|_{\X}^{m}\big]\Big\}^{1/m},\quad\xi\in\Mcal_{\Afrak}^{m}(\X).
	\]
	
	Let $U$ be another separable Hilbert space and $\{W(t)\}_{t\in [0,T]}$ be a Wiener process\footnote{That is, a $Q$-Wiener process, see e.g. \cite{DaPrZa:2nd} for this notion.} in $U$ with a linear, nonnegative definite, symmetric trace class covariance
	operator $Q:U\to U$ such that $W$ has the representation
	$$
	W(t)=\sum_{i\in\mathbb{I}} Q^\frac 12 \psi_i\beta_i(t),\quad t\in [0,T],
	$$
	where $\{\psi_i:i\in \mathbb{I}\}$ is a complete orthonormal system in $U$, $\mathbb{I}$ a suitably chosen countable index set, and $\{\beta_i:i\in\mathbb{I}\}$ a family of independent real-valued standard Brownian motions on $[0,T]$ modelled in given probability space $\Afrak$.

Our point of interest is a system of the following form
		\begin{align} \label{spdes_org}
		dw(t) =&\big[{A} (w(t))+ {F(w(t))}\big]\, dt +\Sigma(w(t))\,dW(t),\quad w(0)=w_0\in L^m(\Omega;H),
	\end{align}
	where $A$ is a possible nonlinear operator and $F$ a non-linear function.  There are several ways of showing the existence of a solution to the equation \eqref{spdes_org}. However, we are interested in a way to do it as a fixed-point method using a Schauder-Tychonoff-type theorem based on compactness.

	Before turning to the main result of this article, we first clarify the notion of 
	a solution that will be used throughout. As explained above, our approach relies 
	on compactness arguments formulated on the level of probability measures. In this 
	framework, one no longer has direct access to the original probability space or 
	to the driving Wiener processes. Consequently, the limiting object obtained via 
	compactness cannot be interpreted as a solution in the classical (probabilistic 
	strong) sense. Instead, one must reconstruct a probability space on which both 
	the solution and the corresponding Wiener processes are defined.
	
	\renewcommand{\emph}[1]{{\sl #1}}
	
	For this reason, we begin by recalling what we mean by a \emph{probabilistic 
		strong solution}. After that, we introduce the notion of a 
	\emph{probabilistic weak solution}, which is the appropriate concept for the 
	compactness-based arguments used later.

	\begin{definition}[Probabilistic strong solution]
		We are given a stochastic basis $\Afrak:=\big(\Omega,\Fcal,\mathbb{F},\P\big)$ and $w_0\in L^m(\Omega;H)$. Then, the system \eqref{spdes_org} has a pathwise strong probabilistic solution if and only if there exists a progressively measurable process $w:[0,T]\times \Omega\to H$ with $\P$-a.s., paths $w(\cdot,\omega)\in \mathbb{D}([0,T];H)$, and $w$ satisfies 
		for all $t\in[0,T]$ and for any $\phi\in V$, $\P$-a.s.,
		\begin{align*}
			\big(w(t),\phi\big)&=\big(w_0,\phi\big)+\int_0^t \langle A(w(s))+ F(w(s)), \phi\rangle \,ds+\int_0^t\big(\Sigma(w(s))\, dW(s),\phi\big).
		\end{align*}
	\end{definition}
	Having introduced the notion of a probabilistic strong solution, we can now proceed to define the corresponding probabilistic weak solution.
	\begin{definition}[Probabilistic weak solution]\label{def2.2}
		A probabilistic weak solution to the problem  \eqref{spdes_org} is a system 
		\begin{align*}
			\big(\Omega,\Fcal,\mathbb{F},\P, W,w \big), 
		\end{align*}such that 
		\begin{enumerate}
			\item $\Afrak:=\big(\Omega,\Fcal,\mathbb{F},\P\big)$ is a complete filtered probability space with a filtration $\mathbb{F}=\{\Fcal_t\}_{t\geq0}$ satisfying the usual conditions;
			\item $W$ is a cylindrical Wiener process on $U$, over the probability space $\Afrak$; 			
			\item {$w : [0,T]\times \Omega \to H$ is $\mathbb{F}$-progressively measurable with
			$w(\cdot,\omega)\in \mathbb{D}([0,T];H)$ for $\P$-a.e.\ $\omega$, and satisfies
			$A(w),F(w)\in L^{1}(0,T;V^{\ast})$ and
			$\Sigma(w)\in L^{2}(0,T;\Lcal_2(U,H))$, $\P$-a.s. Moreover, for every
			$\phi\in V$ and all $t\in[0,T]$, $\P$-a.s.,
			\begin{align*}
				\big(w(t),\phi\big)
				&=\big(w_0,\phi\big)
				+\int_0^t\langle A(w(s))+F(w(s)),\phi\rangle\,ds
				+\int_0^t\big(\Sigma(w(s))\,dW(s),\phi\big).
			\end{align*}}
		\end{enumerate}
	\end{definition}
{In the infinite-dimensional setting, existence of a probabilistic weak solution
together with pathwise uniqueness yields uniqueness in law and a unique strong
solution by the Yamada--Watanabe theorem for stochastic evolution equations, see
\cite{RSZ2008,Kurtz2007}.}
	
	In this article, we start with the following stochastic partial differential equations (SPDEs):
	\begin{equation} \label{spdes1}
		\left\{
		\begin{aligned}
			dw(t) &=\big[A(w(t))+ F(w(t))\big]\, dt  +\Sigma(w(t))\,dW(t), \ t\in(0,T),\\
			w(0)&=w_0\in L^m(\Omega; H).
		\end{aligned}\right.
	\end{equation}
	For fixed $\Afrak$, $W$, {$m> 1$}, we define the operator 
	$$
	\Vcal=\Vcal_{\Afrak,W}:\Mcal_{\Afrak}^{m}(\X)\times L^m(\Omega,\Fcal_0,\P;H)\to\Mcal_{\Afrak}^{m}(\X)
	$$
	for $\xi\in\Mcal_{\MA}^m(\X)$ by
	\begin{align*}
		\Vcal(\xi):=\Vcal(\xi,w_0):=w,
	\end{align*} 
	where $w$ is the solution to the following SPDEs
	\begin{equation} \label{spdes}
		\left\{
		\begin{aligned}
				dw(t) &=\big[A(w(t))+ { \bar F_1(\xi(t),w(t))}\ +F_2(\xi (t))\big]\, dt  \\&\qquad +\big(\Sigma_1(w(t))+\Sigma_2(\xi(t))\big)\,dW(t), \ t\in(0,T),\\
				 w(0)&=w_0\in L^m(\Omega; H).
		\end{aligned}\right.
	\end{equation}
	Here we decomposed $F$ as a sum $\bar F_1+F_2$ such that $\bar F_1(w,w)=F_1(w)$, and $\Sigma$ as $\Sigma_1+\Sigma_2$, details are given in Hypothesis \ref{hyp}. This kind of decomposition allows us to cover several interesting examples in this framework.

	Here, we implicitly assume that \eqref{spdes} is well-posed and a unique strong solution (in the stochastic sense) $w\in\Mcal_{\MA}^m(\X)$ exists for $\xi\in\Mcal_{\MA}^m(\X)$.
However, this need not always be the case, and one must impose certain constraints on $\xi$.  	
Hence, 	we assume that there exists a bounded subset $\mathcal{X}_\MA$ of $\Mcal_{\MA}^m(\X)$
	such that for any $\xi\in\mathcal{X}_\MA$
	there exists  a unique strong solution  $w(\xi)$ (in the stochastic sense) to the system \eqref{spdes}. Later on, we will see that we need $w(\xi)\in \mathcal{X}_\MA$.
		
Here, it is essential to characterise the set in such a way that the definition can be 
transferred to the set of probability measures. This can be achieved as follows.  
We need to find two measurable functions 
\[
\Phi:\mathbb{X}\to\mathbb{R}, \ \text{ and } \ 
\Psi:\mathbb{X}\to \mathbb{R}\cup\{\infty\},
\]  
such that for any bounded closed interval \(I\subset \mathbb{R}\), the set 
\(\Psi^{-1}(I)\)\footnote{For a measurable function \(f:X\to\mathbb{R}\), the notation 
	\(f^{-1}(A)\) denotes the preimage \(\{x\in X : f(x)\in A\}\) for all Borel sets \(A\).} 
is closed in \(\mathbb{X}\).
	Let us now define for any   $R>0$
	a subset
	$\mathcal{X}_{\MA}(R)$ of $\mathcal{M}_\MA^m(\mathbb{X})$ by
	\DEQSZ\label{characteriseK}
	\mathcal{X}_{\MA}(R):=\Big\{ \xi\in \mathcal{M}_\MA^m(\mathbb{X}):\EE\Phi(\xi)\le R^m \mbox{ and } \PP\lk(\Psi(\xi)<\infty\rk)=1\Big\}.
	\EEQSZ

To get well-posedness of the system \eqref{spdes} for all $\xi\in 	\mathcal{X}_{\MA}(R)$, eventually the well-posedness to \eqref{spdes1} after a fixed-point arguments, we impose some conditions on $A$, $\Sigma_1$ and $\Sigma_2$ (see (H.1), and (H.2)), then on $\bar F_1$ and $F_2$ (see (H.3) below).

	\begin{hypo}\label{hyp}
	
		In order to prove our main theorem, we need several assumptions: 
		
		\medskip 
		
		\item[(H.1)]
	Let 	$A:D(A)\subset V \to V^\ast$ be a bounded, (possibly nonlinear) measurable 
		single-valued operator defined on the Gelfand triplet 
		\[
		V \hookrightarrow H \equiv H^\ast \hookrightarrow V^\ast.
		\]
	For a given $m>1$, we	assume that there exist  constants 
		\[
		 \beta\in [0,\infty), \qquad \theta\in (0,\infty), 
		\qquad C_0\in \mathbb{R},
		\]
a function $f \in L^1(0,T;\mathbb{R}^+_0)$,  and a mapping $\rho:V \to [0,\infty)$ being  measurable, hemicontinuous, and locally bounded in $V$, such that 	
		\begin{enumerate}
			\item\label{HC} \emph{Hemicontinuity}. The map $\R \ni \lambda \mapsto \langle A(w_1+\lambda w_2), w_3\rangle \in\R$ is continuous for any $w_1,w_2,w_3\in V$ and for a.e. $t\in[0,T]$;
			\item\label{WM} \emph{Local monotonicity}. For all $w_1,w_2\in V$,  we have
			\begin{align*}
			&	2\la A(w_1)-A(w_2), w_1-w_2\ra + {|\Sigma_1(w_1)-\Sigma_1(w_2)|_{\mathcal{L}_2(U,H)}^2} \\
				& \qquad\quad \leq \big(f (t)+{\rho(w_2)}\big) |w_1-w_2|_H^2,
			\end{align*}where $f$ and $\rho$ are defined above;
		\item\label{C} \emph{Coercivity}. 
		\begin{align*}
			2\la A (w),w\ra +|\Sigma_1(w)|_{\mathcal{L}_2(U,H)}^2\leq C_0|w|_H^2-\theta |w|_V^m+f (t);
		\end{align*}
	\item\label{G} \emph{Boundedness}. 
	\begin{align*}
		|A(w)|_{V^\ast}^{\frac{m}{m-1}} \leq \big(f (t)+C_0|w|_V^m \big)\big(1+|w|_H^\beta\big);
	\end{align*}
		\end{enumerate}
		
\item[(H.2)] Let $\Sigma_1,\Sigma_2:V \to \mathcal{L}_2(U, H)$ be a mapping, such that there exists a positive constant $C$ satisfying
\begin{align*}
	|\Sigma_1(w)|_{\mathcal{L}_2(U,H)}^2& \leq C(1+|w|_H^2),\\
		|\Sigma_2(\xi)|_{\mathcal{L}_2(U,H)}^2& \leq C(1+|\xi|_V^{2r}),
\end{align*}for $w,\xi\in V$ and $r\in(0,1)$. Moreover, there exists a positive constant $C$ and $\kappa\in (0,1)$ such that 
\begin{align*}
	|\Sigma_2(\xi_1)-\Sigma_2(\xi_2)|_{\mathcal{L}_2(U,H)}^2\leq C |\xi_1-\xi_2|_V^{2\kappa};
\end{align*}
\item[(H.3)] 
  To make the assumptions more general, we have divided the assumptions on $F_1$ into two cases, depending on the functions $g_1$ and $g_2$.  
Note that the assumption on $F_2$ remains the same in both cases, which is given as follows:

	For $\xi_1,\xi_2\in  \mathcal{X}_\MA(R) $ there exists a constant $C>0$ and $\delta\in (0,1)$ such that
\begin{align}\label{F_2}
	| F_2(\xi_1)-F_2(\xi_2)|_{V^\ast}^m
\le C  | \xi_1-\xi_2|_{V}^{m\delta}.\end{align}

Let us move to the assumptions of $F_1$.
\begin{enumerate}
	
\item[(H.3)$_1$]
In the first case, we assume that there exist
two continuous and measurable positive functions  $g_1$ and $g_2$, two constants   $R_{g_1}\in\RR^+$ and $R_{g_2}\in\RR^+$, and {a number $q \ge\ \max\!\big\{\tfrac{1}{1-\zeta_1},\,\tfrac{1}{1-\zeta_2}\big\}$, for $\zeta_1,\zeta_2\in (0,1)$,} with
\begin{align}\label{Condition}\nonumber
	\mathcal{X}_\MA(R) \subset \bigg\{\xi \in \mathcal{M}_\MA^m(\X): &\ \EE \bigg[\int_0^Tg_1^q(\xi(s))\,ds \bigg]\leq R_{g_1},\\& \text{ and } \EE \bigg[\int_0^Tg_2^q(\xi(s))\,ds \bigg]\leq R_{g_2} \bigg\}.
\end{align}
such that for  $w,w_1, w_2\in V$ and $\xi,\xi_1,\xi_2\in \mathcal{X}_\MA(R)$, we assume the following:
\begin{enumerate}
	\item there exist constants $C>0$, $\zeta_1\in(0,1)$ and a continuous function $g_1$ such that 
	\begin{align*}
		|\bar F_1(\xi,w_1)-\bar  F_1(\xi,w_2)|_{V^\ast}^m\leq g_1(\xi)|w_1-w_2|_V^{m\zeta_1};
	\end{align*}
\item there exists a constant $C>0$, $\zeta_2\in(0,1)$ and a measurable function $g_2$ such that 
	\begin{align*}
	|\bar  F_1(\xi_1,w)-\bar  F_1(\xi_2,w)|_{V^\ast}^m\leq g_2(w)|\xi_1-\xi_2|_V^{m\zeta_2};
\end{align*}
\end{enumerate}
\item[(H.3)$_2$] In the second case, we assume that $g_1$ and $g_2$ are bounded, that is, there exists a positive constant $C$ such that 
\begin{align*}
	g_1(\xi)\leq C, \ \text{ and } \ g_2(\xi)\leq C, \text{ for all } \xi\in\Xcal_\MA(R). 
\end{align*} For $w,w_1, w_2\in V$ and $\xi,\xi_1,\xi_2\in \mathcal{X}_\MA(R)$, we assume the following:
\begin{enumerate}
	\item there exists a constant $C>0$,  such that 
	\begin{align*}
		|\bar F_1(\xi,w_1)-\bar  F_1(\xi,w_2)|_{V^\ast}^m\leq C|w_1-w_2|_V^{m};
	\end{align*}
	\item there exists a constant $C>0$, such that 
	\begin{align*}
		|\bar  F_1(\xi_1,w)-\bar  F_1(\xi_2,w)|_{V^\ast}^m\leq C|\xi_1-\xi_2|_V^{m}.
	\end{align*}
\end{enumerate}
\end{enumerate}
\end{hypo}

\begin{remark}\label{note cutoff}
	Note that, in order to control the nonlinear term $\bar F_1(\cdot,\cdot )$, one can introduce an appropriate cut-off function as in the works \cite{EHBJMPAR,EHJMT}. We are not introducing a cut-off function in our settings because doing so involves several technical details and is very specific to the application. 
\end{remark}

With the above assumptions on the operator \(A\) and the mappings \(F\) and \(\Sigma\), we are ready to formulate our main result, namely the stochastic Schauder--Tychonoff theorem. This result extends the deterministic Schauder--Tychonoff fixed-point theorem; see \cite[\S\S 6--7]{granas}.

	\begin{theorem}\label{ther_main}
		Fix $m>1$. Let $U$ be a separable Hilbert space, $Q: U\to U$ such that $Q$ is linear, symmetric, nonnegative definite, and of trace class, let $H$ be another Hilbert space, and let us assume that we have a compact and dense embedding $\mathbb{X}_1\hookrightarrow\X$. 
		Suppose that for any filtered probability space $\Afrak=(\Omega,\Fcal,\mathbb{F},\P)$
		and for any $Q$-Wiener process $W$ with values in $U$ that is modelled on $\Afrak$ the following holds.
		
		Suppose that there exist constants $R_1,\ldots,R_K>0$, $K\in\N$, continuous functions $\Phi_i:\X\to [0,\infty)$, $1\le i\le K$, measurable functions $\Psi_i:\X\to[0,\infty]$, $1\le i\le K$ with closed sublevel sets $\Psi_i^{-1}([0,\alpha])$, $\alpha\ge 0$, $1\le i\le K$, and
		a nonempty, sequentially weak$^\ast$-closed, measurable and bounded subset 
		$\Xcal_{\Afrak}(R)$ ($:=\Xcal_{\Afrak}({R_1,\ldots,R_K})$ along with \eqref{Condition}) of $\Mcal_{\MA}^m(\X)$ such that:
		\begin{enumerate}[(a)]
			\item $\mathbb{E}[\Phi_i(\xi)]\le R_i$, for every $\xi\in \Xcal_{\Afrak}(R)$ and every $1\le i\le K$;
			\item $\mathbb{P}(\{\Psi_i(\xi)<\infty\})=1$ for every $\xi\in\Xcal_{\Afrak}(R)$ and for every $1\le i\le K$.
		\end{enumerate}
	
		Define the operator 
		\begin{align*}
			\Vcal_{\Afrak,W}: \Mcal_{\MA}^m(\X) \to \Mcal_{\MA}^m(\X), \ \text{ such that } \ \xi \mapsto \Vcal_{\Afrak,W}(\xi):=w(\xi),
		\end{align*}where $w$ is a solution to the system \eqref{spdes}.

	Now, we assume that the operator $\Vcal_{\Afrak,W}$ restricted to $\Xcal_{\Afrak}(R)$ satisfies the following	properties:
		\begin{enumerate}[(i)]
			\item\label{i} the operator $\Vcal_{\Afrak,W}$ is well-defined on $\Xcal_{\Afrak}(R)$ for all choices of $R_i>0$, $1\le i\le K$;
			\item\label{ii} there exist constants $R^0_i>0$, $1\le i\le K$ such that 
			\[\Vcal_{\Afrak,W}(\Xcal_{\Afrak}(R))\subset \Xcal_{\Afrak}(R),\]
			for all $R_i\ge R^0_i$ and all $1\le i\le K$;
			\item\label{iii} for all choices of $R_i>0$, $1\le i\le K$, the restriction $\Vcal_{\Afrak,W}\big\vert_{\Xcal_{\Afrak}(R)}$ is
			uniformly continuous on bounded subsets with respect to the strong topology of $\Mcal_{\Afrak}^{m}(\X)$;
			\item\label{iv} there exist constants $R>0$, $m_0>1$ such that
			\[
			\Eb\left[\|\Vcal_{\Afrak,W}(\xi)\|_{\mathbb{X}_1}^{m_0}\right]\le R,\quad\text{for every}\quad \xi\in\Xcal_{\Afrak}(R),
			\]
			for all $R_i>0$ and all $1\le i\le K$, where $\X_1 $ is embedded compactly in $\X$; 
			\item\label{v} for all $R_i>0$ and all $1\le i\le K$, $\Vcal_{\Afrak,W}(\Xcal_{\Afrak}(R))\subset \D([0,T];H),$ $\P$-a.s.\footnote{Here, $\mathbb{D}([0,T];H)$ denotes the
				Skorokhod space of c\`adl\`ag paths in $H$ endowed with the Skorokhod $J_1$-topology, see \cite[Appendix A2]{Kallenberg}.}
		\end{enumerate}
		Then, there exists a filtered probability space $\Afrak^\ast=(\Omega^\ast,\Fcal^\ast,\mathbb{F}^\ast,\P^\ast)$
		(that satisfies the usual
		conditions)
		together with a $Q$-Wiener process $W^\ast$ modelled on $\Afrak^\ast$
		and an element $w^\ast\in\Mcal_{\Afrak^\ast}^{m}(\X)$ such that
		for all $t\in[0,T]$, $\P^\ast$-a.s.
		\[
		\Vcal_{\Afrak^\ast,W^\ast}(w^\ast)(t)=w^\ast(t)
		\]
		for any initial datum $w_0\in L^m(\Omega,\Fcal_0,\P;H)$,
		where $w^\ast_0\in L^m(\Omega^\ast,\Fcal_0^\ast,\P^\ast;H)$ satisfies $\operatorname{Law}(w^\ast_0)=\operatorname{Law}(w_0)$.
	\end{theorem}

	The proof of Theorem \ref{ther_main} is provided in the subsequent sections. We note that
	by construction, we get that $w^\ast$ solves
	\begin{align*} 
		dw^\ast(t) =&\DeltaA (w^\ast(t))\, dt +F(w^\ast(t))\,dt +\Sigma(w^\ast(t))\,dW^\ast(t),\quad w^\ast(0)=w_0,
	\end{align*} 
	on $\Afrak^\ast$. Therefore, $w^\ast$ is a probabilistic weak solution.

\section{Applications}\label{sec:app}
The main theorem (Theorem~\ref{ther_main}) stated in the previous section is applicable to a large class of stochastic partial differential equations. To illustrate its applicability, we now present several representative examples.

\subsection{The stochastic heat equation with non-Lipschitz nonlinearity}\label{Example1}
To make the use of Theorem~\ref{ther_main} transparent, we first introduce the following model:
\begin{equation}\label{firstex}
	\left\{
	\begin{aligned}
		d u(t)&= \big(\Delta u(t)+u^{[\frac{1}{2}]}(t)\big)\, d t+\Sigma (u(t))\, dW(t),\ t> 0,\\
		u(0)&=u_0.
	\end{aligned}
	\right.
\end{equation}
We point out that, in the example, we assume that $F(u)=F_2(u)=u^{[\frac{1}{2}]}$, that is, $F_1(u)\equiv 0$ and $\Sigma=\Sigma_1$, that is $\Sigma_2\equiv 0$, in the settings of Theorem \ref{ther_main}, where \(u^{[\gamma]}:=|u|^{\gamma-1}u\) with \(\gamma\in(0,\infty)\), and \(\Ocal=(0,1)\).
Our aim is to establish the existence of a weak solution by applying Theorem~\ref{ther_main}. To make the argument transparent, we first specify the precise functional setting and verify that the equation at hand satisfies the structural assumptions of Theorem~\ref{ther_main}. 
 Let us fix $V=H_2^1(\Ocal)$ and $H=L^2(\Ocal)$ and consider the Gelfand triple $V \subset H \subset V^\ast $, where $V^\ast=H_2^{-1}(\Ocal)$. Set $\X:=L^2(0,T;H)$. For $t\in[0,T]$, the operator $\Sigma(t,\cdot)$ is defined from $H_2^{1}(\CO)$ to $\mathcal{L}_2(L^2(\CO),L^2(\CO))$.  Now, we can state our claim:

 \begin{corollary}\label{Existenceheat}
	Let $Q$ be a  linear, nonnegative definite, symmetric trace class covariance
operator $Q:L^2(\CO)\to L^2(\CO) $ and  $H_2^1(\Ocal) \subset L^2(\Ocal) \subset H_2^{-1}(\Ocal) $ be the Gelfand triple  defined above with the initial condition $u_0\in L^2(\Omega; L^2(\Ocal))$.

Then, there exists a probabilistic weak solution $u$, to be more precise, there exists a triplet
$$
(\mathfrak{A},  W,u),
$$
where 
\begin{itemize}
	\item $\Afrak=(\Omega,\Fcal,\mathbb{F},\P)$ is a filtered probability space
	with filtration $\mathbb{F}=\{\Fcal_{t}\}_{t\in [0,T]}$ satisfying the usual conditions;
	\item $W$ is a $L^2$-valued Wiener process over $\mathfrak{A}$ with covariance $Q$;
	\item and $u$ is a strong solution of \eqref{firstex} over $\mathfrak{A}$ such that $u$ is c\`adl\`ag in $L^2(\Ocal)$ progressively measurable in $H_2^{1}(\CO)$ and
	$$
 \EE\bigg[\sup_{t\in[0,T]}|u(t)|_{L^2}^2\bigg] +\EE \Big[\int_0^T|u(t)|_{H_2^1}^2d t\Big]<\infty.
	$$
\end{itemize}
 \end{corollary}
 
\begin{proof} 
 We prove Corollary \ref{Existenceheat} by applying Theorem~\ref{ther_main}. To this end, we first specify the precise setting and the underlying function spaces. Then we verify that the coefficients and data of the equation satisfy all assumptions of Theorem~\ref{ther_main}. Once these conditions are checked, the corollary follows directly from the theorem. For notation convenience, we will denote $H_2^1(\CO)$, $L^2(\CO)$ and $H_2^{-1}(\CO)$ by $V$, $H$ and $V^\ast$, respectively.
 
 Let us  define 
\begin{align*}
	\Mcal_{\Afrak}^{2}(\X):=\Big\{ \xi&:\Omega\times[0,T]\to H\;\colon\;\\
	&\qquad\text{\ensuremath{\xi} is \ensuremath{\mathbb{F}}-progressively measurable}\;\text{and}\;\EE\big[\|\xi\|_{\X}^{2}\big]<\infty\Big\},
\end{align*}equipped with  the norm 
\begin{align*}
	\|\xi\|_{\Mcal_{\Afrak}^{2}(\X)}^2:=\EE\big[\|\xi\|_{\X}^2\big], \quad \xi \in \Mcal_{\Afrak}^{2}(\X).
\end{align*}
Now, we define 
\begin{align*}
	\mathcal{X}_\MA(R) :=\bigg\{\xi \in \Mcal_{\Afrak}^{2}(\X): \EE\bigg[\int_0^T \|\xi(t)\|_H^2d t\bigg]<R \bigg\}, 
\end{align*}where $R$ is a positive real number. Clearly, the set $ \mathcal{X}_\MA(R) $ is a bounded convex set in $\Mcal_{\Afrak}^{2}(\X)$.

For fixed $\MA$ and $W$, we define the operator $$\Vcal:=\Vcal_{\MA,W}:\Mcal_{\Afrak}^2(\X) \times L^2(\Omega,\mathcal{F}_0,\P;H)\break\to \Mcal_{\Afrak}^2(\X), \text{ for }\xi\in \Mcal_{\Afrak}^2(\X)
$$
  by
\begin{align*}
	\Vcal(\xi):=\Vcal(\xi,u_0):=u,
\end{align*}where $u$ is the solution to the following It\^o stochastic partial differential equation
\begin{equation}\label{TS1}
	\left\{
	\begin{aligned}
		d u(t)&= \big(\Delta u(t)+\xi^{[\frac{1}{2}]}(t)\big)\, d t+\Sigma (u(t))\, dW(t), \ t>0, \\
		u(0)&=u_0.
	\end{aligned}
	\right.
\end{equation}

For the well-posedness of the system \eqref{firstex}, we apply Theorem \ref{ther_main}. For this, we have to show the Conditions (\ref{i})--(\ref{v}). Let us start by verifying Condition (\ref{i}).

\begin{step}
	\item 
\textbf{Verification of Condition (\ref{i})}.
In order to verify this, we need to show that for any \(R>0\), initial data \(u_0\in L^2(\Omega;H)\), and \(\xi\in \Mcal_{\Afrak}^2(\X)\cap \mathcal{X}_\MA(R)\), there exists a unique solution to the system \eqref{TS1}. For this part, we apply  \cite[Theorem 5.1.3]{weiroeckner}.
To be more precise, to show that \eqref{TS1} is well-posed,  for a given $\xi \in \mathcal{X}_\MA(R)$, and  $t\in [0,T]$, we define an operator 
$\CA_\xi(t,\cdot):V \to V^\ast$ by
\begin{align*}
	\CA_\xi (u):=  \Delta u+\xi^{[\frac{1}{2}]}, \quad \text{ for } \ \xi\in \Mcal_{\Afrak}^2(\X),\ u\in V. 
\end{align*}
Let us provide the verification of the assumptions on the operator $\CA_\xi(t,\cdot)$\footnote{Note that, the operator $Au=\Delta u$ satisfies the Hypothesis \ref{hyp} (H.1) and for the well-posedness we are again replacing the operator $A$ by $\mathcal{A}_\xi=A+\xi^{[\frac{1}{2}]}$.} given by \cite[Theorem 5.1.3]{weiroeckner}. 
In this setting, we have $V$, $H$ and $V^\ast$ as $H_2^1(\CO)$, $L^2(\CO)$ and $H_2^{-1}(\CO)$, respectively,
and $f_\xi=|\xi|_{L^1}$. Note that $f_\xi\in L^1(\Omega;L^1(0,T;\RR^+))$ for $\xi\in\mathcal{X}_\MA(R)$.
Now, let us start:

\begin{enumerate}
	
	\item \emph{Hemicontinuity}. We know that, the operator $\CA_\xi(t,\cdot)$ is affine, therefore, for all $u_1,u_2,u_3\in V$, the map 
	\begin{align*}
		\lambda \mapsto \big\la \CA_\xi(t,u_1+\lambda u_2),u_3\big\ra
	\end{align*}
	is continuous on $\R$. Since $\Delta$ is hemicontinuous, the hemicontinuity holds.
	
	\item \emph{Local monotonicity}. Fix $u_1,u_2\in V$ and note that 
	\begin{align*}
		&\big\la \CA_\xi(t,u_1)-\CA_\xi(t,u_2),u_1-u_2\big\ra+|\Sigma(u_1)-\Sigma(u_2)|_{\mathcal{L}_2(H,H)}^2\\& \leq -|u_1-u_2|_{V}^2+C|u_1-u_2|_H^2
		\leq C|u_1-u_2|_H^2, 
	\end{align*}and hence the local monotonicity property holds.
	\item \emph{Coercivity}. For $u\in V$, we have 
	\begin{align*}
		2\big\la\CA_\xi(t,u),u\big\ra+|\Sigma(u)|_{\mathcal{L}_2(H,H)}^2 \leq C|u|_{H}^2-|u|_V^2+C(1+|\xi|_{L^1}).
	\end{align*}
	\item \emph{Growth}.  For $u\in V$, we have 
	\begin{align*}
		|\CA_\xi(t,u)|_{V^\ast}^2 \leq |u|_V^2+C|\xi|_{L^1}.
	\end{align*}
\end{enumerate}
We verified the Hypothesis of \cite[Theorem 5.1.3]{weiroeckner}, which ensures the existence of a unique solution to the system \eqref{TS1}.

\item  \textbf{Verification of Condition (\ref{ii})}.
In this step, we prove the invariance of the set \(\mathcal{X}_\MA(R)\), i.e., we verify Condition~\eqref{ii}. 
The key ingredient is the energy estimate established in the following claim:

\begin{claim}\label{eeestimate}
	Let \(R>0\) and let \(\xi\in \Mcal_{\Afrak}^2(\X)\cap \mathcal{X}_\MA(R)\). 
	Then the (unique) solution \(u\) to \eqref{TS1} satisfies
	\begin{align}\label{TS2}
		\EE\Big[\sup_{t\in[0,T]} | u(t) | _H^2\Big]
		+4\,\EE\Big[\int_0^T  | u(t) | _V^2\,dt\Big]
		\le
		\Big(2\,\EE[ | u_0 | _H^2]+C R^{\frac{1}{2}}+CT\Big)e^{CT}.
	\end{align}
\end{claim}
Assuming Claim~\ref{eeestimate}, we obtain the following invariance criterion:
if \(R>0\) is such that
\[
\Big(2\,\EE[|u_0|_H^2]+C R^{\frac{1}{2}}+CT\Big)e^{CT}\le R,
\]
then, for every \(\xi\in \mathcal{X}_\MA(R)\), the mapping \(\Vcal(\xi)\) belongs to \(\mathcal{X}_\MA(R)\). In particular, the set $\mathcal{X}_\MA(R)$ is invariant under $\Vcal$.

\begin{proof}[Proof of Claim \ref{eeestimate}]

Applying It\^o's formula to the process $ | u | _{H}^2$, we find for all $t\in[0,T], \ \P$-a.s.,
\begin{align*}
	 | u(t) | _H^2& =  | u_0 | _H^2-2\int_0^t  | u(s) | _V^2\,ds +2\int_0^t \big(\xi^{[\frac{1}{2}]}(s),u(s)\big)\,ds+\int_0^t | \Sigma(u(s)) | _{\mathcal{L}_2(H,H)}^2 \,ds 
	\\& \quad + 2\int_0^t \big(\Sigma(u(s))\, dW(s), u(s)\big).
\end{align*} Using the linear growth of $\Sigma$, we arrive at for all $t\in[0,T],\ \P$-a.s.,
\begin{align*}
	 | u(t) | _H^2+2\int_0^t  | u(s) | _V^2\,ds & \leq   | u_0 | _H^2+2\int_0^t \big(\xi^{[\frac{1}{2}]}(s),u(s)\big)\,ds+C \int_0^t\big(1+ | u(s) | _{H}^2\big) \,ds 
	\\& \quad + 2\int_0^t \big(\Sigma(u(s))\, dW(s), u(s)\big).
\end{align*}Taking the supremum over time from $0$ to $T$ and then expectation on both sides of the above inequality, we deduce 
\begin{align*}
	&\EE\bigg[\sup_{t\in[0,T]}	 | u(t) | _H^2\bigg]+2\EE\bigg[\int_0^T  | u(s) | _V^2\,ds\bigg] \\& \leq  \EE\big[ | u_0 | _H^2\big]+2\EE\bigg[\int_0^T  | \xi(s) | _{L^1}^\frac{1}{2} | u(s) | _H\,ds\bigg]+C \EE\bigg[\int_0^T\big(1+ | u(s) | _{H}^2\big) \,ds\bigg] 
	\\& \quad + 2\EE\bigg[\sup_{t\in[0,T]}\bigg|\int_0^t \big(\Sigma(u(s))\, dW(s), u(s)\big)\bigg|\bigg]
	\\& \leq 
	\EE\big[ | u_0 | _H^2\big]+\EE\bigg[\int_0^T  | \xi(s) | _{L^1}\,ds\bigg]+C \EE\bigg[\int_0^T\big(1+ | u(s) | _{H}^2\big) \,ds\bigg] 
	\\& \quad + 2\EE\bigg[\bigg(\int_0^T  | \Sigma(u(s)) | _{\mathcal{L}_2(H,H)}^2 | u(s) | _H^2\,ds\bigg)^\frac{1}{2}\bigg]
	\\& \leq 
	\EE\big[ | u_0 | _H^2\big]+C\bigg\{\EE\bigg[\int_0^T  | \xi(s) | _{L^2}^2\,ds\bigg]\bigg\}^\frac{1}{2}+C \EE\bigg[\int_0^T\big(1+ | u(s) | _{H}^2\big) \,ds\bigg] 
	\\& \quad + 2\EE\bigg[\bigg(\int_0^T  | \Sigma(u(s)) | _{\mathcal{L}_2(H,H)}^2 | u(s) | _H^2\,ds\bigg)^\frac{1}{2}\bigg]
	\\& \leq 
	\EE\big[ | u_0 | _H^2\big]+CR^\frac{1}{2}+C \EE\bigg[\int_0^T\big(1+ | u(s) | _{H}^2\big) \,ds\bigg] 
	\\& \quad + \frac{1}{2}\EE\bigg[\sup_{t\in[0,T]} | u(t) | _H^2\bigg]+C\EE \bigg[\int_0^T  | \Sigma(u(s)) | _{\mathcal{L}_2(H,H)}^2\bigg]
	\\& \leq 
	\EE\big[ | u_0 | _H^2\big]+CR^\frac{1}{2}+C \EE\bigg[\int_0^T\big(1+ | u(s) | _{H}^2\big) \,ds\bigg] 
	+\frac{1}{2} \EE\bigg[\sup_{t\in[0,T]} | u(t) | _H^2\bigg].
\end{align*}An application of Gronwall's inequality in the above inequality yields
\begin{align*}
	&	\EE\bigg[\sup_{t\in[0,T]}	 | u(t) | _H^2\bigg]+4\EE\bigg[\int_0^T  | u(s) | _V^2\,ds\bigg] 
	\leq  
	\Big(2\EE\big[ | u_0 | _H^2\big]+CR^\frac{1}{2}+CT \Big)e^{CT} . 
\end{align*}
\end{proof}

\item 
\textbf{Verification of Condition (\ref{iii})}. Our aim is to show the continuity of the operator $\Vcal$ in the norm induced by $\Mcal_{\Afrak}^2(\X)$. Let $\xi_1,\xi_2\in \Mcal_{\Afrak}^2(\X)\cap \mathcal{X}_\MA(R)$ be arbitrary, then there exists a positive constant $C$ such that 
\begin{align}\label{conti}
	\EE\big[\|\Vcal(\xi_1)-\Vcal(\xi_2)\|_\X ^2\big] \leq C\Big\{\EE\big[\|\xi_1-\xi_2\|_\X^2\big]\Big\}^\frac{1}{2}.
\end{align}

In order to prove the inequality \eqref{conti}, we consider the difference $u_1-u_2$, which solves the following system:
\begin{equation*}
	\left\{
	\begin{aligned}
		d ( u_1-u_2)(t)&= \big(\Delta (u_1-u_2)+\xi_1^{[\frac{1}{2}]}-\xi_2^{[\frac{1}{2}]}\big)\, d t+\big(\Sigma (u_1)-\Sigma(u_2)\big)\, dW, \\
		(u_1-u_2)(0)&=0.
	\end{aligned}
	\right.
\end{equation*}Applying It\^o's formula to the process $ | u_1-u_2 | _H^2$, we obtain  for all $t\in[0,T], \ \P$-a.s.,
\begin{align*}
	 | u_1(t)-u_2(t) | _H^2&=   -2\int_0^t  | u_1(s)-u_2(s) | _V^2\,ds+2 \int_0^t \big(\xi_1^{[\frac{1}{2}]}(s)-\xi_2^{[\frac{1}{2}]}(s), u_1(s)-u_2(s)\big)
	\\&\quad +\int_0^t  | \Sigma (u_1(s))-\Sigma(u_2(s)) | _{\mathcal{L}_2(H,H)}^2\,ds
	\\&
	\quad + \int_0^t \Big(\big(\Sigma (u_1(s))-\Sigma(u_2(s))\big)\, dW(s), u_1(s)-u_2(s)\Big).
\end{align*} 
A standard computation yields, for every $t \in [0,T]$, $\mathbb{P}$-a.s.,
\begin{align*}
	& | u_1(t)-u_2(t) | _H^2  +2\int_0^t  | u_1(s)-u_2(s) | _V^2\,ds
	\\&
	= \underbrace{2 \int_0^t \big(\xi_1^{[\frac{1}{2}]}(s)-\xi_2^{[\frac{1}{2}]}(s), u_1(s)-u_2(s)\big)\,ds}_{I_1}
	+C \int_0^t  | u_1(s)-u_2(s) | _{H}^2\,ds
	\\&
	\quad +2 \underbrace{\int_0^t \Big(\big(\Sigma (u_1(s))-\Sigma(u_2(s))\big)\, dW(s), u_1(s)-u_2(s)\Big) }_{I_2}.
\end{align*}Let us consider the term $I_1$ and estimate it as 
\begin{align*}
	|I_1| &\leq 2 \int_0^t  | \xi_1^{[\frac{1}{2}]}(s)-\xi_2^{[\frac{1}{2}]}(s) | _{H} |  u_1(s)-u_2(s) | _H\,ds
	\\& \leq 2 \int_0^t  | \xi_1(s)-\xi_2(s) | _{L^1}^{\frac{1}{2}} |  u_1(s)-u_2(s) | _H\,ds
	\\&\leq 
	\int_0^t  | \xi_1(s)-\xi_2(s) | _{L^1}\,ds + \int_0^t |  u_1(s)-u_2(s) | _H^2\,ds
	\\&\leq 
	C \int_0^t  | \xi_1(s)-\xi_2(s) | _{H}\,ds + \int_0^t |  u_1(s)-u_2(s) | _H^2\,ds
	\\&\leq 
	C t^{\frac{1}{2}}\bigg(\int_0^t  | \xi_1(s)-\xi_2(s) | _{H}^2\,ds\bigg)^{\frac{1}{2}}+ \int_0^t |  u_1(s)-u_2(s) | _H^2\,ds,
\end{align*}where we have used  H\"older's continuity of the square-root, more precisely, $|u^{[a]}-v^{[a]}|\leq |u-v|^a$, for $0<a\leq 1$. 

For the final term $I_2$, we apply Burkholder-Davis-Gundy's inequality together with the Lipschitz continuity of $\Sigma$, and estimate it as follows
\begin{align*}
	\EE\bigg[\sup_{t\in[0,T]}|I_2(t)|\bigg] \leq \frac{1}{2}\EE\bigg[\sup_{t\in[0,T]} | u_1(t)-u_2(t) | _H^2\bigg]+ C \EE\bigg[\int_0^T | u_1(t)-u_2(t) | _H^2\,ds \bigg].
\end{align*}Combining all the above estimates, we obtain
\begin{align*}
	&	\EE\bigg[\sup_{t\in[0,T]} | u_1(t)-u_2(t) | _H^2\bigg]+4\EE\bigg[\int_0^T  | u_1(t)-u_2(t) | _V^2d t\bigg] \\&\leq 
	C T^{\frac{1}{2}}\bigg\{\EE\bigg[\int_0^T | \xi_1(t)-\xi_2(t) | _{L^2}^2\,dt\bigg]\bigg\}^{\frac{1}{2}}+C \EE\bigg[\int_0^T |  u_1(t)-u_2(t) | _H^2\,dt\bigg].
\end{align*}An application of Gronwall's inequality yields
\begin{align*}
	\EE\bigg[\sup_{t\in[0,T]}|u_1(t)-u_2(t)|_H^2\bigg] \leq C \Big\{\EE\big[\|\xi_1-\xi_2\|_{\X}^2\big]\Big\}^{\frac{1}{2}}, 
\end{align*} and hence continuity is established.

\item \textbf{Verification of Condition (\ref{iv})}. In this step we prove that the operator $\mathcal{V}$ maps the set $\mathcal{X}_{\MA}(R)$ to a pre-compact set.  For our purposes, we will use the compactness result by Simon (see \cite{Simon1986}). For our requirement, we need the following:
\begin{enumerate}
	\item[(CC1)]\label{CCi} \emph{Compact containment condition:} From the energy estimate \eqref{TS2}, there exists a positive constant $C$ depending on initial data $u_0$ and $R$ such that 
	\begin{align*}
		\EE\bigg[\int_0^T|u(t)|_V^2\,dt\bigg]\leq C, 
	\end{align*}that means the compact containment condition holds;

\item[(CC2)]\label{CCii} From compactness result by Simon (see \cite{Simon1986}), we know that 
\begin{align*}
\X_1:=	L^2(0,T;V)\cap \mathbb{W}^\alpha_2(0,T;V^\ast) \hookrightarrow L^2(0,T;H).
\end{align*}In view of the condition (CC1), it only remains to verify that the solution belongs to $\mathbb{W}^\alpha_2(0,T;V^\ast)$, that is, the time regularity of the solution.
\end{enumerate}Assuming the conditions (CC1) and (CC2), we obtain that operator $\mathcal{V}$ maps the set $\mathcal{X}_{\MA}(R)$ to a pre-compact set. Let us provide a verification of time regularity of the solution.

\vspace{1mm}
\noindent
 \textbf{Proof of time regularity of the solution.} {Our aim is to show that the solution $u(\cdot,\omega)\in \mathbb{W}^{\alpha}_2(0,T;V^{\ast})$\footnote{For any given separable Banach space $E$, $p\in (1,\infty)$ and $\alpha\in (0,1)$, the symbol $\mathbb{W}^\alpha_p(0,T;E)$ represent a Sobolev space of all $u\in L^p(0,T;E)$ such that 
 	\begin{align*}
 		\int_0^T\int_0^T\frac{|u(t)-u(s)|^p_{E}}{|t-s|^{1+\alpha p}}\,ds\,dt<\infty,
 	\end{align*}
 equipped with the norm 
 \begin{align*}
 	\|u\|_{\mathbb{W}^\alpha_p(0,T;E)}^p:=\int_0^T |u(s)|_E^p\,ds+	\int_0^T\int_0^T\frac{|u(t)-u(s)|^p_{E}}{|t-s|^{1+\alpha p}}\,ds\,dt.
 \end{align*}},  $\P$-a.s. In other words, we need to show the following 
\begin{align*}
	\|u\|_{\mathbb{W}_2^\alpha(0,T;V^{\ast})}^2 = \int_0^T|u(s)|_{V^\ast}^2\, d s+\int_0^T\int_0^T \frac{|u(t)-u(s)|_{V^\ast}^2}{|t-s|^{1+2\alpha}}\, ds \, dt <\infty.
\end{align*}The first term in the above expression is already finite. Therefore, we need to focus on the second term appearing in the right hand side of above expression. From \eqref{TS1}, we obtain for $s,t\in [0,T]$\footnote{We define  $t\vee s:=\max \{t,s\}$ and $t\wedge s:=\min \{t,s\}$.}, 
\begin{align}\label{380}\nonumber
	\EE \big[|u(t)-u(s)|_{V^\ast}^2\big]  &\leq C \EE\bigg[\int_{t\wedge s}^{t\vee s} |\Delta u(r)|_{V^\ast} \, dr\bigg]^2 +C \EE\bigg[\int_{t\wedge s}^{t\vee s} |\xi^{[\frac{1}{2}]}(r)|_{V^\ast}\, dr \bigg]^2
	\\&\nonumber
	\quad + C\EE\bigg[\bigg|\int_{t\wedge s}^{t\vee s}\Sigma(u(r))\, dW(r)\bigg|_{V^\ast}\bigg]^2
	\\&\nonumber
	\leq C |t-s| \EE\bigg[\int_{t\wedge s}^{t\vee s} | u(r)|_{V}^2\, dr\bigg] +C \EE\bigg[\int_{t\wedge s}^{t\vee s} |\xi(r)|_{L^1}^\frac{1}{2}\, dr \bigg]^2
	\\&\nonumber
	\quad + C\EE\bigg[\bigg|\int_{t\wedge s}^{t\vee s}\Sigma(u(r))\, dW(r)\bigg|_{V^\ast}\bigg]^2
	\\&\nonumber
	\leq C |t-s| \EE\bigg[\int_{t\wedge s}^{t\vee s} | u(r)|_{V}^2\, dr\bigg] +C \EE\bigg[\int_{t\wedge s}^{t\vee s}\big(1+ |\xi(r)|_{H}\big)\, dr \bigg]^2
	\\&\nonumber
	\quad + C\EE\bigg[\int_{t\wedge s}^{t\vee s}|\Sigma(u(r))|_{\mathcal{L}_2(H,V^\ast)}^2\, d r\bigg]
	\\&\nonumber
	\leq C |t-s| \EE\bigg[\int_{0}^{T} | u(r)|_{V}^2\, dr\bigg]+C|t-s|^2 +C|t-s| \EE\bigg[\int_{0}^{T} |\xi(r)|_{H}^{2}\, dr \bigg]
	\\&\nonumber
	\quad + C|t-s|\EE\bigg[\sup_{r\in[0,T]}\big\{1+|u(r)|_{H}^2\big\}\bigg] 
	\\& \leq C(|t-s|+|t-s|^2), 
\end{align} where we have used the embedding $H\subset V^\ast$, H\"older's, Young's and Burkholder-Davis-Gundy's inequalities, the Hypothesis \ref{hyp} (H.2) and the energy estimate \eqref{TS2}. One should note that the constant $C$ is a generic constant, whose value may change from line to line.

Let us compute 
\begin{align}\label{381}\nonumber
	\int_0^T\int_0^T \frac{\EE\big[|u(t)-u(s)|_{V^\ast}^2\big]}{|t-s|^{1+2\alpha}}\, ds\, dt &\leq C \int_0^T \int_0^T\frac{1}{|t-s|^{1+2\alpha}}\big(|t-s|+|t-s|^2\big)\, ds\,  dt\\& <\infty, 
\end{align}provided $\alpha <\frac{1}{2}$. 

From \eqref{380} and \eqref{381}, we conclude that $u(\cdot,\omega)\in \mathbb{W}^{\alpha}_2(0,T;V^{\ast})$, $\PP$-a.s.}

Using the above facts, that is, the solution $u\in \X_1:= L^2(0,T;V)\cap \mathbb{W}^\alpha_2(0,T;V^\ast)$, for some $\alpha\in(0,\frac{1}{2})$, and the compactness result by Simon \cite{Simon1986}, we conclude that the embedding of $\X_1$ in $L^2(0,T;H)$ is compact. In particular, for any $M>0$ the set 
\begin{align*}
	\Big\{u\in L^2(0,T;H): \|u\|_{L^2(0,T;V)}+\|u\|_{\mathbb{W}^\alpha_2(0,T;V^\ast)}\leq M \Big\}
\end{align*} is compact in $L^2(0,T;H)$. This completes the verification of Assumption (\ref{iv}) of Theorem \ref{ther_main}. Note that, the fractional-regularity estimate is only used for the verification of compactness.  

\item \textbf{Verification of Condition (\ref{v}).} 
{We only have to verify the last assumption of Theorem~\ref{ther_main}, that is, the Skorokhod property  $\Vcal_{\Afrak,W}(\Xcal_{\Afrak}(R))\subset \D([0,T];H),$ $\P$-a.s. This c\`adl\`ag property follows from the variational structure of \eqref{spdes} (see \cite{KrylovRozovskii,weiroeckner}).}
\end{step}Thus, we verified all the assumptions of Theorem \ref{ther_main}, and the assertion holds.
\end{proof}

\subsection{The stochastic porous media equation with non-Lipschitz nonlinearity}\label{Example2}
Our aim is to fit the following model in the setting of Theorem \ref{ther_main}:
\begin{equation}\label{E2}
	\left\{
	\begin{aligned}
		du(t)&= \big(\Delta u^{[\gamma]}(t)+u^{[\frac{1}{2}]}(t)\big)\, d t+ \Sigma(u(t))\, dW(t),\ t>0, \\
		u(0)&=u_0,
	\end{aligned}
	\right.
\end{equation}where $u^{[\gamma]}:=|u|^{\gamma-1}u$ with $\gamma\in (0,\infty)$, and $\Ocal=(0,1)$. We assume that $\gamma> 1$.

In this example, we assume that $F(u)=F_2(u)=u^{[\frac{1}{2}]}$, that is, $F_1(u)\equiv 0$ and $\Sigma=\Sigma_1$, that is $\Sigma_2\equiv 0$, in the settings of Theorem \ref{ther_main}.

In view of Theorem \ref{ther_main}, we need a Gelfand triple $V\subset H \subset V^\ast$ (here $V$ is a reflexive Banach space  and $H$ is a separable Hilbert space) such that the first embedding is compact and dense, whereas the second embedding is continuous.  For our purposes, let us fix $V=L^{\gamma+1}(\Ocal)$ and $H=H_2^{-1}(\Ocal)$, then $V^\ast = L^{\frac{\gamma+1}{\gamma}}(\Ocal)$
and  the duality pairing is defined as 
\begin{align*}
	_{V^\ast}\langle u,v\rangle_{V}=\int_{\Ocal}((-\Delta)^{-1}u)(x)v(x)\, d x.
\end{align*} Let us fix $\X: =L^{\gamma+1}(0,T;H)$. One should note that we chose this Gelfand triple because our solution lies in these spaces. 

 \begin{corollary}\label{ExistencePorous}
Let $Q$ be a  linear, nonnegative definite, symmetric trace class covariance
	operator $Q:L^2(\CO)\to L^2(\CO)$ and  $L^{\gamma+1}(\CO) \subset H_2^{-1}(\CO) \subset L^\frac{\gamma+1}{\gamma}(\CO) $ be the Gelfand triple with the initial condition $L^{\max\{\gamma+1,2p\}}(\Omega;L^{\gamma+1}(\mathcal{O}))$ for $\gamma>1$ and $p\in [1,\gamma+1)$.
	
	Then, there exists a probabilistic weak solution $u$. To be more precise, there exists a triplet
	$$
	(\mathfrak{A},  W,u),
	$$
	where 
	\begin{itemize}
		\item $\Afrak=(\Omega,\Fcal,\mathbb{F},\P)$ is a filtered probability space
		with filtration $\mathbb{F}=\{\Fcal_{t}\}_{t\in [0,T]}$ satisfying the usual conditions;
		\item $W$ is a $L^2$-valued Wiener process over $\mathfrak{A}$ with covariance $Q$;
		\item and $u$ is a strong solution of \eqref{E2} over $\mathfrak{A}$ such that $u$ is c\`adl\`ag in $H_2^{-1}(\CO)$ progressively measurable in $L^{\gamma+1}(\CO)$ and
		$$
		\EE\bigg[\sup_{t\in[0,T]} | u(t) | _{H_2^{-1}}^2\bigg] +\EE\bigg[\int_0^T | u(t) | _{L^{\gamma+1}}^{\gamma+1}\,dt\bigg] <\infty.
		$$
	\end{itemize}
\end{corollary}
\begin{proof} The proof of Corollary \ref{ExistencePorous} relies on Theorem \ref{ther_main}. First, we have to define the required spaces $\mathcal{M}_\MA^{\gamma+1}(\X)$ and $\mathcal{X}_\MA(R)$. Then, we need to verify the assumptions of Theorem \ref{ther_main}. Once the assumptions of Theorem \ref{ther_main} are verified, the corollary is straightforward.

Let us define 
\begin{align*}
\Mcal_{\Afrak}^{\gamma+1}(\X) := \big\{ \xi:&\Omega \times [0,T]\to H: \\& \quad \xi \text{ is $\mathbb{F}$-progressively measurable and } \EE\big[\|\xi\|_{\X}^{\gamma+1}\big]<\infty\big\},
\end{align*}equipped with the norm 
\begin{align*}
	\|\xi\|_{\Mcal_{\Afrak}^{\gamma+1}(\X)}^{\gamma+1}:=\EE\big[\|\xi\|_{\X}^{\gamma+1}\big], \ \quad \xi\in \Mcal_{\Afrak}^{\gamma+1}(\X).
\end{align*}Now, we define 
\begin{align*}
	\Xcal_\Afrak(R):=\bigg\{\xi \in \Mcal_\Afrak^{\gamma+1}(\X):\EE\bigg[\int_0^T  |\xi(t)|_{V}^{\gamma+1}d t \bigg]<R \bigg\},
\end{align*}where $R$ is a positive real number. From the definition, it is clear that the set $\Xcal_\Afrak(R)$ is a bounded convex subset of $\Mcal_\Afrak^{\gamma+1}(\X)$.

For fixed $\Afrak$ and $W$, we define the operator $$\Vcal:=\Vcal_{\Afrak,W}:\Mcal_{\Afrak}^{\gamma+1}(\X)\times L^{\gamma+1}(\Omega,\mathcal{F}_0,\P;L^{\gamma+1}(\CO))\to \Mcal_{\Afrak}^{\gamma+1}(\X), \text{ for } \xi\in \Mcal_{\Afrak}^{\gamma+1}(\X)$$ via
\begin{align*}
	\Vcal(\xi):=\Vcal(\xi,u_0):=u, 
\end{align*}where $u$ is the solution of the following SPDEs
\begin{equation}\label{E21}
	\left\{
	\begin{aligned}
		d u(t)&= \big(\Delta u^{[\gamma]}(t)+\xi^{[\frac{1}{2}]}(t)\big)\, d t+\Sigma(u(t))\, dW(t), \ t\in (0,T),\\
		u(0)&=u_0.
	\end{aligned}
	\right.
\end{equation}

To apply Theorem $\ref{ther_main}$, we must verify Conditions $(\ref{i})$–$(\ref{v})$ stated therein. We now check these assumptions one by one, beginning with Condition $(\ref{i})$.

\begin{step}
	\item
\textbf{Verification of Condition (\ref{i})}. As mentioned in the previous example, we need to show that for any $R>0$, initial data $u_0$ and $\xi\in \Mcal_{\Afrak}^{\gamma+1}(\X)\cap \Xcal_{\Afrak}(R)$, there exists a unique solution to the system \eqref{E21}, where the idea is the same as in \cite[Theorem 5.1.3]{weiroeckner}.
For the well-posedness result, for $t\in[0,T]$, we define an operator 
$\CA_\xi(t,\cdot):V \to V^\ast$ by 
\begin{align*}
	\CA_\xi (u):= A(u)+ F(\xi)= \Delta u^{[\gamma]}+\xi^{[\frac{1}{2}]}, \quad \text{ for } \ \xi\in \Mcal_{\Afrak}^{\gamma+1}(\X),\ u\in V. 
\end{align*} 
 Let us start with the proof of verification of the Hypothesis \ref{hyp} for the operator $\CA_\xi(t,\cdot)$\footnote{Note that, the operator $A(u)=\Delta u^{[\gamma]}$ satisfies the Hypothesis \ref{hyp} (H.1) and for the well-posedness we are replacing the operator $A$ by $\mathcal{A}_\xi=A+\xi^{[\frac{1}{2}]}$ }. 
\begin{enumerate}
	\item \emph{Hemicontinuity}. For any $\lambda\in\R$, we need to prove that, for $u_1,u_2,u_3\in V$ 
	\begin{align*}
		\lambda\mapsto\ _{V^\ast}\langle \CA_\xi(t,u_1+\lambda u_2), u_3\rangle_{V}
	\end{align*}
is continuous on $\R$. One should note that the mapping $\lambda\mapsto (u_1+\lambda u_2)^{[\gamma]}$ is continuous and by the Dominated Convergence Theorem, the mapping 
$$\R_0^+\ni \lambda \mapsto \int_\Ocal (u_1+\lambda u_2)^{[\gamma]}u_3dx$$
 is continuous.

	\item \emph{Local monotonicity}. First note that, for $u_1,u_2\in V$, we have (cf. \cite[pp. 87]{weiroeckner})
	\begin{align*}
		-\int_\Ocal \Big((u_1^{[\gamma]}-u_2^{[\gamma]})(u_1-u_2)\Big)\, dx \leq 0.
	\end{align*}Therefore, 
	\begin{align*}
	&	_{V^\ast}\langle \CA_\xi(t,u_1)-\CA_\xi(t,u_2), u_1-u_2\rangle _{V}+ | \Sigma(u_1)-\Sigma(u_2) | _{\mathcal{L}_2(L^2(\CO),H)}^2 
	\\& \leq  	-\int_\Ocal \Big((u_1^{[\gamma]}-u_2^{[\gamma]})(u_1-u_2)\Big)\, dx + | \Sigma(u_1)-\Sigma(u_2) | _{\mathcal{L}_2(L^2(\CO),H)}^2 \\&
	\leq C | u_1-u_2 | _{H}^2 .
		\end{align*} 

\item \emph{Coercivity}. For $u\in V$, we have 
\begin{align*}
&2	_{V^\ast}\langle \CA_\xi(t,u),u\rangle_{V}+ | \Sigma(u) | _{\mathcal{L}_2(L^2(\CO),H)}^2\\
&\nonumber= -2 | u | _V^{\gamma+1}+2	_{V^\ast}\langle \xi^{[\frac{1}{2}]},u\rangle_V+ | \Sigma(u) | _{\mathcal{L}_2(L^2(\CO),H)}^2
	\\&\nonumber\leq 
	-2 | u | _V^{\gamma+1}+2 | \xi | _{L^1}^\frac{1}{2} | u | _H+C(1+ | u | _{H}^2)
	\\&\nonumber\leq
	-2 | u | _V^{\gamma+1}+C | \xi | _{L^2}+C(1+ | u | _{H}^2)
		\\&\leq
		C | u | _{H}^2-2 | u | _V^{\gamma+1}+C(\gamma)+ C(\gamma)f_\xi(t),
\end{align*}with $f_\xi(t)= | \xi(t) | _{V}^{\gamma+1}$, since $f_\xi\in L^1(\Omega;L^1(0,T;\R^+_0))$ for $\xi\in \mathcal{X}_\MA(R)$, the coercivity condition holds;
\item \emph{Growth}. For $u,v\in V$, 
\begin{align*}
	_{V^\ast}\langle \CA_\xi(t,u),v\rangle_V
	& \leq  | u | _V^\gamma | v | _V +	_{V^\ast}\langle \xi^{[\frac{1}{2}]},v\rangle_V
\leq C\big( | u | _V^\gamma+ | \xi | _{L^1}^\frac{1}{2}\big) | v | _V.
\end{align*}Thus, we have
\begin{align*}
	 | \CA_\xi(t,u) | _{V^\ast}^{\frac{\gamma+1}{\gamma}} &\leq C \big( | u | _V^\gamma+ | \xi | _{L^1}^\frac{1}{2}\big)^\frac{\gamma+1}{\gamma}
	\\&\nonumber\leq C  | u | _V^{\gamma+1}+C  | \xi | _{L^1}^{\frac{\gamma+1}{2\gamma}}
	\\&\leq 
	C | u | _V^{\gamma+1}+C(\gamma)+C(\gamma) | \xi | _V^{\gamma+1}.
\end{align*}
\end{enumerate}
Thus, we verified the Hypothesis of \cite[Theorem 5.1.3]{weiroeckner}, and ensured the existence of a unique solution to the system \eqref{E21}.

\item \textbf{Verification of Condition (\ref{ii})}.
In this step, we establish the invariance of the set $\mathcal{X}_\MA(R)$, i.e., we verify the Assumption (\ref{ii}) and the key ingredient is the energy estimate obtained in the following claim:
\begin{claim}\label{TCl2} Let $R>0$, and let $\xi\in \Mcal_{\Afrak}^{\gamma+1}(\X)\cap \mathcal{X}_\MA(R)$. Then the (unique) solution $u$ to the system \eqref{E21} satisfies 
\begin{align}\label{EEs}\nonumber
&	\EE\bigg[\sup_{t\in[0,T]}	 | u(t) | _{H}^2\bigg] +4\EE\bigg[\int_0^T | u(t) | _V^{\gamma+1}\,dt\bigg] \\& \leq \Big(2\EE\big[	 | u_0 | _{H}^2\big]+C(T)R^\frac{1}{\gamma+1}+C(T)\Big)e^{CT},
\end{align} and for $u_0\in L^{\max\{\gamma+1,2p\}}(\Omega;L^{\gamma+1}(\mathcal{O}))$, where $p\in [1,\gamma+1)$, we have
	\begin{align*}
		&	\EE\bigg[\sup_{t\in[0,T]}	 | u(t) | _{H}^{2p}\bigg] +C\EE\bigg[\int_0^T | u(t) | _V^{\gamma+1}\,dt\bigg]^p \\& \leq \Big(C\EE\big[	 | u_0 | _{H}^{2p}\big]+CR^\frac{p}{\gamma+1}+CT\Big)e^{CT}. 
\end{align*}
\end{claim}
Assuming Claim \ref{TCl2} is true, then it is obvious to choose $R>0$ such that 
\begin{align*}
	\Big(2\EE\Big[	 | u_0 | _{H}^2\Big]+C(T)R^\frac{1}{\gamma+1}+C(T)\Big)e^{CT}\leq R,
\end{align*}for every $\xi\in \mathcal{X}_\MA(R)$, that means, the operator $\Vcal(\xi)$ belongs to $\Xcal_\Afrak(R)$. In particular the set $\mathcal{X}_\MA(R)$ is invariant under $\mathcal{V}$.
\begin{proof}[Proof of Claim \ref{TCl2}]
Applying It\^o's formula to the process $ | u | _{H}^2$, we obtain for all $t\in[0,T]$, $\P$-a.s., 
\begin{align*}
&	 | u(t) | _{H}^2 +2 \int_0^t | u(s) | _V^{\gamma+1}\,ds 
	\\&\leq 	 | u_0 | _{H}^2+2\int_0^t \big(\nabla^{-1}\xi^{[\frac{1}{2}]}(s), \nabla^{-1}u(s)\big)\,ds+  \int_0^t | \Sigma(u(s)) | _{\mathcal{L}_2(L^2(\CO),H)}^2\,ds \\&\quad + 2\int_0^t \big(\Sigma(u(s))\, dW(s),u(s)\big)_H.
\end{align*}Using the embedding $V\hookrightarrow H$ and Young's inequality, we find 
\begin{align*}
	\big|\big(\nabla^{-1}\xi^{[\frac{1}{2}]}, \nabla^{-1}u\big)\big| \leq  | \xi^{[\frac{1}{2}]} | _{L^{\gamma+1}} | u | _H \leq C  | \xi | _{L^\frac{\gamma+1}{2}}^\frac{1}{2} | u | _H \leq \frac{1}{2} | \xi | _{V}+C | u | _H^2.  
\end{align*}Thus, we have for all $t\in[0,T], \ \P$-a.s., 
\begin{align*}
	&	 | u(t) | _{H}^2 +2 \int_0^t | u(s) | _V^{\gamma+1}\,ds 
	\\&\leq 	 | u_0 | _{H}^2+\int_0^t  | \xi(s) | _V\,ds +C\int_0^t\big(1+ | u(s) | _H^2)\,ds+ 2\int_0^t \big(\Sigma(u(s))\, dW(s),u(s)\big)_H.
\end{align*}Taking supremum over $t$ from $0$ to $T$, and then the expectation on both sides of the above inequality, we obtain 
\begin{align*}
	&\EE\bigg[\sup_{t\in[0,T]}	 | u(t) | _{H}^2\bigg] +2 \EE\bigg[\int_0^T | u(s) | _V^{\gamma+1}\,ds\bigg] 
	\\&\leq \EE\Big[	 | u_0 | _{H}^2\Big]+\EE\bigg[\int_0^T  | \xi(s) | _V\,ds\bigg] +C\EE\bigg[\int_0^T\big(1+ | u(s) | _H^2)\,ds\bigg]\\&\quad + 2\EE\bigg[\sup_{s\in[0,t]}\bigg|\int_0^t \big(\Sigma(u(s))\, dW(s),u(s)\big)_H\bigg|\bigg]
		\\&\leq \EE\Big[	 | u_0 | _{H}^2\Big]+T^{\frac{\gamma}{\gamma+1}}\bigg\{\EE\bigg[\int_0^T  | \xi(s) | _V^{\gamma+1}\,ds\bigg]\bigg\}^\frac{1}{\gamma+1} +C\EE\bigg[\int_0^T\big(1+ | u(s) | _H^2)\,ds\bigg]\\&\quad + C\EE\bigg[\bigg(\int_0^T  | \Sigma(u(s)) | _{\mathcal{L}_2(L^2(\CO),H)}^2 | u(s) | _H^2\,ds\bigg)^\frac{1}{2}\bigg]
			\\&\leq \EE\Big[	 | u_0 | _{H}^2\Big]+T^{\frac{\gamma}{\gamma+1}}R^\frac{1}{\gamma+1} +C\EE\bigg[\int_0^T\big(1+ | u(s) | _H^2)\,ds\bigg] + \frac{1}{2}\EE\bigg[\sup_{t\in[0,T]}  | u(t) | _{H}^2\bigg].
\end{align*}An application of Gronwall's inequality in the above inequality yields
\begin{align*}
	\EE\bigg[\sup_{t\in[0,T]}	 | u(t) | _{H}^2\bigg] +4\EE\bigg[\int_0^T | u(t) | _V^{\gamma+1}\,dt\bigg] \leq \Big(2\EE\Big[	 | u_0 | _{H}^2\Big]+C(T)R^\frac{1}{\gamma+1}+C(T)\Big)e^{CT}.
\end{align*} 
\end{proof}

\item
\textbf{Verification of Condition (\ref{iii})}. Our main goal is to prove the continuity of the operator $\Vcal$ in the norm induced by $\Mcal_\Afrak^{\gamma+1}(\X)$. Let $\xi_1,\xi_2\in \Mcal_\Afrak^{\gamma+1}(\X)\cap\Xcal_\Afrak(R)$ be arbitrary, then there exists a positive constant $C$ such that 
\begin{align*}
	\EE\big[\|\Vcal(\xi_1)-\Vcal(\xi_2)\|_\X ^{\gamma+1}\big] \leq C\Big\{\EE\big[\|\xi_1-\xi_2\|_\X^{\gamma+1}\big]\Big\}^\frac{1}{\gamma+1}.
\end{align*}
Let us consider the difference $u_1-u_2$, solving the following system 
\begin{equation*}
	\left\{
	\begin{aligned}
		d ( u_1-u_2)(t)&= \big(\Delta \big(u_1^{[\gamma]}(t)-u_2^{[\gamma]}\big)(t)+\xi_1^{[\frac{1}{2}]}(t)-\xi_2^{[\frac{1}{2}]}(t)\big)\, d t\\&\quad +\big(\Sigma (u_1(t))-\Sigma(u_2(t))\big)\, dW(t), \ t\in (0,T),\\
		(u_1-u_2)(0)&=0.
	\end{aligned}
	\right.
\end{equation*}Applying It\^o's formula to the process $|u_1-u_2|_H^2$ and performing a standard calculation, we find $\P$-a.s.,
\begin{align}\label{CT3}\nonumber
&\sup_{t\in[0,T]} | u_1(t)-u_2(t) | _H^2+\frac{1}{2^{\gamma-1}} \int_0^T | u_1(s)-u_2(s) | _V^{\gamma+1}\,ds 
\\&\nonumber
\leq 2\int_0^T \big(\nabla^{-1}(\xi_1^{[\frac{1}{2}]}(s)-\xi_2^{[\frac{1}{2}]}(s)), \nabla^{-1}(u_1(s)-u_2(s))\big)\,ds \\&\nonumber \quad +\int_0^T | \Sigma(u_1(s))-\Sigma(u_2(s)) | _{\mathcal{L}_2(L^2(\CO),H)}^2\,ds \\&\quad +2\sup_{t\in[0,T]}\bigg|\int_0^t\big((\Sigma(u_1(s))-\Sigma(u_2(s)))\, dW(s), u_1(s)-u_2(s)\big)_H \bigg|. 
\end{align}
We focus on the first term on the right-hand side of \eqref{CT3} and estimate it as follows
\begin{align*}
&	\bigg|2\int_0^T \big(\nabla^{-1}(\xi_1^{[\frac{1}{2}]}(s)-\xi_2^{[\frac{1}{2}]}(s)), \nabla^{-1}(u_1(s)-u_2(s))\big)\,ds \bigg|
\\&\leq C T^{\frac{\gamma}{\gamma+1}} \bigg(\int_0^T | \xi_1(s)-\xi_2(s) | _V^{\gamma+1}\,ds \bigg)^\frac{1}{\gamma+1}+C \int_0^T | u_1(s)-u_2(s) | _{H}^2\,ds.
\end{align*}For the final term appearing in the right-hand side of \eqref{CT3}, we use Burkholder-Davis-Gundy's and Young's inequalities and estimate it as 
\begin{align*}
	&2\EE\bigg[\sup_{t\in[0,T]}\bigg|\int_0^t\big((\Sigma(u_1(s))-\Sigma(u_2(s)))\, dW(s), u_1(s)-u_2(s)\big)_H \bigg|\bigg] \\&\leq 
	\frac{1}{2}\EE\bigg[\sup_{t\in[0,T]} | u_1(t)-u_2(t) | _H^2\bigg]+C \EE\bigg[\int_0^T | \Sigma(u_1(s))-\Sigma(u_2(s)) | _{\mathcal{L}_2(L^2(\CO),H)}^2\,ds  \bigg]. 
\end{align*}Combining the above facts with \eqref{CT3}, we arrive at 
\begin{align*}
&	\EE \bigg[\sup_{t\in[0,T]} | u_1(t)-u_2(t) | _H^2\bigg]+C(\gamma)\EE\bigg[ \int_0^T | u_1(s)-u_2(s) | _V^{\gamma+1}\,ds \bigg]
	\\&\nonumber
	\leq C T^{\frac{\gamma}{\gamma+1}} \bigg\{\EE\bigg[\int_0^T | \xi_1(s)-\xi_2(s) | _V^{\gamma+1}\,ds\bigg]\bigg\}^\frac{1}{\gamma+1}+C \EE\bigg[\int_0^T | u_1(s)-u_2(s) | _{H}^2\,ds\bigg].
\end{align*}An application of Gronwall's inequality in the above inequality yields
\begin{align*}
	&	\EE \bigg[\sup_{t\in[0,T]} | u_1(t)-u_2(t) | _H^2\bigg]+C(\gamma)\EE\bigg[ \int_0^T | u_1(s)-u_2(s) | _V^{\gamma+1}\,ds \bigg]
\\&	\leq C T^{\frac{\gamma}{\gamma+1}} \bigg\{\EE\bigg[\int_0^T | \xi_1(s)-\xi_2(s) | _V^{\gamma+1}\,ds\bigg]\bigg\}^\frac{1}{\gamma+1},
\end{align*}and hence the continuity holds.

\item \textbf{Verification of Condition (\ref{iv})}. In this step, we establish that the operator $\Vcal$ maps the set $\Xcal_{\MA}(R)$ to a pre-compact set. We are going to follow the same idea as in Step (IV) of the Example \ref{Example1}.
From the energy estimate \eqref{EEs}, there exists a positive constant $C$ depending on initial data $u_0$ and $R$ such that 
\begin{align*}
	\EE\bigg[\int_0^T|u(t)|_V^{\gamma+1}\,dt\bigg]\leq C, 
\end{align*}that means the compact containment condition holds.

Now, we prove the solution's time regularity.  

\textbf{Time regularity of the solution}. {Our aim is to show that the solution $u(\cdot,\omega)\in \mathbb{W}^{\alpha}_2(0,T;H)$, $\P$-a.s. From \eqref{E21}, we obtain for $s,t\in[0,T]$
\begin{align}\label{320}\nonumber
\EE\big[	|u(t)-u(s)|_H^2\big]  & \leq 
	 C \EE \bigg[\int_{t\wedge s}^{t\vee s} |\Delta u^{[\gamma]}(r)|_{H}\,  dr \bigg]^2+C\EE\bigg[\int_{t\wedge s}^{t\vee s} |\xi^{[\frac{1}{2}]}(r)|_H\, d r\bigg]^2\\&\quad +C\EE\bigg[\bigg|\int_{t\wedge s}^{t\vee s}\Sigma(u(r))\, dW(r)\bigg|_{H}^2\bigg]. 
\end{align}We consider the first term appearing in the right hand side of \eqref{320}, and estimate it as 
\begin{align}\label{321}\nonumber
C	\EE \bigg[\int_{t\wedge s}^{t\vee s} |\Delta u^{[\gamma]}(r)|_{H}\,  dr \bigg]^2 & \leq C \EE\bigg[ \int_{t\wedge s}^{t\vee s} |u(r)|_V^\gamma \; dr \bigg]^2
\\&\nonumber \leq C |t-s|^{\frac{2}{\gamma+1}}\EE\bigg[\int_{t\wedge s}^{t\vee s} |u(r)|_V^{\gamma+1}\, dr \bigg]^{\frac{2\gamma}{\gamma+1}}
\\& \leq \nonumber C |t-s|^{\frac{2}{\gamma+1}}\bigg\{\EE\bigg[\int_0^T |u(r)|_V^{\gamma+1}\, dr \bigg]\bigg\}^{\frac{2\gamma }{\gamma+1}}
\\& \leq C |t-s|^{\frac{2}{\gamma+1}}. 
\end{align}The above calculation can be justified as follows: 
\begin{align*}
	|\Delta |u|^{\gamma-1}u|_{H} &\leq C| \Delta|u|^{\gamma-1}u|_{V^\ast},\\
	\big|_{V^\ast}\langle \Delta|u|^{\gamma-1}u,v\rangle_{V}\big| &=\bigg|\int_{\mathcal{O}} |u(x)|^{\gamma-1}u(x)v(x)\, dx\bigg|\leq \int_\mathcal{O} |u(x)|^{\gamma}|v(x)|dx \leq C |u|_{V}^\gamma |v|_{V},
\end{align*}for $u,v\in V$, where we have used the fact that  $(V^\ast:=)L^{\frac{\gamma+1}{\gamma}}(\mathcal{O}) \hookrightarrow H_2^{-1}(\mathcal{O})(=:H)$, since $H_2^{1}(\mathcal{O})\hookrightarrow L^{\gamma+1}(\mathcal{O})(=:V)$.

Now, we move the second term appearing in the right hand side of \eqref{320} and estimate it as follows: 
\begin{align}\label{322}\nonumber
	C\EE\bigg[\int_{t\wedge s}^{t\vee s} |\xi^{[\frac{1}{2}]}(r)|_H\, d r\bigg]^2& \leq C \EE\bigg[\int_{t\wedge s}^{t\vee s} |\xi^{[\frac{1}{2}]}(r)|_{L^{\gamma+1}}\,  dr\bigg]^2  
	\leq C \EE\bigg[\int_{t\wedge s}^{t\vee s} |\xi(r)|_{L^{\frac{\gamma+1}{2}}}^\frac{1}{2} \, dr\bigg]^2\\&\nonumber
		\leq C \EE\bigg[\int_{t\wedge s}^{t\vee s} |\xi(r)|_{L^{\gamma+1}}^\frac{1}{2}\,  dr\bigg]^2
	\\& \leq \nonumber
	C |t-s|^{\frac{2\gamma+1}{\gamma+1}}\bigg\{\EE\bigg[\int_{t\wedge s}^{t\vee s} |\xi(r)|_V^{\gamma+1}\,  d r\bigg] \bigg\}^\frac{1}{\gamma+1}
		\\&\nonumber \leq 
	C |t-s|^{\frac{2\gamma+1}{\gamma+1}}\bigg\{\EE\bigg[\int_0^T |\xi(r)|_V^{\gamma+1}\,  d r\bigg] \bigg\}^\frac{1}{\gamma+1} \\&\leq 
	C |t-s|^{\frac{2\gamma+1}{\gamma+1}}. 
\end{align}Finally, we consider the final term appearing in the right hand side of \eqref{320} and estimate it using Burkholder-Davis-Gundy's inequality as 
\begin{align}\label{323}\nonumber
	C\EE\bigg[\bigg|\int_{t\wedge s}^{t\vee s}\Sigma(u(r))\, dW(r)\bigg|_{H}^2\bigg] & \leq C \EE\bigg[\int_{t\wedge s}^{t\vee s} |\Sigma(u(r))|_{\mathcal{L}_2(L^2(\mathcal{O}), H)}^2dr  \bigg] 
	\\& \leq C |t-s| \EE\bigg[\sup_{r\in[0,T]}\big\{1+ |u(r)|_H^2\big\}\bigg].
\end{align}With the help of \eqref{EEs}, \eqref{321}, \eqref{322}, \eqref{323} and \eqref{320}, we obtain 
\begin{align*}
&\int_0^T \int_0^T \frac{\EE\big[|u(t)-u(s)|_H^2\big]}{|t-s|^{1+2\alpha}}	
\\& \nonumber \leq C\int_0^T \int_0^T \frac{1}{|t-s|^{1+2\alpha }} \Big(|t-s|^{\frac{2}{\gamma+1}}+|t-s|^{\frac{2\gamma+1}{\gamma+1}}+|t-s|  \Big)\, ds\, dt 
\\& <\infty,
\end{align*} provided $\alpha <\frac{1}{\gamma+1}$.
Thus, we finally conclude that $u(\cdot,\omega)\in \mathbb{W}^{\alpha}_2(0,T;H)$, $\PP$-a.s., for $\alpha <\frac{1}{\gamma+1}$.}
Using the above fact and the energy estimate \eqref{EEs}, we know that the solution $u\in \X_1:= L^{\gamma+1}(0,T;V)\cap \mathbb{W}^\alpha_2(0,T;H)\cap L^\infty(0,T;H)$, for some $\alpha\in (0,\frac{1}{\gamma+1})$. Combining this with the compactness result \cite[Theorem A.2]{FHEHKK} (cf.~\cite{Simon1986}), we conclude that the embedding $\X_1$ in $L^{\gamma+1}(0,T;H)$ is compact. In particular, for any $M>0$, the set 
\begin{align*}
		\Big\{u\in L^{\gamma+1}(0,T;H): \|u\|_{L^{\gamma+1}(0,T;V)}+\|u\|_{\mathbb{W}^\alpha_2(0,T;H)}+\|u\|_{L^\infty(0,T;H)}\leq M \Big\}
\end{align*} is compact in $L^{\gamma+1}(0,T;H)$. This completes the verification of Assumption (\ref{iv}) of Theorem \ref{ther_main}. {Note that, the fractional-regularity estimate is only used for the verification of compactness.}

\item \textbf{Verification of Condition (\ref{v}).} 
{We only have to verify the last assumption of Theorem~\ref{ther_main}, that is, the Skorokhod property  $\Vcal_{\Afrak,W}(\Xcal_{\Afrak}(R))\subset \D([0,T];H),$ $\P$-a.s. This c\`adl\`ag property follows from the variational structure of \eqref{spdes} (see \cite{KrylovRozovskii,weiroeckner}).}
\end{step}Thus, we verified all the assumptions of Theorem \ref{ther_main}, and the assertion holds.
\end{proof}

\subsection{The stochastic porous media equation driven by nonlinear gradient noise}\label{Example3}

In order to match the framework settings given in Section \ref{schauder}, we fix $A(u)=\Delta u^{[\gamma]}$, $F(u)\equiv 0$, and {$\Sigma(u)=\Sigma_2(u)=\nabla u^{[\frac{1}{2}]} $}.
{Then, we can modify Step (VII) in the proof. Here, we go along the lines of the article in \cite{FKEHMH}}.
Now, the modification will also apply to a non-Lipschitz diffusion coefficient. 
In particular, we can show the existence of a solution to the following model:
\begin{equation}\label{E3}
	\left\{
	\begin{aligned}
		du(t)&=\Delta u^{[\gamma]}(t)\, dt + {\nabla u^{[\frac{1}{2}]}(t)}dW(t),\ \text{ in } \ (0,T)\times \Ocal,\\
		u(0)&=u_0, \ \text{ on } \{0\}\times \Ocal. 
	\end{aligned}
	\right.
\end{equation}where $u^{[\gamma]}:=|u|^{\gamma-1}u, \ \gamma\in (0,\infty)$, and $\Ocal=(0,1)$. We assume that $\gamma>1$. One should note that the difference between Example 2 in Subsection \ref{Example2} and this example lies in the diffusion. Here, we are allowing a noise coefficient that depends on the gradient of the unknown in the system. Therefore, in this section, our focus will be on the diffusion term only; for the remaining terms, we are referring to Subsection \ref{Example2}. For the basic definitions of the function space, we are referring to the previous section (Subsection \ref{Example2}). One should note that, in this example, our aim is to use a fixed-point argument in the diffusion term, and then the system \eqref{E3} is reduced to the system \eqref{E31}, where the noise behaves as additive noise, that is, independent of $u$.

 \begin{corollary}\label{ExistencePorousGrad}
Let $Q$ be a  linear, nonnegative definite, symmetric trace class covariance
	operator $Q:L^2(\CO)\to L^2(\CO)$ and  $L^{\gamma+1}(\CO) \subset H_2^{-1}(\CO) \subset L^{\frac{\gamma+1}{\gamma}}(\CO) $ be the Gelfand triple  defined above with the initial condition $L^{\max\{\gamma+1,2p\}}(\Omega;L^{\gamma+1}(\mathcal{O}))$ for $\gamma>1$ and $p\in [1,\gamma+1)$.
	
	Then, there exists a probabilistic weak solution $u$.
	To be more precise, there exists a triplet
	$$
	(\mathfrak{A},  W,u),
	$$
	where 
	\begin{itemize}
		\item $\Afrak=(\Omega,\Fcal,\mathbb{F},\P)$ is a filtered probability space
		with filtration $\mathbb{F}=\{\Fcal_{t}\}_{t\in [0,T]}$ satisfying the usual conditions;
		\item $W$ is a $L^2$-valued Wiener process over $\mathfrak{A}$ with covariance $Q$;
		\item and $u$ is a strong solution of \eqref{E3} over $\mathfrak{A}$ such that $u$ is c\`adl\`ag in $H_2^{-1}(\CO)$ progressively measurable in $L^{\gamma+1}(\CO)$ and
		$$
		\EE\bigg[\sup_{t\in[0,T]} | u(t) | _{H_2^{-1}}^2\bigg] +\EE\bigg[\int_0^T | u(t) | _{L^{\gamma+1}}^{\gamma+1}\,dt\bigg] <\infty.
		$$
	\end{itemize}
\end{corollary}
\begin{proof} Theorem \ref{ther_main} is necessary for the proof of Corollary \ref{ExistencePorousGrad}. Let us provide the exact setup and the underlying function spaces to conclude this. The assumptions of Theorem \ref{ther_main} must then be confirmed. The corollary is a simple consequence of Theorem \ref{ther_main}, once the assumptions of Theorem \ref{ther_main} are confirmed. For notation convenience, we will denote $L^{\gamma+1}(\CO)$, $H_2^{-1}(\CO)$ and $L^\frac{\gamma+1}{\gamma}(\CO)$ by $V$, $H$ and $V^\ast$, respectively.
	 
For fixed $\Afrak$, $W$, we define the operator $$\Vcal:=\Vcal_{\Afrak,W}:\Mcal_{\Afrak}^{\gamma+1}(\X)\times L^{\gamma+1}(\Omega,\mathcal{F}_0,\P;L^{\gamma+1}(\CO))\break\to \Mcal_{\Afrak}^{\gamma+1}(\X), \text{ for } \xi\in \Mcal_{\Afrak}^{\gamma+1}(\X)$$ via
\begin{align*}
	\Vcal(\xi):=\Vcal(\xi,u_0):=u, 
\end{align*}where $u$ is the solution of the following SPDE
\begin{equation}\label{E31}
	\left\{
	\begin{aligned}
		d u(t)&=A(u(t))\, d t+\Sigma(\xi(t))\, dW(t), \\
		u(0)&=u_0,
	\end{aligned}
	\right.
\end{equation}where $A(u)=\Delta u^{[\gamma]}$ and {$\Sigma(\xi)=\Sigma_2(\xi)=\nabla \xi^{[\frac{1}{2}]}$}.

The operator {$\Sigma(\xi)=\Sigma_2(\xi)=\nabla \xi^{[\frac{1}{2}]}(t)$} satisfies the following properties: 
	\begin{enumerate}
	\item For $\xi\in \Xcal_\MA(R)$ there exists a constant $C>0$ such that \begin{align*}
		|\Sigma(\xi)|_{\mathcal{L}_2(L^2(\mathcal{O}),H)}^2 \leq C\big(1+|\xi|_V);
	\end{align*}
	\item For $\xi_1, \xi_2\in \Xcal_\MA(R)$, there exists a positive constant $C$ such that 
	$${ { | \Sigma(\xi_1)-\Sigma(\xi_2) | _{\mathcal{L}_2(L^2(\CO),H)}^2}\le C  | \xi_1-\xi_2 | _{V}.}
	$$
\end{enumerate}

 First, we assume that there exists a solution $u$ to the system \eqref{E31}. Using this solution, we aim to prove the existence of a solution to the system \eqref{E3}. To achieve our goal, we employ Theorem \ref{ther_main}. Therefore, we just need to verify the assumptions of Theorem \ref{ther_main}. Note that we already verified the same assumptions in Section \ref{Example2}, with a different diffusion. In view of the previous example, we will only provide verification for the noise term; the rest can be handled as in the previous example. Let us start with the energy estimate. 

\begin{step}
	\item \textbf{Verification of Condition (\ref{i})}. One should note that we only need to verify this condition of $A(u)=\Delta u^{[\gamma]}$, which is straightforward from Step (I), Example \ref{Example2}. Therefore, we are omitting the verification of this condition.

\item \textbf{Verification of Condition (\ref{ii})}. In this step, our aim is to show the invariance of the set $\Xcal_\MA(R)$, i.e., we verify the Assumption (\ref{ii}) and the main ingredient is the energy estimate established in the following claim:
\begin{claim}\label{CLI3}
	Let $R>0$, and let $\xi\in \Mcal_{\Afrak}^{\gamma+1}(\X)\cap \mathcal{X}_\MA(R)$. Then, the (unique) solution $u$ to the system \eqref{E31} satisfies
	\begin{align}\label{EG}
			\EE\bigg[\sup_{t\in[0,T]} | u(t) | _H^2\bigg]+ 4\EE\bigg[\int_0^T | u(t) | _V^{\gamma+1}\,dt \bigg]
		 \leq 2 \EE\big[ | u_0 | _{H}^2\big]+C (\gamma,T)	R^{\frac{1}{\gamma+1}},
		 \end{align} and for $u_0\in L^{\max\{\gamma+1,2p\}}(\Omega;L^{\gamma+1}(\mathcal{O}))$, where $p\in [1,\gamma+1)$, we have
		 \begin{align*}
		 &	\EE\bigg[\sup_{t\in[0,T]}	 | u(t) | _{H}^{2p}\bigg] +C\EE\bigg[\int_0^T | u(t) | _V^{\gamma+1}\,dt\bigg]^p \\& \leq \Big(C\EE\big[	 | u_0 | _{H}^{2p}\big]+CR^\frac{p}{\gamma+1}+CT\Big)e^{CT}. 
		 \end{align*}
\end{claim}
Assuming Claim \ref{CLI3}, we obtain the following invariance criterion: if $R>0$ is such that
\begin{align*} 2 \EE\big[ | u_0 | _{H}^2\big]+C (\gamma,T)	R^{\frac{1}{\gamma+1}}\leq R,
\end{align*}
then, for every $\xi\in\Xcal_{\MA}(R)$, the operator $\Vcal(\xi)$ belongs to $\Xcal_\MA(R)$. In particular, the set $\Xcal_\MA(R)$ is invariant under $\Vcal$. 
\begin{proof}[Proof of Claim \ref{CLI3}]
Applying It\^o's formula to the process $ | u | _H^2$ and performing a similar calculation as in the previous example, we obtain 
\begin{align}\label{EG1}\nonumber
&	\EE\bigg[\sup_{t\in[0,T]} | u(t) | _H^2\bigg]+ 2\EE\bigg[\int_0^T | u(s) | _V^{\gamma+1}\,ds \bigg]
	\\& \leq \EE\big[ | u_0 | _{H}^2\big]+\EE\bigg[\int_0^T | \Sigma(\xi(s))  | _{\mathcal{L}_2(U,H)}^2\,ds\bigg]+2\EE\bigg[\sup_{t\in [0,T]}\bigg|\int_0^t \big(\Sigma(\xi(s))\, dW(s),u(s)\big)\bigg|\bigg]. 
\end{align}Now, we consider the first term appearing on the right-hand side of \eqref{EG1} and estimate it as follows
\begin{align*}
	\EE\bigg[\int_0^T |\Sigma(\xi(s)) | _{\mathcal{L}_2(U,H)}^2\,ds\bigg]& \leq C	\EE\bigg[\int_0^T \big(1+| \xi(s) |_V\big)\,ds\bigg]
 \\&\leq C (T)+C(\gamma,T)	\bigg\{\EE\bigg[\int_0^T |  \xi(s) | _{V}^{\gamma+1}\,ds\bigg]\bigg\}^{\frac{1}{\gamma+1}}, 
\end{align*}provided $\gamma>1$.

For the final term appearing on the right-hand side of \eqref{EG1}, we use Burkholder-Davis-Gundy's and Young's inequalities and estimate as
\begin{align*}
&	2\EE\bigg[\sup_{t\in [0,T]}\bigg|\int_0^t \big(\Sigma(\xi(s))\, dW(s),u(s)\big)\bigg|\bigg] 
	\\& \leq C \EE\bigg[\bigg(\int_0^T  | \Sigma(\xi(s))| _{\mathcal{L}_2(U,H)}^2 | u(s) | _{H}^2\,ds\bigg)^{\frac{1}{2}}\bigg]
	\\& \leq C \EE\bigg[\sup_{s\in[0,T]} | u(s) | _{H}\bigg(\int_0^T   | \Sigma(\xi(s))| _{\mathcal{L}_2(U,H)}^2 \,ds\bigg)^{\frac{1}{2}}\bigg]
	\\& 
	\leq \frac{1}{2}\EE\bigg[\sup_{s\in[0,T]} | u(s) | _{H}^2\bigg]+ C	\EE\bigg[\int_0^T \big(1+| \xi(s) |_V\big)\,ds\bigg]
	\\& 
	\leq \frac{1}{2}\EE\bigg[\sup_{s\in[0,T]} | u(s) | _{H}^2\bigg]+C (T)+C(\gamma,T)	\bigg\{\EE\bigg[\int_0^T |  \xi(s) | _{V}^{\gamma+1}\,ds\bigg]\bigg\}^{\frac{1}{\gamma+1}}. 
\end{align*}
Combining the estimates above in \eqref{EG1}, we obtain the required energy estimate \eqref{EG}. 
\end{proof}

\item 
\textbf{Verification of Condition (\ref{iii})}. 
Our main goal is to prove the continuity of the operator $\Vcal$ in the norm induced by $\Mcal_\Afrak^{\gamma+1}(\X)$. Let $\xi_1,\xi_2\in \Mcal_\Afrak^{\gamma+1}(\X)\cap\Xcal_\Afrak(R)$ be arbitrary, then there exists a positive constant $C$ such that 
\begin{align*}
	\EE\big[\|\Vcal(\xi_1)-\Vcal(\xi_2)\|_\X ^{\gamma+1}\big] \leq C\Big\{\EE\big[\|\xi_1-\xi_2\|_\X^{\gamma+1}\big]\Big\}^\frac{1}{\gamma+1}.
\end{align*}
Let us consider the difference $u_1-u_2$, solving the following system 
\begin{equation*}
	\left\{
	\begin{aligned}
		d ( u_1-u_2)(t)&= \big(\Delta u_1^{[\gamma]}(t)-\Delta u_2^{[\gamma]}(t)\big)\, dt+\big(\Sigma(\xi_1(t))-\Sigma(\xi_2(t))\big)\, dW(t), \\
		(u_1-u_2)(0)&=0.
	\end{aligned}
	\right.
\end{equation*}Applying It\^o's formula to the process $|u_1-u_2|_H^2$ and performing a standard calculation, we find $\P$-a.s.,
\begin{align}\label{EG2}\nonumber
	&\EE\bigg[\sup_{t\in[0,T]} | u_1(t)-u_2(t) | _H^2\bigg]+\frac{1}{2^{\gamma-1}}\EE\bigg[ \int_0^T | u_1(s)-u_2(s) | _V^{\gamma+1}\,ds \bigg]
	\\&\nonumber
	\leq  \underbrace{\EE\bigg[\int_0^T | \Sigma(\xi_1(s))-\Sigma(\xi_2(s)) | _{\mathcal{L}_2(L^2(\CO),H)}^2\,ds\bigg]}_{I_1} \\&\quad +\underbrace{2\EE\bigg[\sup_{t\in[0,T]}\bigg|\int_0^t\big((\Sigma(\xi_1(s))-\Sigma(\xi_2(s)))\, dW(s), u_1(s)-u_2(s)\big)_H \bigg|\bigg]}_{I_2}. 
\end{align}Let us estimate the term $I_1$ in the following manner
\begin{align*}
	|I_1| &\leq \EE\bigg[\int_0^T | \Sigma(\xi_1(s))-\Sigma({\xi_2(s)}) | _{\mathcal{L}_2(L^2(\CO),H)}^2\,ds\bigg] \leq C \EE\bigg[\int_0^T  | \xi_1(s)-\xi_2(s) | _V\,ds \bigg].
\end{align*}Using Burkholder-Davis-Gundy's and Young's inequalities helps us to estimate the term $|I_2|$ as follows
\begin{align*}
	|I_2| &\leq C \EE\bigg[\bigg(\int_0^T | \Sigma({\xi_1(s)})-\Sigma({\xi_2(s)}) | _{\mathcal{L}_2(L^2(\CO),H)}^2 | u_1(s)-u_2(s) | _H^2\,ds \bigg)^\frac{1}{2}\bigg]
	\\&\leq 
	\frac{1}{2}\EE\bigg[\sup_{t\in[0,T]} | u_1(t)-u_2(t) | _H^2\bigg]+C\EE\bigg[\int_0^T | \Sigma({\xi_1(s)})-\Sigma({\xi_2(s)}) | _{\mathcal{L}_2(L^2(\CO),H)}^2\,ds\bigg] 
	\\&\leq 
		\frac{1}{2}\EE\bigg[\sup_{t\in[0,T]} | u_1(t)-u_2(t) | _H^2\bigg]+C \EE\bigg[\int_0^T  | \xi_1(s)-\xi_2(s) | _V\,ds \bigg].
\end{align*}Substituting the above estimates in \eqref{EG2}, we obtain 
\begin{align*}
	&\EE\bigg[\sup_{t\in[0,T]} | u_1(t)-u_2(t) | _H^2\bigg]+C(\gamma)\EE\bigg[ \int_0^T | u_1(s)-u_2(s) | _V^{\gamma+1}\,ds \bigg]
	\\&
	\leq  C \EE\bigg[\int_0^T  | \xi_1(s)-\xi_2(s) | _V\,ds \bigg]
	\\&
	\leq C (\gamma,T)\bigg\{\EE\bigg[\int_0^T  | \xi_1(s)-\xi_2(s) | _V^{\gamma+1}\,ds \bigg]\bigg\}^\frac{1}{\gamma+1}
\end{align*} which is the required continuity.

\item \textbf{Verification of Condition (\ref{iv})}. In this step, our aim is to demonstrate that the operator $\Vcal$ maps the set $\Xcal_\MA(R)$ to a pre-compact set. The idea is the same as in Example \ref{Example2}. 
From the energy estimate \eqref{EG}, there exists a positive constant $C$ depending on initial data $u_0$ and $R$ such that 
\begin{align*}
	\EE\bigg[\int_0^T|u(t)|_V^{\gamma+1}dt\bigg]\leq C, 
\end{align*}that means the compact containment condition holds.

{
Next, we move to the proof of the time regularity of the solution. Note that most of the computation in this part is similar as in Step (IV) of Example \ref{Example2}; therefore, we only provide the noise term estimate. 
\begin{align*}
C\EE	\bigg[\bigg|\int_{t\wedge s}^{t\vee s} \Sigma(\xi(r))\, dW(r)\bigg|_{H}^2\bigg]& \leq C \EE\bigg[\int_{t\wedge s}^{t\vee s} |\Sigma(\xi(r))|_{\mathcal{L}_2(L^2(\mathcal{O}), H)}^2\, dr \bigg] 
\\& \nonumber \leq C \EE\bigg[\int_{t\wedge s}^{t\vee s} \big(1+|\xi(r)|_V\big)\, d r\bigg]
\\&\nonumber\leq C |t-s|+C|t-s|^{\frac{\gamma}{\gamma+1}} \EE\bigg[\int_{t\wedge s}^{t\vee s} |\xi(r)|_V^{\gamma+1}\,d r\bigg]^{\frac{1}{\gamma+1}}
\\&\leq  \nonumber C |t-s|+C|t-s|^{\frac{\gamma}{\gamma+1}} \bigg\{\EE\bigg[\int_{0}^T |\xi(r)|_V^{\gamma+1}\,d r\bigg]\bigg\}^{\frac{1}{\gamma+1}}
\\&\leq C|t-s|+C|t-s|^{\frac{\gamma}{\gamma+1}}.
\end{align*} Using the estimates of the remaining terms from  Step (IV) of Example \ref{Example2}, we arrive at 
\begin{align*}\label{324}\nonumber
	&\int_0^T \int_0^T \frac{\EE\big[|u(t)-u(s)|_H^2\big]}{|t-s|^{1+2\alpha}}	
	\\& \nonumber \leq C\int_0^T \int_0^T \frac{1}{|t-s|^{1+2\alpha }} \Big(|t-s|^{\frac{2}{\gamma+1}}+|t-s|+|t-s|^{\frac{\gamma}{\gamma+1}}  \Big)\, ds \, dt 
	\\& <\infty,
\end{align*}provided $\alpha<\min\big\{\frac{\gamma}{2(\gamma+1)}, \frac{1}{\gamma+1}\big\}$. The rest follows from Step (IV) of Example \ref{Example2}.}

\item 
\textbf{Verification of Condition (\ref{v}).} 
{We only have to verify the last assumption of Theorem~\ref{ther_main}, that is, the Skorokhod property  $\Vcal_{\Afrak,W}(\Xcal_{\Afrak}(R))\subset \D([0,T];H),$ $\P$-a.s. This c\`adl\`ag property follows from the variational structure of \eqref{spdes} (see \cite{KrylovRozovskii,weiroeckner}).}
\end{step}Thus, we verified all the assumptions of Theorem \ref{ther_main}, and the assertion holds.
\end{proof}

\section{Proof of the stochastic Schauder-Tychonoff theorem \ref{ther_main}}\label{sec:proof_of_schauder}

In this section, we prove our main result (Theorem \ref{ther_main}). Note that we have collected several auxiliary results in
the Appendix, such as the extension of probability spaces, the Haar system, and the L\'evy-Ciesielski construction. 
\begin{proof}[Proof of Theorem \ref{ther_main}]

Fix
$\Afrak$ and $W$, and $R_1,\ldots,R_K>0$, $K\in\N$. In addition, for simplification, we shall omit $R_1,\ldots,R_K$ in the notation for $ \Xcal_{R_1,\ldots,R_K}(\Afrak)$ and write $ \Xcal_{\Afrak}(R)$ instead of $ \Xcal_{R_1,\ldots,R_K}(\Afrak)$. Fix an initial datum $w_0\in L^m(\Omega,\Fcal_0,\mathbb{P};H)$.

\begin{step}
\item In this step we approximate the operator $\mathcal{V}_{\Afrak,W}$. We discretise time, as we will apply the classical Schauder-Tychonoff theorem to a compact subset, which is given by a tight collection of laws on a finite-time grid. 
Let us fix a sequence $\{\ep_\iota\}_{\iota\in\N}$ such that $\ep_\iota\to 0$.

First, let us introduce a dyadic time grid $\pi_n=\{t^n_0=0<t^n_1<t^n_2<\cdots <t^n_{2^n}=T\}$
by $t^n_k= T\frac{k}{ 2^n}$, $k=0,1,2,\ldots,2^n$.
The stochastic process will be approximated by an averaging operator over the dyadic time interval.
To this end, let us define a step-function
$\phi_n:[0,T]\to[0,T]$ by $\phi_n(s)=T\frac{k}{ 2^n}$, if
$k=0,1,2,\ldots,2^n-1$ and $T\frac{k}{ 2^n}\le s<T\frac{k+1}{2^n}$, i.e.,
$\phi_n(s)= T2 ^ {-n}\lfloor 2 ^ ns\rfloor$, $s\ge 0$, where $\lfloor t\rfloor$ is the largest
integer that is less than or equal $t\in \mathbb{R}$. Let $\{w_n\}_{n\in\NN}\subset L^m(\Omega,\Fcal_0,\mathbb{P};V)$ be a sequence, such that $w_n\to w_0$ in $L^m(\Omega,\Fcal_0,\mathbb{P};H)$ and $$\Vert w_n-w_0\Vert_{L^m(\Omega,\Fcal_0,\mathbb{P};H)}\le \frac {\ep_n}n.$$

For a function $\xi\in\Mcal_{\Afrak}^{m}(\X)$, we define
\renewcommand{\Pro}{\mbox{Proj}}
\begin{equation*}
	 \Pro_n(\xi)(s):=\left\{
	 \begin{aligned}
	 	 &w_n, 
&&\mbox{ if } s\in [0,T2^{-n}),
\\
&\frac{2^n}{T}\int_{\phi_n(s)-T2^{-n}}^{\phi_n(s)} \xi(r)\: dr, &&\mbox{ if
} s\geq T2^{-n}.
	 \end{aligned}
\right.
\end{equation*}
Note that $\Pro_n(\xi)$ is a
progressively measurable, $\P$-a.s.~piecewise constant, $H$-valued stochastic process.
\begin{remark}\label{projection}
Observe that the projection operator satisfies the following properties (cf.~\cite[Appendix B]{BHM} or \cite[Appendix B]{FKEHMH}):
\begin{enumerate}[(i)]
	\item\label{PP1} $\Pro_n$ is a linear bounded contraction operator from $\mathbb{X}$ into $\mathbb{X}$;
	\item\label{PP2} {For every fixed $\xi\in \mathbb{X}$ one has $\Pro_n \xi\to \xi$ in $\mathbb{X}$, as
	$n\to\infty$ (strong, i.e.\ pointwise, convergence). We do not claim
	uniform convergence on bounded subsets of $\mathbb{X}$, which would be equivalent to
	operator-norm convergence $\Pro_n\to \mathrm{I}$ and fails in infinite-dimensions.}
\end{enumerate}
\end{remark}

{Thus, by Proposition \ref{prop:haar-uniform} (uniform approximation on compact subsets)},
 we know that for any $\ep>0$, there exists some $n_0\in\NN$ 
 such that for any $r\in (1,m]$,
$$
\Big\{ \EE\Big[\Vert\Pro_n(\xi)-\xi\Vert^r_{\mathbb{X}}\Big]\Big\}^\frac 1r \le  \ep ,\quad \forall \xi\in \mathcal{V}_{\mathfrak{A},W}\left( \mathcal{X}_\mathfrak{A}(R)\right),\quad \forall \, n\ge n_0.
$$
Let $\{n_\iota\}_{\iota\in\NN}$ be a sequence such that
\begin{equation*}
\Big\{ \EE\Big[ \Vert\Pro_{n_\iota}(\xi)-\xi\Vert_{\mathbb{X}}^{m}\Big]\Big\}^\frac{1}{m}\le \frac {\ep_\iota} 3,\quad \forall \xi\in \mathcal{V}_{\mathfrak{A},W}\left( \mathcal{X}_\mathfrak{A}(R)\right).
\end{equation*}
Finally, let us define the operator
$$
\mathcal{V}^\iota _{\mathfrak{A},W}(\xi):=  (\Pro_{n_\iota} \circ\mathcal{V}_{\mathfrak{A},W})(\xi),\quad \xi\in \mathcal{X}_\mathfrak{A}(R).
$$

\item\label{ST2}
Denote $\mathbb{H}:=\mathbb{D}([0,T];H)$, the Skorokhod space of c\`adl\`ag paths in $H$ endowed with the Skorokhod $J_1$-topology, (see \cite[Appendix A2]{Kallenberg}). Given the probability space $\mathfrak{A}=\left(\Omega,\CF,\mathbb{F},\PP\right)$,
for any $\iota \in\NN$ the operator $\mathcal{V}^\iota _{\mathfrak{A},W}$
induces  an operator $\mathscr{V}_\iota$ on $\mathscr{M}_1(\mathbb{X}\cap\mathbb{H})$ (the set of Borel probability measures on $\mathbb{X}\cap\mathbb{H}$).
The construction of the operator $\mathscr{V}_\iota$ is done in this step.
Let $\mathscr{X}$ denote the subset of probability measures $\mathscr{M}_1(\mathbb{X}\cap\mathbb{H})$, defined as
\begin{align*}
\lqq{ \mathscr{X}:=\left\{ \mu\in \mathscr{M}_1(\mathbb{X}\cap\mathbb{H}):\int_{\mathbb{X}}\Phi_j(\xi)\,\mu(d\xi)\le R_j,\right.
}
\\
&& \left. \phantom{\int_{\mathbb{X}}}
\qquad
\mbox{ and } \mu\left(\{\Psi_j<\infty\}\right)=1,\, \quad \forall j=1,\ldots,K\right\}.
\end{align*}

Now, let $\mu\in\mathscr{X}$. Then, by the Skorokhod lemma \cite[Theorem 4.30]{Kallenberg}, we know that there exists a
probability space  $(\Omega_0,\CF^0,\PP_0)$ and a random variable
$\xi:\Omega_0\to \mathbb{X}\cap\mathbb{H}$ such that the law of $\xi$ coincides with $\mu$.
In particular,
the probability measure $\nu_\xi:\CB(\mathbb{X}\cap\mathbb{H})\to[0,1]$ induced by $\xi$ and given by
$$
\nu_\xi:\CB(\mathbb{X}\cap\mathbb{H})\ni A\mapsto \PP_0\big(\big\{ \omega:\xi(\omega)\in A\big\}\big)
$$
coincides with the probability measure $\mu$.

Due to the definition of $\mathbb{H}$, we know that $\xi$ is a progressively measurable stochastic process, in particular, $\xi:\Omega_0\times [0,T]\to H$ such that $\PP_0(\xi\in\mathbb{X}\cap\mathbb{H})=1$.
Let
$$
\CG^0_t:=\sigma \Big( \big\{\,\xi (s)\;\colon\; 0\le s\le t \, \big\}\cup \CN_0\Big),\quad t\in [0,T],
$$
where $\CN_0$ denotes the collection of zero sets of ${\mathfrak{A}}_0$.
Set $\mathfrak{A}_0:=(\Omega_0,\CF^0,(\CG_t^0)_{t\in[0,T]},\PP_0)$.

Next, we have to construct the  Wiener process and extend the probability space.
Now, let $\mathfrak{A}_1=(\Omega_1,\CF^1,(\CG_t^1)_{t\in[0,T]},\PP_1)$ be a probability space
where a cylindrical Wiener process $W$ on $U$ is defined,
and let $\mathfrak{A}_{\mu}$ be the product probability space of $\mathfrak{A}_0$ and $\mathfrak{A}_1$.
In particular, we set
\begin{align*}
&\Omega_\mu =  \Omega_0\times \Omega_1,
\\
&\CF_\mu = \CF^0\otimes \CF^1,
\\
&\CG^\mu_t = \CG^0_t\otimes \CG^1_t,\quad t\in [0,T],
\\
\mbox{and}\quad &\PP_\mu = \PP_0\otimes \PP_1.
\end{align*}
Here, we know that $\xi(t)$, $t\in [0,T]$, is independent of the increments $W(t)-W(s)$, $t>s$ and $\{\CG_t^\mu\}_{t\in [0,T]}$-progressively measurable.
Since $\mu\in\mathscr{X}$, we know that the process $\xi$ is in the set $\mathcal{X}_{\mathfrak{A}_\mu}(R)$.

Next, we have to verify that the family    of operators
$$
\left\{ \Vcal^\iota_{\mathfrak{A}_\mu,{W}}: \iota\in\NN\right\}
$$
is well-defined. However, this follows from Assumption (\ref{i}),
and since the process $\xi$ belongs to the set $\mathcal{X}_{\mathfrak{A}_\mu}(R)$. In fact, the $\mu$-dependence of the stochastic basis can be removed by lifting to the space of probability measures (path laws). We aim to find a fixed-point in the space of probability measures.

Now, fix $B\in\CB(\mathbb{X}\cap\mathbb{H})$. Then, let us define the mapping $\mathscr{V}_\iota$ that maps
the probability measure $\mu$, in other words, the probability measure $\nu_\xi:\CB(\mathbb{X}\cap\mathbb{H})\to[0,1]$ that is induced by $\xi$ to the probability measure $\nu_{\mathcal{V}_{\Afrak_\mu,W}^\iota(\xi)}:\CB(\mathbb{X}\cap\mathbb{H})\to[0,1]$ given by
$$
\nu_{\mathcal{V}_{\Afrak_\mu,W}^\iota(\xi)}(B):=\PP_\mu\left( \left\{ \omega\in\Omega_\mu: \Vcal^\iota_{\mathfrak{A}_\mu,W}(\xi(\omega))\in B\right\}\right),\quad \mathscr{V}_\iota(\mu):=\nu_{\mathcal{V}_{\Afrak_\mu,W}^\iota(\xi)}.
$$
Note, since $\mathbb{X}\cap\mathbb{H}$ is a complete metric space,
the space of probability measures over $\mathbb{X}\cap\mathbb{H}$ equipped with the
Prokhorov metric\footnote{Let $\mathscr{M}_1(X)$ be the set of Borel probability measures on the metric space $(X,d)$ equipped with the weak topology. Let $\nu,\mu\in\mathscr{M}_1(X)$. Then the weak topology can be metrised by the \emph{Prokhorov metric}, cf. \cite{dudley2002},
$$d_\alpha(\mu,\nu):=\inf\{\alpha>0: \mu(A)\le \nu(A_\alpha)+\alpha \mbox{ and }
\nu(A)\le \mu(A_\alpha)+\alpha \mbox{ for all } A\in\CB(X)\},
$$
were $A_\alpha:=\{ x\in X: d(x,A)<\alpha$\}.}   is complete.

The following points can be easily verified:
\begin{enumerate}[(1)]
\item\label{1p}  $\mathscr{X}$  is invariant under $\mathscr{V}_\iota$.
This follows directly from the Assumption (\ref{ii})
and the properties of the projection $\Pro_{n_\iota}$, see Remark \ref{projection}.

\item\label{02} {Due to the Assumption (\ref{iii}), we know that the operator $\mathcal{V}_{\mathfrak{A}_\mu,W}^\iota$ restricted to $\mathcal{X}_{\mathfrak{A}_\mu}(R)$	is uniformly continuous. Therefore, the mapping $\mathscr{V}_\iota$ restricted to $\mathscr{X}$ is continuous
	on $\mathscr{M}_1(\X \cap\mathbb{H})$ in the Prokhorov metric by \cite[Theorem 11.7.1]{dudley2002}. A detailed explanation is added in the Remark \ref{Remcts}.}

\item\label{3p} Note that by the Assumption (\ref{v}), $\mathcal{V}^\iota_{\mathfrak{A}_\mu,W}(\xi)\in \mathbb{H}$ for $\xi\in\mathcal{X}_{\mathfrak{A}_\mu}(R)$. We claim that $\mathscr{V}_\iota$ restricted to $\mathscr{X}$ is compact on $\mathscr{M}_1(\mathbb{X}\cap\mathbb{H})$. 
In particular, it maps bounded sets into compact sets.
In fact, we have to show that for all $\iota\in\NN$ and $\ep>0$ there exists a compact subset $K_\ep\subset\mathbb{X}\cap\mathbb{H}$ such that
$$
\nu_ {\mathscr{V}_\iota(\xi)}\left ((\mathbb{X}\cap\mathbb{H})\setminus K_\ep\right)< \ep, \quad \forall {\nu_\xi\in \mathscr{X}}\quad \mbox{and}\quad \nu_ {\mathscr{V}_\iota(\xi)}:=\mathscr{V}_\iota(\nu_\xi).
$$
However, by the Assumption (\ref{iv})
there exists a constant $R>0$ with
$$
\EE\big[\Vert\mathcal{V}^\iota_{\mathfrak{A}_\mu,W}(\xi)\Vert_{\mathbb{X}_1}^{m_0}\big]\le R,\quad \xi\in\mathcal{X}_{\mathfrak{A}_\mu}(R).
$$
Let $\tilde{R}>R^{1/{m_0}}\ep^{-1/{m_0}}$
and let $K_\ep:=\{\zeta \in\mathbb{X}\cap\mathbb{H}:\Vert \zeta \Vert_{\mathbb{X}_1}\le \tilde{R}\}$.
Due to the construction of the operator $\mathscr{V}_\iota$, we know
that the law is preserved. In particular,
\begin{align*}
&\PP_\mu\left(  \left\{
\zeta\in\mathbb{X}\cap\mathbb{H}\cap \mathcal{X}_{\mathfrak{A}_\mu}(R) :\Vert\mathcal{V}_{\mathfrak{A}_\mu,W}^\iota(\zeta)\Vert_{\mathbb{X}_1} \ge \tilde{R}\right\} \right)
\\&=\nu_ {\mathscr{V}_\iota(\xi)}\left ( {{\left\{
		\zeta\in K_\ep^c:\Vert \zeta\Vert_{\mathbb{X}_1} \ge \tilde{R}\right\}}}\right).
\end{align*}
Next, by Chebyshev's inequality, we get that
$$\nu_{ \mathscr{V}_\iota(\xi)}\big(\mathbb{X}\cap\mathbb{H}\setminus K_\ep\big)
=
\nu_ {\mathscr{V}_\iota(\xi)}\left( \left\{
		\zeta\in K_\ep^c:\Vert \zeta\Vert_{\mathbb{X}_1} \ge \tilde{R}\right\}\right)
< \ep.$$
Since $\mathbb{X}_1\hookrightarrow \mathbb{X}$ compactly,
we have proved the tightness in the Banach space $\X$.
\item\label{4p} Now, we have to show that the set $\mathscr{X}$ is a convex subset of $\mathscr{M}_1(\mathbb{X}\cap\mathbb{H})$.
Let $\nu,\mu\in\mathscr{X}$. We have to show that for any $\alpha\in(0,1)$ we have $\alpha \nu+(1-\alpha)\mu\in\mathscr{X}$. 
First, analysing the expectation with respect to $\Phi_1,\ldots,\Phi_K$, this follows by the linearity of the expectation value. Secondly, since $\nu,\mu\in\mathscr{X}$
we know that $\nu\left(\{\Psi<\infty\}\right)=1$ and $\mu\left(\{\Psi<\infty\}\right)=1$. Let $\alpha\in(0,1)$. Then
\begin{align*}
	\left(\alpha \nu+(1-\alpha)\mu\right)\left(\{\Psi<\infty\}\right)
=\alpha \underbrace{\nu\left(\{\Psi<\infty\}\right)}_{=1}
	+(1-\alpha)\underbrace{\mu\left(\{\Psi<\infty\}\right)}_{=1}=1.
\end{align*}
\end{enumerate}
In particular, the mapping $\mathscr{V}_\iota$ restricted to $\mathscr{X}$
satisfies all assumptions of the classical Schauder-Tychonoff theorem, see \cite[§ 7, Theorem 1.13, p.~148]{granas}:
\begin{lemma}[Schauder-Tychonoff]
Let $\mathcal{C}$ be a nonempty convex subset of a locally convex linear topological space $\mathcal{E}$, and let $\mathcal{D}:\mathcal{C}\to\mathcal{C}$ be a compact map, i.e., $\mathcal{D}(\mathcal{C})$ is contained in a compact subset of $\mathcal{C}$. Then $\mathcal{D}$ has a fixed-point.
\end{lemma}

Hence, for any $\iota\in\NN$ there exists a probability measure $\nu^\ast_\iota\in\mathscr{X}$ such that $$\mathscr{V_\iota}(\nu^\ast_\iota)=\nu^\ast_\iota.$$

\item
Note that, the tightness argument in the point (\ref{3p}) of the Step (II) is independent of $\iota$, thus the family of measures
$$
\left\{ \nu^\ast_\iota:\iota\in\NN\right\}
$$				
is tight, therefore there exists a subsequence $\{\iota_j:j\in\NN\}$ and a Borel probability measure $\nu^\ast$ such that
$ \nu_{\iota_j}^\ast\to 	 \nu^\ast$, as $j\to\infty$.	
In this step, we will construct from the family of probability measures $\{\nu^\ast_{\iota_j}:j\in\NN\}$ and $\nu^\ast\in\mathscr{X}$, a filtered probability space $\mathfrak{A}^\ast$, a Wiener process
${W}^\ast$, a
progressively measurable process $w^\ast$, and
a family of progressively measurable processes  $\{ w_{\iota_j}^\ast:j\in\NN\}$ that
are $\PP^\ast$-a.s. contained in $\mathbb{X}\cap\mathbb{H}$
over $\mathfrak{A}^\ast$ such that these objects have probability measures $\{\nu^\ast_{\iota_j}:j\in\NN\}$ and $\nu^\ast\in\mathscr{X}$.

By the Skorokhod lemma \cite[Theorem 4.30]{Kallenberg},
there exists a probability space $\mathfrak{A}^\ast_0=(\Omega^\ast_0,\CF^{\ast}_0,\PP^\ast_0)$ and a  sequence of $\mathbb{X}$-valued random variables $\{ {w}^\ast_{\iota_j}:j\in\NN\}$
and ${w}^\ast$ where
the random variable $w^\ast_{\iota_j}:\Omega^\ast_0\to\mathbb{X}\cap\mathbb{H}$ has the  law $\nu^\ast_{\iota_j}$ in $\mathbb{X}\cap\mathbb{H}$.
In addition, by tightness and the Skorokhod lemma, we know that 
\begin{equation*}
{w}^\ast_{\iota_j}\to{w}^\ast,\quad \mbox{ as  $j\to \infty$ }\quad  {\P}^\ast_0\text{-a.s.,}
\end{equation*}
on $\X$.
Moreover, let us introduce the filtration $\mathcal{G}^{0,\ast}_t=(\mathcal{F}_t^{\ast,0})_{t\in[0,T]}$  given by
$$
\CF^{\ast,0}_t:=\sigma \left( \left\{\,(w _{\iota_j}^\ast(s),w^\ast(s))\;\colon\; 0\le s\le t, \,j\in\NN \right\}\cup \CN_0^\ast\right),\quad t\in [0,T],
$$
where $\CN_0^\ast$ denotes the collection of zero sets of ${\mathfrak{A}}_0^\ast$.

Next, similarly to above, let us construct the Wiener process.
Let 
$$
\mathfrak{A}_1^\ast=\left(\Omega_1^\ast,\CF_1^\ast,(\mathcal{G}_t^{1,\ast})_{t\in [0,T]},\PP_1^\ast\right)
$$
be a filtered probability space with a cylindrical Wiener process $W^\ast$ on $U$ being adapted to the filtration $(\mathcal{G}_t^{1,\ast})_{t\in [0,T]}$. 				Let $\mathfrak{A}^\ast:=\mathfrak{A}_0^\ast\times \mathfrak{A}_1^\ast$.	
In particular, we put
\begin{align*}
&\Omega ^\ast=  \Omega_0^\ast\times \Omega^\ast_1,
\\
&\CF^\ast = \CF_0^\ast\otimes \CF_1^\ast,
\\
&\CG^\ast_t = \CG^{0,\ast}_t\otimes \CG^{1,\ast}_t,\quad t\in [0,T],
\\
\mbox{and}\quad &\PP^\ast = \PP^\ast_0\otimes \PP_1^\ast.
\end{align*}
On $\mathfrak{A}^\ast$, $\mathcal{X}_{\mathfrak{A}^\ast}(R)$ can be defined in the obvious way, as well as the operators $\mathcal{V}^{\iota_j}_{\mathfrak{A}^\ast,W^\ast}$
and ${\mathcal{V}}_{\mathfrak{A}^\ast,W^\ast}$

\item

In this step, we construct a $\P^\ast$-a.s. piecewise constant and $\{\CG^\ast_t\}_{t\in [0,T]}$-progressively measurable process that is a fixed-point for the operator $\mathcal{V}_{\mathfrak{A}^\ast,W^\ast}^{\iota_j}$.

Since $\nu^\ast_{\iota_j}\in \mathscr{X}$, the process $w^\ast_{\iota_j}\in\mathcal{X}_{\mathfrak{A}^\ast}(R)$ and, hence, 
$\mathcal{V}_{\mathfrak{A}^\ast,W^\ast}^{\iota_j}(w^\ast_{\iota_j})$ is well-defined.
Since $\mathscr{V}_{\iota_j}(\nu^\ast_{\iota_j})=\nu^\ast_{\iota_j}$, $\Law(w^\ast_{\iota_j})=\nu^\ast_{\iota_j}$.
However, we do not know if
the process $w^\ast_{\iota_j}$ satisfies
\begin{align*}
\PP^\ast\left(
\mathcal{V}_{\mathfrak{A}^\ast,W^\ast}^{\iota_j}\big(w^\ast_{\iota_j}\big) (s) =w^\ast_{\iota_j}(s)\right)=&1, \quad \mbox{for} \quad 0\le s\le T
.
\end{align*}

In this step, we are going to construct a fixed-point for the operator $\mathcal{V}_{\mathfrak{A}^\ast,W^\ast}^{\iota_j}$.
Let us define a new process by induction. To start with, let
\begin{equation}\label{nummer1}
w^\ast_{\iota_j,1}(s) :=  \begin{cases}
	w^\ast_{\iota_j}(s),& \mbox{ if } 0\le s<t_1,
	\\
	\left(\mathcal{V}_{\mathfrak{A}^\ast,W^\ast}^{\iota_j}(w^\ast_{\iota_j})\right)(s),& \mbox{ if } t_1\le s\le T,
\end{cases}
\end{equation}
and
\begin{equation}\label{nummer2}
w^\ast_{\iota_j,2}(s) :=\begin{cases}
	w^\ast_{\iota_j,1}(s),& \mbox{ if } 0\le s<t_2,
	\\
	\left(\mathcal{V}_{\mathfrak{A}^\ast,W^\ast}^{\iota_j}(w^\ast_{\iota_j,1})\right)(s),& \mbox{ if } t_2\le s\le T.
\end{cases}
\end{equation}
Now, having defined $w^\ast_{\iota_j,k}$, let
\begin{equation}\label{nummerk}
w^\ast_{\iota_j,k+1}(s) :=  \begin{cases}
	w^\ast_{\iota_j,k}(s),& \mbox{ if } 0\le s<t_{k+1},
	\\
	\left(\mathcal{V}_{\mathfrak{A}^\ast,W^\ast}^{\iota_j}(w^\ast_{\iota_j,k})\right)(s),& \mbox{ if } t_{k+1}\le s\le T,
\end{cases}
\end{equation}
where $t_k\in\pi_n$ are dyadic time points.
Let us put $  w^\ast_{\iota_j,0}(s)=w_0^\ast$, for $0\le s\le T$, where $ w_0^\ast$ is a $\Gcal^\ast_0$-measurable version of $w_0$, and
\begin{equation}\label{definfty}
w_{\iota_j,\infty}^\ast(s) :=   w^\ast_{\iota_j,k}(s),\quad \mbox{if}\quad t^{\iota_j}_{k-1}\le s< t^{\iota_j}_{k},\,\,k=1,\ldots, 2^{\iota_j}.
\end{equation}
We claim that the process $w^\ast_{\iota_j,\infty}$ satisfies
\begin{equation*}
\PP^\ast\left(
\mathcal{V}_{\mathfrak{A}^\ast,W^\ast}^{\iota_j}\big(w^\ast_{\iota_j,\infty}\big) (s) =w^\ast_{\iota_j,\infty }(s)\right)=1, \quad \mbox{for} \quad 0\le s\le T
.
\end{equation*}
Note that
by the definition of $\Pro_{\iota_j}$ on $[0,t_1^{\iota_j})$,
the process on $[0,t_{1}^{\iota_j})$ is defined by the initial data.
In fact, we have
for $0\le s< t^{\iota_j}_1$
$$
\mathcal{V}_{\mathfrak{A}^\ast,W^\ast}^{\iota_j}\big(w^\ast_{\iota_j,\infty}\big)(s)= w^\ast_0.
$$
By equation \eqref{definfty} and equation \eqref{nummer1} we have $w^\ast_{\iota_j,\infty}(s)=w^\ast_{\iota_j,1}(s)=w^\ast_0$, 
for $0\le s< t^{\iota_j}_1$.
In particular, the process on $[0,t_{1}^{\iota_j})$ is defined by the initial data
and we have  $\PP^\ast$-a.s.
$$\mathcal{V}_{\mathfrak{A}^\ast,W^\ast}^{\iota_j}\big(w^\ast_{\iota_j,\infty}\big)(s)=w^\ast_{\iota_j,\infty}(s)
, \quad \mbox{for} \quad 0\le s<t^{\iota_j}_1.
$$
At time $t^{\iota_j}_1$, we have by equation \eqref{definfty}
and equation \eqref{nummer1}, 
$$
\mathcal{V}_{\mathfrak{A}^\ast,W^\ast}^{\iota_j}\big(w^\ast_{\iota_j,\infty }\big)(t^{\iota_j}_1)=\mathcal{V}_{\mathfrak{A}^\ast,W^\ast}^{\iota_j}  \big(w^\ast_{\iota_j}\big)(t^{\iota_j}_1)=w^\ast_{\iota_j,1}(t^{\iota_j}_1).
$$
However, we have $ w^\ast_{{\iota_j },\infty}(t^{\iota_j}_1)=w^\ast_{\iota_j,1}(t^{\iota_j} _1)$.
Let us analyse what is happening at the next time interval $[t^{\iota_j }_1,t^{\iota_j }_2)$. {Here the process} is constant and equals   $\PP^\ast$-a.s. to the value at $t^{\iota_j }_1$, i.e.
$$
\mathcal{V}_{\mathfrak{A}^\ast,W^\ast}^{\iota_j}\big( w^\ast_{{\iota_j },\infty}\big)(s)
=w^\ast_{\iota_j ,1}(s),  \quad \mbox{for} \quad t^{\iota_j}_1\le s<t^{\iota_j}_2.
$$
Note, also that   $\PP^\ast$-a.s. we have $w^\ast_{\iota_j ,1}(s)=w^\ast_{\iota_j ,\infty}(s)$, for $ t^{\iota_j}_1\le s<t^{\iota_j}_2$, and hence
$$
\PP^\ast\left(\mathcal{V}^{\iota_j }_{\MA^\ast,W^\ast}\big(w^\ast_{\iota_j ,\infty }\big)(s)= w^\ast_{\iota_j ,\infty }(s)\right) =1,  \quad \mbox{for} \quad t^{\iota_j}_1\le s<t^{\iota_j}_2.
$$
Let us analyse what happens in $t^{\iota_j}_2$. By equation \eqref{definfty}, we have
$$
\mathcal{V}_{\mathfrak{A}^\ast,W^\ast}^{\iota_j}\big(w^\ast_{\iota_j,\infty}\big)(t^{\iota_j}_2)=\mathcal{V}_{\mathfrak{A}^\ast,W^\ast}^{\iota_j}\big(w^\ast_{\iota_j,1}\big)(t^{\iota_j}_2).
$$
From the equation \eqref{nummer2} of $ w^\ast_{\iota_j,2}$,  we have 
$$
\mathcal{V}_{\mathfrak{A}^\ast,W^\ast}^{\iota_j}\big(w^\ast_{\iota_j,\infty}\big)(t^{\iota_j}_2)=w^\ast_{\iota_j,2}(t^{\iota_j}_2).
$$
Now, we can proceed by induction. Let us assume that in $[0,t^{\iota_j}_k)$ we have shown that 
\begin{equation*}
\PP^\ast\left(
\mathcal{V}_{\mathfrak{A}^\ast,W^\ast}^{\iota_j}\big(w^\ast_{\iota_j,\infty }\big) (s) =w^\ast_{\iota_j,\infty}(s)\right)=1, \quad \mbox{for} \quad 0\le s\le t^{\iota_j}_k
.
\end{equation*}
Then,  we have by equation \eqref{definfty} and equation \eqref{nummerk} for $t^{\iota_j}_k\le s< t^{\iota_j}_{k+1}$
$$
\mathcal{V}_{\mathfrak{A}^\ast,W^\ast}^{\iota_j}\big(w^\ast_{\iota_j,\infty}\big) (s) =\mathcal{V}_{\mathfrak{A}^\ast,W^\ast}^{\iota_j}\big(w^\ast_{\iota_j,k-1}\big)(s)=w^\ast_{\iota_j,k}(s).
$$
By tackling one interval after another, one can show that
\begin{equation*}
	\PP^\ast\left(
	\mathcal{V}_{\mathfrak{A}^\ast,W^\ast}^{\iota_j}\big(w^\ast_{\iota_j,\infty }\big) (s) =w^\ast_{\iota_j,\infty}(s)\right)=1, \quad \mbox{for} \quad 0\le s\le T
	.
\end{equation*}

\del{ \item 

We have to identify the limit
of the sequence $\{w^\ast_{\iota_j,\infty}:j\in\NN\}$.	
Note, the definition of given by\begin{eqnarray}\label{hatdefined-1} \Pro_n(\xi)(s):=
	\bcase w_n, &
	\mbox{ if } s\in [0,T2^{-n}),
	\\
	\frac{2^n}{T}\int_{\phi_n(s)-T2^{-n}}^{\phi_n(s)} \xi(r)\: dr, &\mbox{ if
	} s\geq T2^{-n}.\ecase
\end{eqnarray}
and 
$$
\mathcal{V}^\iota _{\mathfrak{A},W}(\xi):=  (\Pro_{n_\iota} \circ\mathcal{V}_{\mathfrak{A},W})(\xi),\quad \xi\in \mathcal{X}(\mathfrak{A}).
$$

}

\item

Next, we verify several statements to pass to the limit. We point out that the same construction used for $w^\ast_{\iota_j,\infty}$ can be carried out on the original probability space $\mathfrak{A}$.
The resulting process will be denoted by $w_{\iota_j,\infty}$.
Due to the construction and the properties of the projection, it is straightforward to verify that the laws are preserved. In particular, we have $\mathscr{L}(w_{\iota_j,\infty}) = \mathscr{L}(w^\ast_{\iota_j,\infty})$.

\del{\begin{claim}\label{claim222}
	We claim that:
	\begin{enumerate}[(1)]
		\item there exists a constant $C>0$ such that 
		\[
		\sup_{j\in\mathbb{N}} \mathbb{E}^\ast \left[ \| w^\ast_{\iota_j,\infty} \|_{\mathbb{X}}^{m_0} \right] \leq C,
		\]
	\end{enumerate}
\end{claim}

\begin{proof}[Proof of Claim \ref{claim222}]
	Since $\{ w^\ast_{\iota_j,\infty} \}_{j\in\mathbb{N}} \subset \mathcal{X}(\mathfrak{A}^\ast)$ and $\mathcal{X}(\mathfrak{A}^\ast)$ is bounded in $\mathbb{X}$, we conclude from the application of Skorokhod’s lemma that
	\[
	\mathbb{E} \left\| w_{\iota_j,\infty} \right\|_{\mathbb{X}}^{r} = \mathbb{E}^\ast \left\| w^\ast_{\iota_j,\infty} \right\|_{\mathbb{X}}^{r},
	\]
	for all $r\in [1, m_0]$. Therefore, under assumption (iv), we obtain
	\[
	\sup_{j} \mathbb{E}^\ast \left\| w^\ast_{\iota_j,\infty} \right\|_{\mathbb{X}}^{m_0} \leq R \tilde{C} =: C,
	\]
	where $\tilde{C}>0$ is a constant satisfying $\| \cdot \|_{\mathbb{X}} \leq \tilde{C} \| \cdot \|_{\mathbb{X}_1}$, which exists due to the compact and dense embedding $\mathbb{X}_1 \hookrightarrow \mathbb{X}$.
	As a consequence, the family $\{ \| w^\ast_{\iota_j,\infty} \|_{\mathbb{X}}^r \}$ is uniformly integrable for any $r \in (1, m_0]$ with respect to $\mathbb{P}^\ast$.
	Next, we demonstrate that $\{ w_{\iota_j}^\ast : j\in\mathbb{N} \}$ is a Cauchy sequence in $\mathbb{X}$.}
Recall that
\[
\mathcal{V}^{\iota_j}_{\mathfrak{A},W}(\xi) :=  (\Pro_{n_{\iota_j}} \circ \mathcal{V}_{\mathfrak{A},W})(\xi), \quad \xi \in \mathcal{X}_{\mathfrak{A}}(R),
\]
where $ \Pro_{n}$ denotes a shifted Haar projection on the dyadic grid given by  $t^n_k= T\frac{k}{ 2^n}$, $k=0,1,2,\ldots,2^n$. For more details, see Appendix \ref{app:haar-system}.
That means in our situation, the fixed-point ${w}_{\iota_j,\infty}^\ast$ can be written as follows:
	\begin{align}\label{a1}\nonumber
		\big( {w}_{\iota_j,\infty}^\ast(t),\phi \big) 
		&=\big( x_{\iota_j},\phi\big) +\int_{0}^{t}\la  A ( u_{\iota_j}^\ast(s)),\phi \ra \,ds+\int_{0}^{t}\la F_1 ( u_{\iota_j}^\ast(s))+ F_2 ( u_{\iota_j}^\ast(s)),\phi \ra \,ds
	\\&\quad	+\int^t_0 \big( {\Sigma( {u}_{\iota_j}^\ast)}\,dW(s),\phi\big) , 
		\ \ t\in [0,T],\ \ \phi\in V, 
	\end{align} 
	and $u_{\iota_j}^\ast $ is
	defined by 
	\begin{equation*}
		 u_{\iota_j}^\ast:=
		\left\{
		\begin{aligned} &w_{\iota_j,\infty}^\ast, &&
		\mbox{ if } s\in [0,T2^{-n}),
		\\
		&\frac{2^n}{T}	\int_{t^n_{k}-t^n_1}^{t^n_k} w_{\iota_j,\infty}^\ast(r)\: dr, &&\mbox{ if
		} s \geq  T2^{-n} .
	\end{aligned} 
\right.
	\end{equation*}
	{The $H$-valued process} $u_{\iota_j}^\ast$ is
	piecewise constant {adapted and}  hence progressively measurable.
	Between the grid points, equation \eqref{a1} is linear and has values in $V$, therefore,
	${w}_{\iota_j,\infty}^\ast$ is well defined in $V^\ast$ for all $j\in{\mathbb{N}}$. One should note that we are denoting the sequence $n_{\iota_j}$ by $\iota_j$, for the sake of convenience.

	Let us collect the information known about the process ${w}_{\iota_j,\infty}^\ast$:
\begin{itemize}
	\item  {For any $j\in\NN$,  $w^\ast_{\iota_j,\infty}$ is {$\PP^\ast$}-a.s.\ a fixed-point of  $	\mathcal{V}_{\mathfrak{A}^\ast,W^\ast}^{\iota_j}$. In particular,  $w^\ast_{\iota_j,\infty}$ solves  {$\PP^\ast$}-a.s.\ the following equation
		\begin{equation} \label{8.1}
			\left\{\begin{aligned}
				d w^\ast_{\iota_j,\infty}(s)&=A^{{\iota_j}}(w^\ast_{\iota_j,\infty}(s))\,ds+F_1^{\iota_j}(w^\ast_{\iota_j,\infty}(s))+F_2^{\iota_j}(w^\ast_{\iota_j,\infty}(s))\,ds\\&\quad +\big(\Sigma_1^{{\iota_j}}(w^\ast_{\iota_j,\infty}(s)) +\Sigma_2^{{\iota_j}}(w^\ast_{\iota_j,\infty}(s))\big)\, d W(s)  ,
				\\
				w^\ast_{\iota_j,\infty}(0)&=x_{\iota_j},
			\end{aligned}\right.
		\end{equation}
		where $A^{{\iota_j}}(\xi)=A(\Pro_{{\iota_j}}\xi)$, $F_i^{\iota_j}(\xi)=F_i(\Pro_{\iota_j}\xi )$ and 	$\Sigma_i^{{\iota_j}} (\xi):=  \Sigma_i(\Pro_{{\iota_j}}\xi)$, for $i\in\{1,2\}$.}	
	
	\item There exists a constant $C>0$ such that
	\[
	\mathbb{E}^{\ast} \big[\left\| \mathcal{V}_{\mathfrak{A}^{\ast},W^{\ast}}(\xi) \right\|_{\mathbb{X}_1} \big]\leq C, \quad \xi \in \mathcal{X}_{\mathfrak{A}^{\ast}}(R).
	\]
	In particular, the set
	\[
	\left\{ \mathcal{V}_{\mathfrak{A}^{\ast},W^{\ast}}(\xi) : \xi \in \mathcal{X}_{\mathfrak{A}^{\ast}}(R) \right\}
	\]
	is tight in $\mathbb{X}$.
	
	\item 
	For every $\varepsilon > 0$, there exists a compact subset $K_\varepsilon \subset \mathbb{X}$ such that
	\[
	\mathbb{P}^{\ast} \left( \xi \notin K_\varepsilon \right) \leq \varepsilon.
	\]
	Now, we decompose
	\[
	\mathbb{E}^{\ast}\Big[\big\|\Pro_{{\iota_j}}\xi - \xi \big\|^p_{\mathbb{X}}\Big]
	\le 
	\mathbb{E}^{\ast} \Big[ \mathds{1}_{K_\varepsilon} \big\| \Pro_{{\iota_j}}\xi - \xi \big\|^p_{\mathbb{X}} \Big]
	+
	\mathbb{E}^{\ast} \Big[ \mathds{1}_{\mathbb{X}\setminus K_\varepsilon} \big\| \Pro_{{\iota_j}}\xi - \xi \big\|^p_{\mathbb{X}} \Big].
	\]
	By compactness of $K_\varepsilon$ and by Proposition \ref{prop:haar-uniform}, we know for sufficiently large $n_{\iota_j}$, that
	\[
	\big\| \Pro_{{\iota_j}} \xi  - \xi \big\|_{\mathbb{X}} \leq \varepsilon^{1/p}, \quad \forall\, \xi \in K_\varepsilon.
	\]
	Thus,
	\begin{align*}
	\mathbb{E}^{\ast}\Big[ \big\| \Pro_{{\iota_j}}\xi - \xi \big\|^p_{\mathbb{X}}\Big]
	\leq 
	\varepsilon +\big(  \mathbb{P}^{\ast} \left( \xi \notin K_\varepsilon \right)\big)^{\frac{r-p}{r}} \left\{\mathbb{E}^{\ast} \Big[\big\| \Pro_{{\iota_j}}\xi - \xi \big\|^r_{\mathbb{X}}\Big]\right\}^{\frac{p}{r}}
	\leq 2\varepsilon,
\end{align*} for $m>r\geq p$, where we used boundedness of moments for the second term in the above inequality from Step (I).	
\end{itemize}
\newcommand{\Xw}{\widetilde{w}}
\newcommand{\Xwt}{\widetilde{w}}
\newcommand{\Xwtt}{\widetilde{\widetilde{w}}}

It follows from above that the law of the random variables  $\{{w}_{\iota_j,\infty}^\ast:j\in\NN\}$ over $\mathfrak{A}^\ast$ 
is tight in $\mathbb{X}$.  In addition, we know that 
$\{ w _{{\iota_j}}\}_{j\in\mathbb{N}}$ is a sequence of $\mathbb{X}\cap \mathbb{H}$-valued random variables
such that $\left\{\Law\big( w _{{\iota_j}}\big):j\in\N\right\}$ is weakly convergent to a probability measure $\mathscr{P}^\ast $ on $\mathbb{H}$.

\item 
In the next step, we apply the Skorokhod representation theorem.
To be more precise, by the Skorokhod theorem, there exists a probability space $\widetilde{\mathfrak{A}}:={\tiny }(\widetilde{\Omega}, \widetilde{\mathcal{F}}, \widetilde{\mathbb{P}})$ and $\X\cap \mathbb{H}$-valued random variables $\{\Xwt_{{\iota_j}}\}_{
	j\in\N}$, and $\Xwt,$ such that 
\begin{align}\label{assapp}
	\mathscr{L}( \Xwt ) &=\mathscr{P}^\ast ,
	\\ 
	\mathscr{L}\big( w _{{\iota_j}}\big)&=\mathscr{L}\big(\Xw_{{\iota_j}}\big),\ j \in \mathbb{N},\label{assapp5}
\end{align}
and 
\begin{align}
	\Xwt _{{\iota_j}} &\rightarrow  \Xwt  \mbox{  as }j \rightarrow \infty, \mbox{ $\wi{\mathbb{P}}$-a.s.}\label{assapp2}
\end{align}
Additionally, we know from \eqref{a1} that for all $j\in\NN$ we have for all $t\in [0,T]$, $\wi\PP$-a.s.
{\begin{align*} \nonumber 
		\big( \wi{w}_{{\iota_j}} (t),\phi \big)
	&=\big( x_{{\iota_j}}, \phi\big)+\int_{0}^{t} 
\big\la 	A ( \Pro_{{\iota_j}}\wi{w}_{{\iota_j}}(s))+F_1(\Pro_{\iota_j}(\wi{w}_{\iota_j}))+F_2(\Pro_{\iota_j}(\wi{w}_{\iota_j})),\phi\big\ra \, ds
	\\	&\quad +\bigg( \int^t_0 \big(\Sigma_1(\Pro_{{\iota_j}}	\wi{w}_{{\iota_j}})(s)+\Sigma_2((\Pro_{{\iota_j}}	\wi{w}_{{\iota_j}})(s)) \big)\,d\wi{W}(s),\phi\bigg) ,
	\label{a}
	\quad \,\phi\in V.
\end{align*} }
Despite the fact that $\wi{w} $ is already given, we have to construct a Wiener process $\wi{ W}$ such that $\wi{w}$ can be written as a solution of an SPDE on a probability space with a Wiener process. The difficulty is that we do not necessarily have full information available on this space. To obtain a Wiener process with the same characteristics as the original one, we reconstruct the part determined by $\wi{w}$ and generate the missing (unknown) part on an extended probability space.

\del{ 	\begin{claim}\label{claim}
		We need to show the following:
		\begin{enumerate}
			\item there exists a positive constant $C$ such that 
			\begin{align*}	
				\sup_{j\in\N}\wi{\EE}\Big[\big\|\wi{w}_{{\iota_j}}\big\|_\X^m\Big]\leq C;
			\end{align*}
			\item for any $r\in (1,m)$, we have 
			\begin{align*}	
				\lim_{j\to\infty}\wi{\EE}\Big[\big\|\wi{w}_{{\iota_j}}-\wi{w}\big\|_\X^r\Big]=0.
			\end{align*}
		\end{enumerate}
	\end{claim}
	\begin{proof}[Proof of Claim \ref{claim}]
		From the previous steps, we know that the operator $\Vcal_{\wi{\MA},\wi{W}}^{\iota_j}$ is having a fixed-point $\wi{w}_{{\iota_j}}$, which belongs to $\Vcal_{\wi{\MA},\wi{W}}^{\iota_j}(\Xcal_R(\wi{\MA}))$. Since $\{\wi{w}_{{\iota_j}}\}_{j\in\N}\subset \Vcal_{\wi{\MA},\wi{W}}^{\iota_j}(\Xcal_R(\wi{\MA}))$, we know from Assumption (\ref{iv}), that for any $1<r< m$  there exists some positive constant $R$ such that 
		\begin{align*}	
			\wi{\EE}\Big[\big\|\wi{w}_{{\iota_j}}\|_\X^r\Big]\leq R,\ \text{ for every }\ j\in\N,
		\end{align*}which completes the first part of this claim. 
		
		From the above expression, we conclude that $\Big\{\big\|\wi{w}_{{\iota_j}}\big\|_\X^r\Big\}_{j\in\N}$ is uniformly integrable for any $r\in(1,m)$, with respect to the probability measure $\wi{\PP}$. Finally, using the convergence obtained in \eqref{assapp2}, and Vitali convergence theorem, we conclude that 
		\begin{align*}	
			\lim_{j\to\infty}\wi{\EE}\Big[\big\|\wi{w}_{{\iota_j}}-\wi{w}\big\|_\X^r\Big]=0, 
		\end{align*}for any $r\in (1,m)$.
	\end{proof}
	
}		
To conclude, in this  step we will  construct an extension\footnote{For the definition of the extension, see Appendix \ref{extension}, Definition \ref{Definition_extension}.} 
$$\widetilde{\widetilde{\mathfrak{A}}}$$
 of the original probability space $\widetilde{\mathfrak{A}}$ and a Wiener process 
$\widetilde{\widetilde{W}}$ on $\widetilde{\mathfrak{A}}\times \widetilde{\widetilde{\mathfrak{A}}}$ such that $\widetilde{w}_{n_{\iota_j}}$ solves
	\newcommand{\wwt}{\widetilde{w}}
{\begin{align*} \nonumber \big( \wi{w}_{{\iota_j}} (t),\phi \big)
		&=\big( x_{{\iota_j}}, \phi\big)+\int_{0}^{t} 
		\big\la 	A ( \Pro_{{\iota_j}}\wi{w}_{{\iota_j}}(s))+F_1(\Pro_{\iota_j}(\wi{w}_{\iota_j}))+F_2(\Pro_{\iota_j}(\wi{w}_{\iota_j})),\phi\big\ra\, ds
		\\	&\quad +\bigg( \int^t_0 \big(\Sigma_1(\Pro_{{\iota_j}}	\wi{w}_{{\iota_j}})(s)+\Sigma_2((\Pro_{{\iota_j}}	\wi{w}_{{\iota_j}})(s)) \big)\,d\wi{\wi{W}}(s),\phi\bigg), 
		 \end{align*} $t\in[0,T],\  \phi\in V$.}
\begin{remark}\label{wtilde}
	$$
	\pi \wi{\wi{\omega}}=\wi \omega \quad \text { for } \quad \wi {\wi{\omega}}=\left(\wi \omega, \omega^{\prime}\right) \in \wi{\wi{\Omega}}.
	$$
	In particular, the process $\wi{w}$  remains unchanged under this extension. However, for consistency and in order to clearly identify the corresponding probability space in the sequel, we denote $\wi{w}$  on the enlarged space by $\wi{\wi{w}}$. 
\end{remark}
	
	First, observe that for each $j\in\mathbb{N}$ the process $\widetilde{M}_{\iota_j}(t)$, for $t\in[0,T]$ defined by 
	\begin{align*}
		\widetilde{M}_{\iota_j}(t)
		:=\widetilde{w}_{\iota_j}(t)-x_{\iota_j}
		-\int_0^t \Big(A (\Pro_{\iota_j}\widetilde{w}_{\iota_j}(s))+F_1 (\Pro_{\iota_j}(\wi{w}_{\iota_j}))+F_2 (\Pro_{\iota_j}(\wi{w}_{\iota_j}))\Big)\,ds
	\end{align*}
	is an $H$-valued square integrable martingale with respect to the filtration
	\begin{align*}
		\widetilde{\mathcal{F}}_{\iota_j}(t)
		=\sigma\!\left\{\,\widetilde{w}_{\iota_j}(s): s\le t \right\},
		\qquad t\in[0,T].
	\end{align*}
	Its quadratic variation is given by
	\begin{align*}
		\big\langle\!\big\langle \widetilde{M}_{\iota_j}(t) \big\rangle\!\big\rangle
		=\int_0^t 
		\bigl(\Sigma(\Pro_{\iota_j}\widetilde{w}_{\iota_j}(s))\,Q^{\frac{1}{2}}\bigr)
		\bigl(\Sigma(\Pro_{\iota_j}\widetilde{w}_{\iota_j}(s))\,Q^{\frac{1}{2}}\bigr)^{*}
		\,ds,
		\qquad t\in[0,T],
	\end{align*}where $\Sigma=\Sigma_1+\Sigma_2$.
\\	This follows directly from \eqref{assapp5} and the fact that the process $M_{\iota_j}(t)$, for $t\in[0,T]$ defined by
	\begin{align*}
		M_{\iota_j}(t)
		:=w_{\iota_j}(t)-x_{\iota_j}
		-\int_0^t\Big( A(\Pro_{\iota_j}w_{\iota_j}(s))+F_1 (\Pro_{\iota_j}({w}_{\iota_j}))+F_2(\Pro_{\iota_j}w_{\iota_j}(s))\Big)\,ds,
	\end{align*}
	is a $H$-valued square integrable martingale with respect to the filtration
	\begin{align*}
		\mathcal{F}_{\iota_j}(t)
		=\sigma\!\left\{ w_{\iota_j}(s): s\le t \right\},
		\qquad t\in[0,T],
	\end{align*}
	having quadratic variation
	\begin{align*}
		\big\langle\!\big\langle M_{\iota_j}(t)\big\rangle\!\big\rangle
		=\int_0^t 
		\bigl(\Sigma(\Pro_{\iota_j}w_{\iota_j}(s))\,Q^{\frac{1}{2}}\bigr)
		\bigl(\Sigma(\Pro_{\iota_j}w_{\iota_j}(s))\,Q^{\frac{1}{2}}\bigr)^{*}
		\,ds,
		\ t\in[0,T],
	\end{align*}where $\Sigma=\Sigma_1+\Sigma_2$.

\noindent 	
	This follows from identity \eqref{8.1} and the construction of the operator 
	$\mathcal{V}^{\iota_j}_{\mathfrak{A},W}$.
\noindent 	Define the $H$-valued martingale increment
\[
\delta_{k}^{\iota_j}(\widetilde{M}_{\iota_j})
:=\widetilde{M}_{\iota_j}(t_{k+1}^{\iota_j})
-\widetilde{M}_{\iota_j}(t_{k}^{\iota_j}),
\]
for  \(k=0,1,\ldots,2^{\iota_j}-1\).
Since 
\[
\big(\Pro_{\iota_j}\widetilde{w}_{\iota_j}\big)(t_k^{\iota_j}) \in V,
\qquad k=1,\ldots,2^{\iota_j},
\]
and since \(\Phi:V\to \mathcal{L}_2(U,H)\), we know that
\[
\widetilde \Phi_{\iota_j,k}:=\Phi\!\left(\Pro_{\iota_j}\widetilde{w}_{\iota_j}(t_k^{\iota_j}),\omega\right)
\in \mathcal{L}_2(U,H), \qquad \omega\in\Omega.
\]
\noindent 	
Let us define\footnote{The symbol $^\ast$ denotes the adjoint.}
{\[
\widetilde{Q}_{k}^{\iota_j}
:=\widetilde \Phi^\ast _{\iota_j,k}\,\widetilde \Phi_{\iota_j,k}
.
\]
Since \(\widetilde{Q}_{k}^{\iota_j}:U\to U\)} is non-negative definite and symmetric, it admits the spectral decomposition
\begin{equation}\label{eq:spectralQ}
	\widetilde{Q}_{k}^{\iota_j}(\omega)
	=\sum_{\ell=1}^{\infty}
	\widetilde\lambda_{\ell }^{k}(\omega)\;
	\widetilde g_{\ell}^{k,\iota_j}(\omega)
	\otimes 
\widetilde	g_{\ell}^{k,\iota_j}(\omega),
	\qquad \omega\in\Omega,
\end{equation} 
where \(\{\widetilde\lambda_\ell^k(\omega)\}_{\ell\ge1}\) are the eigenvalues and
\(\{\widetilde g_{\ell}^{k,\iota_j}\}_{\ell\ge1}\subset U\) is an orthonormal basis of eigenvectors, both random variables over $\widetilde \MA$.
	Let us fix an orthonormal basis $\left\{f_{\ell}\right\}_{\ell\geq1}$ in $U$ and define an operator valued random variable $\widetilde	V_{k,{{\iota_j}}}:U\to U $, $k,\iota\in\NN$, over ${\wi{\MA}
	}$ by 
	$$
\widetilde	V_{k,{{\iota_j}}}=\sum_{\ell=1}^{\infty} \widetilde  g_{\ell}^{k,{\iota_j}} \otimes f_{\ell}  .
	$$
	In particular, we have for any $h\in U$
	$$
\widetilde	V_{k,{{\iota_j}}}h=\sum_{\ell=1}^{\infty} \la \widetilde g_{\ell}^{k,{{\iota_j}}} ,h\ra \otimes f_{\ell}.
	$$
	Then, let us define the adjoint
	$\widetilde V^{*}:U \to U$
	$$
\widetilde	V^{*}_{k,{{\iota_j}}}=\sum_{\ell=1}^{\infty} f_{\ell} \otimes \widetilde g_{\ell}^{k,{{\iota_j}}}.
	$$
	In particular, we have for any $h\in U$
	$$
\widetilde	V^{*}_{k,{{\iota_j}}}h=\sum_{m=1}^{\infty}\la  f_{\ell},h\ra  \widetilde g_{\ell}^{k,{{\iota_j}}}.
	$$
	Let us define over $\widetilde \MA$ the random operator 
	$$
\widetilde	\Lambda_{k,{{\iota_j}}}=\widetilde V_{k,{{\iota_j}}} \widetilde Q_{k}^{{\iota_j}} \widetilde V^{*}_{k,{{\iota_j}}}.
	$$
	By straightforward calculations, we can show that $\Lambda:U\to U$
	has the diagonal decomposition
	\begin{equation*}
	\widetilde	\Lambda_{k,{{\iota_j}}}=\sum_{\ell=1}^{\infty} \widetilde\lambda_{\ell}^k\,  f_{\ell} \otimes f_{\ell}. \label{8.8}
	\end{equation*}
Here, one should emphasise that $\widetilde{\Lambda}$ and $\widetilde{\lambda}_{\ell}^{k}\in\RR$ are random variables over $\widetilde{\mathfrak{A}}$, whereas $\{f_\ell \}_{\ell\geq1}$ is a deterministic basis of $U$.
Finally, let us define the discrete martingale 
\(\{\widetilde{N}_{k}^{\iota_j} : k=0,1,\ldots,2^{\iota_j}\}\) by
{\begin{equation}\label{8.9}
	\widetilde{N}_{k}^{\iota_j}
	:= \sum_{\ell=1}^{k} 
	\widetilde{V}_{\ell,\iota_j}^{*}\, \widetilde{\Phi}_{\iota_j,\ell}^\ast
	\delta_{\ell}^{\iota_j}(\widetilde{M}_{\iota_j}),
	\qquad k=0,1,\ldots,2^{\iota_j}.
\end{equation}Note that $ \widetilde{\Phi}_{\iota_j,\ell}^\ast
\delta_{\ell}^{\iota_j}(\widetilde{M}_{\iota_j})\in U$, the spectral decomposition of $\widetilde{Q}_{\ell}^{\iota_j}$ on the orthonormal family $\{f_\ell\}_{\ell\geq 1} $ in $U$ is well-posed and the covariance is obtained by}
\[
\big\langle\!\big\langle \widetilde{N}_{k}^{\iota_j}
\big\rangle\!\big\rangle
= \sum_{\ell=1}^{k}
\widetilde{V}_{\ell,\iota_j}^{*}\,
\widetilde{Q}_{\ell}^{\iota_j}\,
\widetilde{V}_{\ell,\iota_j}
= \sum_{\ell=1}^{k}
\widetilde{\Lambda}_{\ell,\iota_j},
\qquad k=0,1,\ldots,2^{\iota_j}.
\]
Consequently, the quadratic variation of the discrete martingale 
\(\widetilde{N}_{k}^{\iota_j}\) admits a diagonal decomposition, and for the
coordinate process
\[
\widetilde{N}^{\iota_j}_{k,\ell}
:= \langle \widetilde{N}_{k}^{\iota_j}, f_\ell\rangle,
\]
its quadratic variation is given by
\[
\big\langle\!\big\langle 
\widetilde{N}^{\iota_j}_{k,\ell}
\big\rangle\!\big\rangle
= \widetilde{\lambda}_{\ell}^{k}.
\]

	Now, we construct the Gaussian distributed increments and the extension $\wi{\wi{\MA}}$. 
	Let the family $\{ {\beta}^0_\ell:\ell\in\NN\}$ be mutually independent Brownian motions on a new filtered probability space ${\mathfrak{A}}^0=(\Omega^0,\CF^0,\PP^0)$ and let
us 	define the extension\footnote{For the exact definition of an extension, see  Appendix \ref{extension}.}
	\DEQSZ\label{sigmatilde}
	\wi{\wi{{\Omega}}}= \wi{\Omega} \times \Omega^{0},  \quad 
	\wi{\wi{\mathscr{F}}}=\wi{\mathscr{F} }\otimes \mathscr{F}^{0}, \quad  \wi{\wi{\mathbb{P}}}=\wi{\mathbb{P}} \otimes \mathbb{P}^{0}
	.
	\EEQSZ
Note that, since $\{\beta_\ell^{0} : \ell\in\mathbb{N}\}$ are not only random variables on $\mathfrak{A}^{0}$ but also random variables on $\widetilde{\widetilde{\mathfrak{A}}}$, we denote $\beta^{0}$ on the enlarged space by $\widetilde{\widetilde{\beta}}$. Similarly, we denote random elements that were originally defined on $\widetilde{\mathfrak{A}}$ by $\widetilde{\widetilde{\cdot}}$ when considered on the extended space. In particular, $\widetilde{w}$ defined over $\widetilde{\widetilde{\mathfrak{A}}}$, we will now write $\widetilde{\widetilde{w}}$.

Let 
	$$
\widetilde{\widetilde{\delta}}^{\iota_j}_{k}(\widetilde{\widetilde{\beta}}_{\ell})
:= 
\widetilde{\widetilde{\beta}}_{\ell}(t_{k+1}^{\iota_j})
-
\widetilde{\widetilde{\beta}}_{\ell}(t_{k}^{\iota_j}),
\qquad 
k=0,1,\ldots,2^{\iota_j},\quad j\in\mathbb{N},\; \ell\in\mathbb{N}.
	$$
	On the extension, define normally distributed random variables with variance 
	\begin{equation*}
		\widehat{\beta}_{\ell}^k =\mathds{1}_{\{\lambda_{\ell}^k=0\} }
		\delta _k^{{\iota_j}} ( \wi{\widetilde{\beta}}_\ell) +\mathds{1}_{\{\lambda_{\ell}^k>0\} }
		(\lambda_{\ell}^k)^{-\frac 12 } \delta_k^{{\iota_j}}(\wi{N}^{{\iota_j}} _{k,\ell}), \quad k=0,1,\ldots , 2^{{\iota_j}},\quad \ell \in \mathbb{N}. \label{8.11}
	\end{equation*}
	From the above process, we can construct the discrete martingale
	on the grid 
	$$\pi_{{\iota_j}}:=\lk\{t_k^{{\iota_j}}:=2^{-{{\iota_j}}}k:k=0,1,\ldots,2^{{\iota_j}} \rk\}$$
	by
	$$
	\lk\{ \widehat{\beta}_{{\iota_j}}^k :k=0,1,\ldots,2^{{\iota_j}}
	\rk\}
	$$
	on  the new probability space $\wi{\widetilde{\mathfrak{A}}}$
	with respect to the filtration $(\CF_{t_k^{{\iota_j}}})_{k=0,1,\dots 2^{{\iota_j}}}$.
	Let us define a sequence of stochastic processes $\widetilde{\widetilde{B}}_n$ 
	by
	$$
	\widetilde{\widetilde{B}}_1(0):=	0,\quad \widetilde{\widetilde{B}}_1(1):=\widehat{\beta}_0^1.
	$$
	Between the points $0$ and $1$, we interpolate the process ${\widetilde{\widetilde{B}}}_1$  linearly.
	Next, we put
	 $$	{\widetilde{\widetilde{B}}}_2(0):=	0,
	 \quad \widetilde{\widetilde{B}}_2(\frac 12 ):=\widehat{\beta}^1_1,\quad  \widetilde{\widetilde{B}}_2(1):=\widehat{\beta}_0^1.
	 $$
	 Again, between the points $0,\frac 12,1$, we  interpolate the process $\widetilde{\widetilde{B}}_2$  linearly.
	 Next, we put
	 $$	\widetilde{\widetilde{B}}_3(0):=	0, \  \widetilde{\widetilde{B}}_3(2^{-2}):=\widehat{\beta}^1_2,\ 
\ \widetilde{\widetilde{B}}_3(\frac 12 ):=\widehat{\beta}^1_1,\  \widetilde{\widetilde{B}}_3(\frac{3}{4}):=\widehat{\beta}^1_2,\ \widetilde{\widetilde{B}}_3(1):=\widehat{\beta}_0^1.
$$
	 Again, between the points $0,\frac 14,\frac 12,\frac 34,1$, we  interpolate the process $\widetilde{\widetilde{B}}_3$  linearly. 
	 Let us assume that we constructed $B_n$. Then, we can construct $B_{n+1}$ by 

	 \begin{align*}
	 	&\widetilde{\widetilde{B}}_{n+1}(0):=	0, \widetilde{\widetilde{B}}_{n+1}(2^{-(n+3)}):=\widehat{\beta}^1_{n+3},\quad 
	 	\widetilde{\widetilde{B}}_{n+1}( 2^{-(n+1)} ):=\widehat{\beta}_{n+1},\quad \\
	 	&\widetilde{\widetilde{B}}_{n+1}(2^{-n}):=\widehat{\beta}^1_{n},\quad \cdots,  \widetilde{\widetilde{B}}_{n+1}(1):=\widehat{\beta}_0^1.
	 \end{align*}
	 Again, between the points $0,2^{-(n+3 )},2^{-(n+1)},2^{-n},\cdots, 1$  we  interpolate the process $\widetilde{\widetilde{B}}_{n+1}$  linearly. 
	 This procedure can be iterated, and in this way, we obtain the
L\'evy–Ciesielski construction of Brownian motion.
Analysing the limit as \(n\to\infty\), we see from Theorem~\ref{LCCthrm}
that the sequence of stochastic processes converges 
\(\wi{\widetilde{\mathbb{P}}}\)-a.s.\ to a Brownian motion, which we denote by 
\(\wi{\widetilde{B}}^{\infty}\).

	Due to the definition of the shifted Haar projection, 
	it is straightforward to show that 
	\begin{align*} \nonumber 
		&\big( \widetilde{\widetilde{w}}_{{\iota_j}}(t),\phi\big)
		\\&\nonumber =\big( x_{{\iota_j}},\phi\big)+\int_{0}^{t}\big\la 
	A( \Pro_{{\iota_j}}(\widetilde{\widetilde{w}}_{{\iota_j}})(s) )+F_1(\Pro_{\iota_j}\widetilde{\widetilde{w}}_{\iota_j}(s))+F_2(\Pro_{\iota_j}\widetilde{\widetilde{w}}_{\iota_j}(s))
	,\phi\big\ra \, d s
		\\	& \quad +\Big(\sum_{i\in \mathbb{I}}Q^{\frac{1}{2}}{\psi_i} \sum_{n=1}^k \Sigma((\Pro_{{\iota_j}}	\widetilde{\widetilde{w}}_{{\iota_j}})(t_j^{{\iota_j}}) ) \big[\wi{\wi{B}}_{n}^{k+1}-\wi{\wi{B}}_{n}^{k}\big] ,  	\end{align*}
	for all $ t\in [0,T],\,\phi\in V$ and $\Sigma=\Sigma_1+\Sigma_2$. 
Due to the fact that 	$(\Pro_{{\iota_j}}	\widetilde{\widetilde{w}}_{{\iota_j}})(s)$ is constant for $s\in [t_j^{\iota_j}, t_{j+1}^{\iota_j})$ and  $\widetilde{\widetilde{B}}$, we can write also 
	\begin{align} \nonumber 
		&\big( \widetilde{\widetilde{w}}_{{\iota_j}}(t),\phi \big)
	\\&\nonumber =\big( x_{{\iota_j}}, \phi \big)+\int_{0}^{t}\big\la 
	A( \Pro_{{\iota_j}}\widetilde{\widetilde{w}}_{{\iota_j}}(s) )+F_1(\Pro_{\iota_j}\widetilde{\widetilde{w}}_{\iota_j}(s))+F_2(\Pro_{\iota_j}\widetilde{\widetilde{w}}_{\iota_j}(s)), \phi \big\ra 
	\,ds
	\\	&\quad  +\Big(\sum_{i\in \mathbb{I}}Q^{\frac{1}{2}}{\psi_i}  \int_0^t \Sigma((\Pro_{{\iota_j}}\widetilde{\widetilde{w}}_{{\iota_j}})(s) ) d\wi{\widetilde B}^\infty ,\phi\Big) ,  \text{ for all }t\in[0,T], \ \phi\in V, \label{a_infty}
\end{align}where $\Sigma=\Sigma_1+\Sigma_2$.

\newcommand{\wwwt}{\widetilde{\widetilde{w} }}
\newcommand{\wwi}[1]{\widetilde{\widetilde{#1}}}

Now, our aim is to show that 
$\widetilde{\widetilde{w}}^\infty $ is a solution to 
\begin{align*}
	\big( {\wwwt}^\infty (t),\phi \big) 
	&=\big( w_0,\phi \big)+\int_{0}^{t}
\big\la 	A( {\wwwt}^\infty (s))+F_1({\wwwt}^\infty(s))+F_2({\wwwt}^\infty(s))  
,\phi\big\ra ds  
	\\&\quad  +\Big(\sum_{i\in \mathbb{I}}Q^{\frac{1}{2}}{\psi_i}  \int_0^t \big(\Sigma_1({\wwwt}^\infty(s))+\Sigma_2({\wwwt}^\infty(s))\big) d\wi{\widetilde B}^\infty ,\phi\Big),
\end{align*}for all $t\in[0,T], \ \phi\in V$, which can be achieved with the help of identification of the limit functions, that is, we need to show that $A(\Pro_{\iota_j}\widetilde{\widetilde{w}}_{\iota_j}) \to A(\wi{\wi{w}}^\infty)$, $F_i(\Pro_{\iota_j}\wi{\wi{w}}_{\iota_j}) \to F_i(\wi{\wi{w}}^\infty)$ and $\Sigma_i(\Pro_{\iota_j}\wi{\wi{w}}_{\iota_j}) \to \Sigma_i(\wi{\wi{w}}^\infty)$, for $i\in\{1,2\}$ as $j\to\infty$, in the appropriate spaces  for all   $\wi{\wi{w}}_{
{\iota_j}}\in\mathcal{X}_{\wwi{\MA}}(R)$.
 Since  $\wi{w}_{{\iota_j}}=\mathcal{V}^{\iota_j}_{\wi{\MA},\wi{W}}(\wi{w}_{{\iota_j}})$ and the definition of the extension, we know that by the Assumption (\ref{iii}) and (\ref{iv}), that the family $\{\mathscr{L}(\wwi{w}_{{\iota_j}}):j\in\N\}$ is tight on $\mathbb{X}$.

We know that $\wi{w}_{{\iota_j}}\to\wi{w}^\infty$, $\wi{\PP}$-a.s. in ${L^m(0,T,V)}\cap \mathbb{D}(0,T;H)$, for $m>1$. The definition of the extension and by the construction of the Brownian motion, we know that 
 $\wwi{w}_{{\iota_j}}\to\wwi{w}^\infty$, $\wwi{\PP}$-a.s.
 
Now, our aim is to establish a strong convergence of $\{\wi{\wi{w}}_{\iota_j}\}_{j\in\N}$ in the space $\X$. For our purposes, we will use Assumption (\ref{iv}), the compact embedding of $\X_1$ into $\X$, uniform integrability, and the Vitali convergence theorem. 
The statement and its proof are presented in the following claim:
 \begin{claim}\label{uclaim}
 	Under the Assumption (\ref{iv}), the following hold:
 	\begin{enumerate}
 		\item for $r\in [1,m]$, there exists a positive constant $C$ such that 
 		\begin{align*}	
 			\sup_{j\in\N}\wi{\wi{\EE}}\Big[\big\|\wi{\wi{w}}_{{\iota_j}}\big\|_\X^r\Big]\leq C;
 		\end{align*}
 		\item for any $r\in [1,m)$, we have 
 		\begin{align*}	
 			\lim_{j\to\infty}\wi{\wi{\EE}}\Big[\big\|\wi{\wi{w}}_{{\iota_j}}-\wi{\wi{w}}^\infty\big\|_\X^r\Big]=0.
 		\end{align*}
 	\end{enumerate}
 \end{claim}
 \begin{proof}[Proof of Claim \ref{uclaim}]
 	From the previous steps, we know that the operator $\Vcal_{\wi{\wi{\MA}},\wi{\wi{W}}}^{\iota_j}$ has a fixed-point $\wi{\wi{w}}_{{\iota_j}}$, which belongs to $\Vcal_{\wi{\wi{\MA}},\wi{\wi{W}}}^{\iota_j}(\Xcal_R(\wi{\wi{\MA}}))$. Since $\{\wi{\wi{w}}_{{\iota_j}}\}_{j\in\N}\subset \Vcal_{\wi{\wi{\MA}},\wi{\wi{W}}}^{\iota_j}(\Xcal_R(\wi{\wi{\MA}}))$, we know from Assumption (\ref{iv}), that for any $r\in [1,m]$  there exists some positive constant $R$ such that 
 	\begin{align*}	
 		\wi{\wi{\EE}}\Big[\big\|\wi{\wi{w}}_{{\iota_j}}\|_\X^r\Big]\leq R,\ \text{ for every }\ j\in\N,
 	\end{align*}which completes the first part of this claim. 
 	
 	From the above expression, we conclude that $\Big\{\big\|\wi{\wi{w}}_{{\iota_j}}\big\|_\X^r\Big\}_{j\in\N}$ is uniformly integrable for any $r\in[1,m]$, with respect to the probability measure $\wi{\wi{\PP}}$. Finally, using the convergence obtained in \eqref{assapp2}, and Vitali convergence theorem, we conclude that 
 	\begin{align*}	
 		\lim_{j\to\infty}\wi{\wi{\EE}}\Big[\big\|\wi{\wi{w}}_{{\iota_j}}-\wi{\wi{w}}^\infty\big\|_\X^r\Big]=0, 
 	\end{align*}for any $r\in [1,m)$.
 \end{proof}

Therefore, we obtain the following convergences:
\begin{enumerate}
	\item $\wwi{w}_{{\iota_j}} \xrightarrow{w} \wi{\wi{w}}^\infty$ in $L^m(\wwi{\Omega};L^m(0,T;V))$; 
	\item $A^{{\iota_j}}(\wwi{w}_{{\iota_j}})\xrightarrow{w} A^\infty$ in $L^\frac m {m-1}(\wwi{\Omega};L^\frac m {m-1}(0,T;V^\ast))$;
\item 
	 $\bar F_i^{\iota_j} (\wwi{w}_{\iota_j} )  \xrightarrow{w} {F}_i^\infty$ in $L^\frac m {m-1}(\wwi{\Omega};L^\frac m {m-1}(0,T;V^\ast))$, for $i\in\{1,2\}$;
	\item $\Sigma_i^{{\iota_j}}(\wwi{w}_{{\iota_j}})\xrightarrow{w} {\Sigma_i^\infty}$ in $L^2(\wwi{\Omega}; L^2(0,T;\mathcal{L}_2(U;H)))$, and\footnote{With $\xrightarrow{w^\ast}$ we denote the weak star convergence.}
	\begin{align*}
		\int_0^{\cdot}\Sigma_i^{\iota_j}(	\wwt_{\iota_j}(s))\, dW(s) 			\xrightarrow{w^\ast}\int_0^{\cdot}\Sigma_i^\infty	(s)\, dW(s),\ \text{ in }\ L^\infty(0,T;L^2(\wwi{\Omega};H)) ,
	\end{align*}for $i\in\{1,2\}$. 
\end{enumerate}Note that $A^{\iota_j}(\wwi{w}_{{\iota_j}})=A(\Pro_{\iota_j} \wwi{w}_{{\iota_j}}), \ F_1^{\iota_j}(\wwi{w}_{{\iota_j}})=F_1(\Pro_{\iota_j}\wwi{w}_{{\iota_j}})$, $F_2^{\iota_j}(\wwi{w}_{{\iota_j}})=F_2(\Pro_{\iota_j}\wwi{w}_{{\iota_j}})$,  $\Sigma_{1}^{\iota_j}(\wwi{w}_{{\iota_j}})=\Sigma_1(\Pro_{\iota_j}\wwi{w}_{{\iota_j}})$ and $\Sigma_{2}^{\iota_j}(\wwi{w}_{{\iota_j}})=\Sigma_2(\Pro_{\iota_j}\wwi{w}_{{\iota_j}})$

Let us fix   
\begin{align*}
\wi{\wi{w}}^\infty(t)&:= w_0+\int_0^t \Big(A^\infty(s)+ F_1^\infty(s)+F_2^\infty(s)\Big)\,ds\\&\qquad +\int_0^t\big(\Sigma_1^\infty(s)+\Sigma_2^\infty(s)\big)\, d\wi{\wi{W}}(s),\ \text{for all } \ t\in[0,T], \ \wi{\wi{\PP}}\text{-a.s.}
\end{align*}

In the sequel, we establish our results in the newly constructed filtered probability space $(\wi{\wi{\Omega}},\wi{\wi{\mathcal{F}}}, \{\wi{\wi{\mathcal{F}}}_t\}_{t\geq 0},\wi{\wi{\PP}})$. For convenience we change the superscript notations, for example, we write $\{\wi{\wi{w}}_{\iota_j}\}_{j\in\N}$ and $\wi{\wi{w}}^\infty$ as $\{\wh{w}_{\iota_j}\}_{j\in\N}$ and $\wh{w}^\infty$, respectively and the probability space by $\big(\wh{\Omega},\wh{\mathcal{F}},\{\wh{\mathcal{F}}_t\}_{t\geq 0},\wh{\PP})$.

To complete the proof of Theorem $\ref{ther_main}$, it remains to justify the passage to the limit in the approximate problems. Once these limits are identified, we will conclude that $w^\infty$ is a solution to the limiting equation corresponding to $\eqref{a_infty}$, thereby finishing the proof of Theorem $\ref{ther_main}$. Note that, for the sequences $\{F_1^{\iota_j}\}_{j\in\N}$, $\{F_2^{\iota_j}\}_{j\in\N}$ and $\{\Sigma_2^{\iota_j}\}_{j\in\N}$, we provide the direct verification of limit processes in Claims \ref{F_1lim}, \ref{F_2lim} and \ref{Sigma_2lim}, respectively. For the remaining sequences $\{A^{\iota_j}\}_{j\in\N}$ and $\{\Sigma_1^{\iota_j}\}_{j\in\N}$, we use the local monotonicity method and the verification is given in Claim \ref{A_lim}. 
We begin the identification with Claim \ref{F_1lim}. Note that we provide the proof in the case of Hypothesis \ref{hyp} (H.3)$_1$, since the other case (whenever the Hypothesis \ref{hyp} (H.3)$_2$ holds, that is, $g_1$ and $g_2$ are bounded)  is comparatively easy. 
\begin{claim}\label{F_1lim}
		$F_1^\infty=F_1(\wh{w}^\infty)=F_1(\wh{w}^\infty,\wh{w}^\infty)$, $\wh{\PP}\otimes dt$-a.e.
\end{claim}
\begin{proof}
	From the Claim \ref{uclaim}, we have the following convergence 
	\begin{align*}	
		\lim_{j\to\infty}{\EE}\Big[\big\|{\wh{w}}_{\iota_j}-{\wh{w}}^\infty\big\|_\X^r\Big]=0.
	\end{align*}for $r\in[1,m)$.
	Our aim is to show that 
	\begin{align*}
		\lim_{j\to\infty}\wh{\EE}\bigg[\int_0^T  | F_1^{\iota_j}(w_{\iota_j},w_{\iota_j})(s)-F_1(w^\infty,w^\infty)(s) | _{V^\ast}^m\,ds\bigg]=0.
			\end{align*}The above convergence can be justified as follows:
			\begin{align*}
			&	\wh{\EE}\bigg[\int_0^T  | F_1^{\iota_j}(\wh{w}_{\iota_j},\wh{w}_{\iota_j})(s)-F_1(\wh{w}^\infty,\wh{w}^\infty)(s) | _{V^\ast}^m\,ds\bigg]\\&
			=
			\wh{\EE}\bigg[\int_0^T  | F_1(\Pro_{\iota_j} \wh{w}_{\iota_j},\Pro_{\iota_j} \wh{w}_{\iota_j})(s)-F_1(\wh{w}^\infty,\wh{w}^\infty)(s) | _{V^\ast}^m\,ds\bigg]
			\\&\leq 
			\wh{\EE}\bigg[\int_0^T  | F_1(\Pro_{\iota_j} \wh{w}_{\iota_j},\Pro_{\iota_j} \wh{w}_{\iota_j})(s)-F_1(\Pro_{\iota_j}\wh{w}_{\iota_j} ,\wh{w}^\infty)(s) | _{V^\ast}^m\,ds\bigg]
			\\&\quad +
			\wh{\EE}\bigg[\int_0^T  | F_1(\Pro_{\iota_j} \wh{w}_{\iota_j},\wh{w}^\infty)(s)-F_1(\wh{w}^\infty ,\wh{w}^\infty)(s) | _{V^\ast}^m\,ds\bigg]
			\\&\leq 
			\wh{\EE}\bigg[\int_0^T g_1(\Pro_{\iota_j} \wh{w}_{\iota_j}(s)) | \Pro_{\iota_j} \wh{w}_{\iota_j}(s)- \wh{w}^\infty(s) | _{V}^{m\zeta_1}\,ds\bigg]
			\\&\quad +
			\wh{\EE}\bigg[\int_0^T g_2(\wh{w}^{\infty}(s)) | \Pro_{\iota_j} \wh{w}_{\iota_j}(s)- \wh{w}^\infty(s) | _{V}^{m\zeta_2}\,ds\bigg]
			\\&\leq 
			\wh{\EE}\bigg[\int_0^T g_1(\Pro_{\iota_j} \wh{w}_{\iota_j}(s)) | \Pro_{\iota_j} \wh{w}_{\iota_j}(s)- \wh{w}_{\iota_j}(s) | _{V}^{m\zeta_1}\,ds\bigg]
			\\&\quad +
			\wh{\EE}\bigg[\int_0^T g_1(\Pro_{\iota_j} \wh{w}_{\iota_j}(s)) |  \wh{w}_{\iota_j}(s)- \wh{w}^\infty(s) | _{V}^{m\zeta_1}\,ds\bigg]
			\\&\quad +
			\wh{\EE}\bigg[\int_0^T g_2(\wh{w}^{\infty}(s)) | \Pro_{\iota_j} \wh{w}_{\iota_j}(s)-\wh{w}_{\iota_j}(s) | _{V}^{m\zeta_2}\,ds\bigg]
			\\&\quad +
			\wh{\EE}\bigg[\int_0^T g_2(\wh{w}^{\infty}(s)) | \wh{w}_{\iota_j}(s)- \wh{w}^\infty(s) | _{V}^{m\zeta_2}\,ds\bigg]
			\\&\leq  {
				\bigg\{\wh{\EE}\bigg[\int_0^T g_1^{q}(\Pro_{\iota_j} \wh{w}_{\iota_j}(s))\,ds\bigg]\bigg\}^{\frac{1}{q}}\bigg\{	\wh{\EE}\bigg[\int_0^T | \Pro_{\iota_j} \wh{w}_{\iota_j}(s)- \wh{w}_{\iota_j}(s) | _{V}^{\frac{m\zeta_1q}{q-1} }\,ds\bigg]\bigg\}^{\frac{q-1}{q}} }
			\\&\quad   {+
				\bigg\{\wh{\EE}\bigg[\int_0^T g_1^{q}(\Pro_{\iota_j} \wh{w}_{\iota_j}(s))\,ds\bigg]\bigg\}^{\frac{1}{q} }\bigg\{	\wh{\EE}\bigg[\int_0^T | \wh{w}_{\iota_j}(s)- \wh{w}^\infty(s) | _{V}^{\frac{m\zeta_1q}{q-1}}\,ds\bigg]\bigg\}^{\frac{q-1}{q} } }
			\\&\quad   {+
				\bigg\{\wh{\EE}\bigg[\int_0^T g_2^{q}( \wh{w}^\infty(s))\,ds\bigg]\bigg\}^{\frac{1}{q}}\bigg\{	\wh{\EE}\bigg[\int_0^T | \Pro_{\iota_j} \wh{w}_{\iota_j}(s)- \wh{w}_{\iota_j}(s) | _{V}^{\frac{m\zeta_2q}{q-1}}\,ds\bigg]\bigg\}^{\frac{q-1}{q} } }
			\\&\quad  {+
				\bigg\{\wh{\EE}\bigg[\int_0^T g_2^{q}( \wh{w}^\infty(s))\,ds\bigg]\bigg\}^{\frac{1}{q} }\bigg\{	\wh{\EE}\bigg[\int_0^T | \wh{w}_{\iota_j}(s)- \wh{w}^\infty(s) | _{V}^{\frac{m\zeta_2q}{q-1}}\,ds\bigg]\bigg\}^{\frac{q-1}{q} }}
			\\&\leq   {
				R_{g_1}^{\frac{1}{q}}\bigg\{	\wh{\EE}\bigg[\int_0^T | \Pro_{\iota_j} \wh{w}_{\iota_j}(s)- \wh{w}_{\iota_j}(s) | _{V}^{\frac{m\zeta_1q}{q-1}}\,ds\bigg]\bigg\}^{\frac{q-1}{q} }}
			\\&\quad   {+
				R_{g_1}^{\frac{1}{q}}\bigg\{	\wh{\EE}\bigg[\int_0^T | \wh{w}_{\iota_j}(s)- \wh{w}^\infty(s) | _{V}^{\frac{m\zeta_1q}{q-1}}\,ds\bigg]\bigg\}^{\frac{q-1}{q}  }}
			\\&\quad  {+
				R_{g_2}^{\frac{1}{q}}\bigg\{	\wh{\EE}\bigg[\int_0^T | \Pro_{\iota_j} \wh{w}_{\iota_j}(s)- \wh{w}_{\iota_j}(s) | _{V}^{\frac{m\zeta_2q}{q-1}}\,ds\bigg]\bigg\}^{\frac{q-1}{q}  }}
			\\&\quad   {+
				R_{g_2}^{\frac{1}{q} }\bigg\{	\wh{\EE}\bigg[\int_0^T | \wh{w}_{\iota_j}(s)- \wh{w}^\infty(s) | _{V}^{\frac{m\zeta_2q}{q-1}}\,ds\bigg]\bigg\}^{\frac{q-1}{q}  }}\\&    { \to 0, \text{ as } j\to \infty,}
		\end{align*}for $\zeta_1,\zeta_2\in (0,1)$, where we have used Hypothesis \ref{hyp} (H.3)$_1$ (see (a) and (b)) and the fact that {$q\ >\ \max\!\big\{\tfrac{1}{1-\zeta_1},\,\tfrac{1}{1-\zeta_2}\big\}$} \big(which implies $\frac{m\zeta_i q}{q-1}<m$, for $i=1,2$\big), \eqref{Condition}, Claim \ref{uclaim}, and Remark \ref{projection}.
\end{proof}

\begin{claim}\label{F_2lim}
	$F_2^\infty=F_2(\wh{w}^\infty)$, $\wh{\PP}\otimes dt$-a.e.
\end{claim}
\begin{proof}
	  Our goal is to show the following: 
\begin{align*}
	\lim_{j\to\infty} \wh{\EE}\bigg[\int_0^T  | F_2^{\iota_j}(\wh{w}_{\iota_j}(s))-F_2(\wh{w}^\infty(s)) | _{V^\ast}^m\,ds\bigg]=0. 
\end{align*}
The above convergence can be justified as follows (along with a further subsequence):
\begin{align*}
&\wh{\EE}\bigg[\int_0^T  | F_2^{\iota_j}(\wh{w}_{\iota_j}(s))-F_2(\wh{w}^\infty(s)) | _{V^\ast}^m\,ds\bigg]\\&=
	\wh{\EE}\bigg[\int_0^T  | F_2(\Pro_{\iota_j}\wh{w}_{\iota_j}(s))-F_2(\wh{w}^\infty(s)) | _{V^\ast}^m\,ds\bigg]
\\&\leq  	\wh{\EE}\bigg[\int_0^T  | F_2(\Pro_{\iota_j}\wh{w}_{\iota_j}(s))-F_2(\wh{w}_{\iota_j}(s)) | _{V^\ast}^m\,ds\bigg]\\&\quad + \wh{\EE}\bigg[\int_0^T  | F_2(\wh{w}_{\iota_j}(s)) -F_2(\wh{w}^\infty(s)) | _{V^\ast}^m\,ds\bigg]
	\\& \leq 
	C \wh{\EE}\bigg[\int_0^T  | \Pro_{\iota_j}\wh{w}_{\iota_j}(s)-\wh{w}_{\iota_j}(s) | _V^{m\delta}\,ds\bigg]
	\\&\quad +C \wh{\EE}\bigg[\int_0^T  | \wh{w}_{\iota_j}(s) -\wh{w}^\infty(s) | _V^{m\delta}\,ds\bigg]
	\\& 
	\to 0, \text{ as } j\to \infty,
\end{align*}
for $\delta\in(0,1)$, where we have used the assumption on $F_2$ (see Hypothesis \ref{hyp} (H.3) \eqref{F_2}), Claim \ref{uclaim} and Remark \ref{projection}. We know that the weak limit is unique and hence $F_2^\infty=F_2(\wh{w}^\infty)$, $\wh{\PP}\otimes dt$-a.e.
\end{proof}

\begin{claim}\label{Sigma_2lim}
	$\Sigma_2^\infty=\Sigma_2(\wh{w}^\infty)$, $\widehat{\PP}\otimes dt$-a.e.
\end{claim}
\begin{proof}  Our goal is to show the following: 
	\begin{align*}
		\lim_{j\to\infty} \wh{\EE}\bigg[\int_0^T  | \Sigma_2^{\iota_j}(\wh{w}_{\iota_j}(s))-\Sigma_2(\wh{w}^\infty(s)) | _{\mathcal{L}_2(U,H)}^2\,ds\bigg]=0. 
	\end{align*}The above convergence can be justified as follows (along with a further subsequence):
	\begin{align*}
		&\wh{\EE}\bigg[\int_0^T  | \Sigma_2^{\iota_j}(\wh{w}_{\iota_j}(s))-\Sigma_2(\wh{w}^\infty(s)) | _{\mathcal{L}_2(U,H)}^2\,ds\bigg]\\&=
		\wh{\EE}\bigg[\int_0^T  | \Sigma_2(\Pro_{\iota_j}\wh{w}_{\iota_j}(s))-\Sigma_2(\wh{w}^\infty(s)) | _{\mathcal{L}_2(U,H)}^2\,ds\bigg]
		\\&\leq  	\wh{\EE}\bigg[\int_0^T  | \Sigma_2(\Pro_{\iota_j}\wh{w}_{\iota_j}(s))-\Sigma_2(\wh{w}_{\iota_j}(s)) | _{\mathcal{L}_2(U,H)}^2\,ds\bigg]\\&\quad + \wh{\EE}\bigg[\int_0^T  | \Sigma_2(\wh{w}_{\iota_j}(s)) -\Sigma_2(\wh{w}^\infty(s)) | _{\mathcal{L}_2(U,H)}^2\,ds\bigg]
		\\& \leq 
		C \wh{\EE}\bigg[\int_0^T  | \Pro_{\iota_j}\wh{w}_{\iota_j}(s)-\wh{w}_{\iota_j}(s) | _V^{2 \kappa}\,ds\bigg]
		\\&\quad +C \wh{\EE}\bigg[\int_0^T | \wh{w}_{\iota_j}(s) -\wh{w}^\infty(s) | _V^{2\kappa}\,ds\bigg]
		\\&
		\to 0, \text{ as } j\to \infty,
	\end{align*}
for $\kappa\in(0,1)$, where we have used the assumption on $\Sigma_2$ (see Hypothesis \ref{hyp} (H.2)), Claim \ref{uclaim} and Remark \ref{projection}. We know that the weak limit is unique and hence $\Sigma_2^\infty=\Sigma_2(\wh{w}^\infty)$, $\wh{\PP}\otimes dt$-a.e.
\end{proof}

Now, it only remains to identify the limit processes $A^\infty$ and $\Sigma_1^\infty$, which is established in:
\begin{claim}\label{A_lim}
	$A^{\infty}=A(\wh{w}^\infty)$,  and  $ \Sigma_1^\infty=\Sigma_1(\wh{w}^{\infty}), \ \wh{\PP}\otimes dt$-a.e.
\end{claim}
\begin{proof}[Proof of Claim \ref{A_lim}]

For our purposes, we will exploit the local monotonicity arguments (for more details, we refer to \cite[Chapter 5]{weiroeckner}, cf. \cite{AKMTM4}).

Let $\varphi$ be a $V$-valued measurable element in $L^m(\wh{\Omega};L^m(0,T;V))$ and let $\tau_\varphi:\wh{\Omega} \to  [0,T]$ be a stopping time such that 
\begin{align*}
	C_\varphi:= \sup_{\Omega}\int_0^{\tau_\varphi}\big(f(s)+\rho(\varphi(s))\big)\,ds<\infty, \quad \wh{\PP}\text{-a.s.}
\end{align*}where $\rho$ can be chosen as in Hypothesis \ref{hyp} (H.1). Note that the choice of $\varphi$ depends on $\wh{w}^\infty$, since we need to replace $\varphi$ by $\wh{w}^\infty$ in the latter case.

Applying the It\^o formula to the process $ | \wh{w}_{\iota_j} | _H^2$, and then taking expectation on both sides, we obtain 
\begin{align}\label{436}\nonumber
	&	  {\wh{\EE}}\Big[e^{-\int_0^{t\wedge \tau_\varphi}(f(s)+\rho(\varphi(s)))\,ds} | {\wh{w}}_{{\iota_j}}(t\wedge \tau_{\varphi}) | _H^2\Big]-	  {\wh{\EE}}\Big[ | x_{{\iota_j}} | _H^2\Big]
	\\
	& \nonumber=  {\wh{\EE}} \bigg[\int_0^{t\wedge \tau_{\varphi}}e^{-\int_0^{s}(f(r)+\rho(\varphi(r)))\, dr} \bigg(2\la A^{\iota_j}( {\wh{w}}_{{\iota_j}}(s)), {\wh{w}}_{{\iota_j}}(s)\ra+2\la F_1^{\iota_j}(\wh{w}_{\iota_j}(s)),\wh{w}_{\iota_j}(s)\ra\\&\nonumber\qquad + 2\la F_2^{\iota_j}(\wh{w}_{\iota_j}(s)),\wh{w}_{\iota_j}(s)\ra
	 + | \Sigma_1^{{\iota_j}}(  {\wh{w}}_{{\iota_j}}(s)) | _{\mathcal{L}_2(U,H)}^2+ | \Sigma_2^{{\iota_j}}(  {\wh{w}}_{{\iota_j}}(s)) | _{\mathcal{L}_2(U,H)}^2\\&\nonumber\qquad-\big(f(s)+\rho(\varphi(s))\big) | \wh{w}_{\iota_j}(s) | _H^2 \bigg)\,ds\bigg]
	\\& \nonumber=  {\wh{\EE}} \bigg[\int_0^{t\wedge \tau_{\varphi}}e^{-\int_0^{s}(f(r)+\rho(\varphi(r)))\, dr}\\&\qquad \times \nonumber \bigg(2\la A^{\iota_j}( {\wh{w}}_{{\iota_j}}(s))-A(\varphi(s))+A(\varphi(s)), {\wh{w}}_{{\iota_j}}(s)-\varphi(s)+\varphi(s)\ra
	\\&\qquad \ \nonumber 
	+ | \Sigma_1^{{\iota_j}}(  {\wh{w}}_{{\iota_j}}(s))-\Sigma_1(\varphi(s))+\Sigma_1(\varphi(s)) | _{\mathcal{L}_2(U,H)}^2+ | \Sigma_2^{{\iota_j}}(  {\wh{w}}_{{\iota_j}}(s)) | _{\mathcal{L}_2(U,H)}^2
		\\&\qquad \ \nonumber+2\la F_1^{\iota_j}(\wh{w}_{\iota_j}(s)),\wh{w}_{\iota_j}(s)\ra
+ 2\la F_2^{\iota_j}(\wh{w}_{\iota_j}(s)),\wh{w}_{\iota_j}(s)\ra
	\\&\qquad \ \nonumber 
	-\big(f(s)+\rho(\varphi(s))\big) | \wh{w}_{\iota_j}(s)-\varphi(s)+\varphi(s) | _H^2 \bigg)\,ds\bigg]
	\\&\  \nonumber\leq 
		{\wh{\EE}}
	\bigg[\int_0^{t\wedge \tau_{\varphi}}e^{-\int_0^{s}(f(r)+\rho(\varphi(r)))\, dr}\bigg(2\la A^{\iota_j}( {\wh{w}}_{{\iota_j}}(s))-A(\varphi(s)), {\wh{w}}_{{\iota_j}}(s)-\varphi(s)\ra
	\\&\ \nonumber\qquad + | \Sigma_1^{{\iota_j}}(  {\wh{w}}_{{\iota_j}}(s))-\Sigma_1(\varphi(s)) | _{\mathcal{L}_2(U,H)}^2-\big(f(s)+\rho(\varphi(s))\big) | \wh{w}_{\iota_j}(s)-\varphi(s) | _H^2  \bigg)\,ds \bigg]
	\\&  \nonumber\quad + 
	{\wh{\EE}}
	\bigg[\int_0^{t\wedge \tau_{\varphi}} e^{-\int_0^{s}(f(r)+\rho(\varphi(r)))\, dr}\bigg(2\la A^{{\iota_j}}( {\wh{w}}_{{\iota_j}}(s))-A(\varphi(s)),\varphi(s)\ra +2\la A	(\varphi(s)),\wh{w}_{\iota_j}(s)\ra\\&\nonumber\qquad  - | \Sigma_1(\varphi(s)) | _{\mathcal{L}_2(U,H)}^2
	+2\la \Sigma_1^{{\iota_j}}(  {\wh{w}}_{{\iota_j}}(s)), \Sigma_1(\varphi(s))\ra_{\mathcal{L}_2(U,H)}
	+2\la  F_1^{\iota_j}(\wh{w}_{\iota_j}(s)),\wh{w}_{\iota_j}(s)\ra  \\& \nonumber\qquad+2\la  F_2^{\iota_j}(\wh{w}_{\iota_j}(s)),\wh{w}_{\iota_j}(s)\ra + | \Sigma_2^{\iota_j}(\wh{w}_{\iota_j}(s)) | _{\mathcal{L}_2(U,H)}^2 -2\big(f(s)+\rho(\varphi(s))\big)\la \wh{w}_{\iota_j}(s),\varphi(s)\ra_H \\& \qquad +\big(f(s)+\rho(\varphi(s))\big) | \varphi(s) | _H^2\bigg)\,ds \bigg], 
\end{align}where we have used the following (for clarity we omit the time variable)
\begin{align*}
&\hspace{-3mm}	\la A^{\iota_j}( {\wh{w}}_{{\iota_j}})-A(\varphi)+A(\varphi), {\wh{w}}_{{\iota_j}}-\varphi+\varphi\ra \\&= 
\la A^{\iota_j}( {\wh{w}}_{{\iota_j}})-A(\varphi), {\wh{w}}_{{\iota_j}}-\varphi\ra+\la A^{\iota_j}( {\wh{w}}_{{\iota_j}})-A(\varphi),\varphi\ra+ \la A(\varphi),\wh{w}_{\iota_j}\ra, 
\\
&\hspace{-3mm}	 | \Sigma_1^{{\iota_j}}(  {\wh{w}}_{{\iota_j}})-\Sigma_1(\varphi)+\Sigma_1(\varphi) | _{\mathcal{L}_2(U,H)}^2
 \\& = 
\la \Sigma_1^{{\iota_j}}(  {\wh{w}}_{{\iota_j}})-\Sigma_1(\varphi), \Sigma_1^{{\iota_j}}(  {\wh{w}}_{{\iota_j}})-\Sigma_1(\varphi)\ra_{\mathcal{L}_2(U,H)} + \la \Sigma_1^{{\iota_j}}(  {\wh{w}}_{{\iota_j}})-\Sigma_1(\varphi), \Sigma_1(\varphi)\ra_{\mathcal{L}_2(U,H)} 
\\&\quad + \la \Sigma_1(\varphi), \Sigma_1^{{\iota_j}}(  {\wh{w}}_{{\iota_j}})-\Sigma_1(\varphi)\ra_{\mathcal{L}_2(U,H)} +\la \Sigma_1(\varphi), \Sigma_1(\varphi)\ra_{\mathcal{L}_2(U,H)} \\&
= | \Sigma_1^{{\iota_j}}(  {\wh{w}}_{{\iota_j}})-\Sigma_1(\varphi) | _{\mathcal{L}_2(U,H)}^2+2\la  \Sigma_1^{{\iota_j}}(  {\wh{w}}_{{\iota_j}}),\Sigma_1(\varphi)\ra_{\mathcal{L}_2(U,H)} - | \Sigma_1(\varphi) | _ {\mathcal{L}_2(U,H)}^2, 
\end{align*}and 
\begin{align*}
		-&\big(f+\rho(\varphi)\big) | \wh{w}_{\iota_j}-\varphi+\varphi | _H^2 \\&
		=	-	\big(f+\rho(\varphi)\big)\la \wh{w}_{\iota_j}-\varphi+\varphi,\wh{w}_{\iota_j}-\varphi+\varphi\ra_H
		\\&=-	\big(f+\rho(\varphi)\big) \Big\{\la \wh{w}_{\iota_j}-\varphi,\wh{w}_{\iota_j}-\varphi\ra_H+2\la \wh{w}_{\iota_j}-\varphi,\varphi\ra_H +\la\varphi,\varphi\ra_H
		 \Big\}
		 \\&=-	\big(f+\rho(\varphi)\big) \Big\{ | \wh{w}_{\iota_j}-\varphi | _H^2+2\la \wh{w}_{\iota_j},\varphi\ra_H- | \varphi | _H^2
		 \Big\}.
\end{align*}
Going back to \eqref{436} and passing $j\to\infty$, we obtain for any $\psi \in L^\infty(0,T;\R^+_0)$ 
\begin{align}\label{MM01}\nonumber
&	\wh{\EE}\bigg[\int_0^T \psi (t)\bigg( e^{-\int_0^{t\wedge \tau_\varphi}(f(s)+\rho(\varphi(s)))\,ds} | {\wh{w}}^\infty(t\wedge \tau_{\varphi}) | _H^2- | \wh{w}_0 | _H^2\bigg)\,dt \bigg]
	\\&\nonumber
	\leq
	\liminf_{j\to\infty}	\wh{\EE}\bigg[\int_0^T \psi (t)\bigg( e^{-\int_0^{t\wedge \tau_\varphi}(f(s)+\rho(\varphi(s)))\,ds} | \wh{w}_{\iota_j}(t\wedge \tau_{\varphi}) | _H^2- | x_{\iota_j} | _H^2\bigg)\,dt \bigg]
		\\& \nonumber\leq 
	{\wh{\EE}}
	\bigg[\int_0^T\psi(t)\bigg( \int_0^{t\wedge \tau_{\varphi}} e^{-\int_0^{s}(f(r)+\rho(\varphi(r)))\, dr}\bigg(2\la A^{\infty} (s)-A(\varphi(s)),\varphi(s)\ra \\&\nonumber\qquad +2\la A	(\varphi(s)),{\wh{w}}^\infty(s)\ra - | \Sigma_1(\varphi(s)) | _{\mathcal{L}_2(U,H)}^2
	+2\la \Sigma_1^{\infty}(s), \Sigma_1(\varphi(s))\ra_{\mathcal{L}_2(U,H)}
	\\&\nonumber\qquad +2\la  F_1^{\infty}(s),\wh{w}^\infty(s)\ra  +2\la  F_2^{\infty}(s),\wh{w}^\infty(s)\ra + | \Sigma_2^{\infty}(s) | _{\mathcal{L}_2(U,H)}^2
	\\&\qquad -2\big(f(s)+\rho(\varphi(s))\big)\la \wh{w}^\infty(s),\varphi(s)\ra
	_H 
	+\big(f(s)+\rho(\varphi(s))\big) | \varphi(s) | _H^2\bigg)\,ds\bigg)\,dt  \bigg],
\end{align}where we have used the Hypothesis \ref{hyp} (H.1) (ii), Claims \ref{F_1lim}, \ref{F_2lim} and \ref{Sigma_2lim} and the convergence of the sequence $\{\wh{w}_{\iota_j}\}_{j\in\N}$. \\
Applying the It\^o formula to the process $ | \wh{w}^\infty | _H^2$ yield
\begin{align}\label{MM02}\nonumber
&	\wh{\EE}\bigg[e^{-\int_0^{t\wedge \tau_\varphi}(f(s)+\rho(\varphi(s)))\,ds}  | \wh{w}^\infty(t) | _H^2
	\bigg]-\wh{\EE}\big[ | \wh{w}_0 | _H^2\big]
	\\& \nonumber
	= \wh{\EE}\bigg[\int_0^{t\wedge \tau_\varphi}e^{-\int_0^{s}(f(r)+\rho(\varphi(r)))\, dr}\Big(2\la A^{\infty}(s),\wh{w}^{\infty}(s)\ra +2\la F_1^{\infty}(s),\wh{w}^{\infty}(s)\ra\\&\nonumber\qquad +2\la F_2^{\infty}(s),\wh{w}^{\infty}(s)\ra+  | \Sigma_1^\infty(s) | _{\mathcal{L}_2(U,H)}^2+ | \Sigma_2^\infty(s) | _{\mathcal{L}_2(U,H)}^2\\&\qquad -\big(f(s)+\rho(\varphi(s))\big) | \wh{w}^\infty(s) | _H^2\Big)\,ds \bigg].
\end{align}Substituting \eqref{MM01} in \eqref{MM02} and a rearrangement of terms leads to
\begin{align}\label{MM03}\nonumber
&	\wh{\EE}\bigg[\int_0^T \psi(t)\bigg(\int_0^{t\wedge \tau_\varphi}e^{-\int_0^{s}(f(r)+\rho(\varphi(r)))\, dr}\Big(2\la A^{\infty}(s),\wh{w}^{\infty}(s)\ra +2\la F_1^{\infty}(s),\wh{w}^{\infty}(s)\ra\\&\nonumber\qquad +2\la F_2^{\infty}(s),\wh{w}^{\infty}(s)\ra+  | \Sigma_1^\infty(s) | _{\mathcal{L}_2(U,H)}^2+ | \Sigma_2^\infty(s) | _{\mathcal{L}_2(U,H)}^2\\&\qquad\nonumber -\big(f(s)+\rho(\varphi(s))\big) | \wh{w}^\infty(s) | _H^2\Big)\,ds\bigg)\,dt \bigg] 
	\\& \nonumber\leq 
	{\wh{\EE}}
	\bigg[\int_0^T\psi(t)\bigg( \int_0^{t\wedge \tau_{\varphi}} e^{-\int_0^{s}(f(r)+\rho(\varphi(r)))\, dr}\bigg(2\la A^{\infty} (s)-A(\varphi(s)),\varphi(s)\ra \\&\nonumber\qquad +2\la A	(\varphi(s)),{\wh{w}}^\infty(s)\ra - | \Sigma_1(\varphi(s)) | _{\mathcal{L}_2(U,H)}^2
	+2\la \Sigma_1^{\infty}(s), \Sigma_1(\varphi(s))\ra_{\mathcal{L}_2(U,H)}
	\\&\nonumber\qquad +2\la  F_1^{\infty}(s),\wh{w}^\infty(s)\ra  +2\la  F_2^{\infty}(s),\wh{w}^\infty(s)\ra + | \Sigma_2^{\infty}(s) | _{\mathcal{L}_2(U,H)}^2
	\\&\nonumber\qquad -2\big(f(s)+\rho(\varphi(s))\big)\la \wh{w}^\infty(s),\varphi(s)\ra 
	+\big(f(s)+\rho(\varphi(s))\big) | \varphi(s) | _H^2\bigg)\,ds\bigg)\,dt  \bigg]
\\&	\nonumber\wh{\EE}\bigg[\int_0^T \psi(t)\bigg(\int_0^{t\wedge \tau_\varphi}e^{-\int_0^{s}(f(r)+\rho(\varphi(r)))\, dr}\Big(2\la A^{\infty}(s),\wh{w}^{\infty}(s)\ra \\&\quad \nonumber+  | \Sigma_1^\infty(s)-\Sigma_1(\varphi(s)) | _{\mathcal{L}_2(U,H)}^2 -\big(f(s)+\rho(\varphi(s))\big) | \wh{w}^\infty(s)-\varphi(s) | _H^2\Big)\,ds\bigg)\,dt \bigg] 
	\\& \nonumber\leq 
	{\wh{\EE}}
	\bigg[\int_0^T\psi(t)\bigg( \int_0^{t\wedge \tau_{\varphi}} e^{-\int_0^{s}(f(r)+\rho(\varphi(r)))\, dr}\bigg(2\la A^{\infty} (s)-A(\varphi(s)),\varphi(s)\ra \\&\qquad +2\la A	(\varphi(s)),{\wh{w}}^\infty(s)\ra 
	\bigg)\,ds\bigg)\,dt  \bigg]. 
\end{align}From \eqref{MM03}, we obtain
\begin{align*}
&\nonumber 
\wh{\EE}\bigg[\int_0^T \psi(t)\bigg(\int_0^{t\wedge \tau_\varphi}e^{-\int_0^s(f(r)+\rho(\varphi(r)))\, dr}\Big(2\la A^\infty(s)-A(\varphi(s)),\wh{w}^\infty(s)-\varphi(s)\ra\\&\nonumber\qquad +  | \Sigma_1(\varphi(s))-\Sigma_1^\infty(s) | _{\mathcal{L}_2(U,H)}^2-\big(f(s)+\rho(\varphi(s))\big) | \wh{w}^\infty(s)-\varphi(s) | _H^2 \Big)\,ds\bigg)\,dt\bigg]\\&\leq 0. 
\end{align*}Taking $\varphi=\wh{w}^\infty$, and for $M>0$, 
\begin{align*}
	\tau_{\wh{w}^\infty}^M:=\tau_\varphi=\inf\bigg\{t\geq 0: \int_0^t\big(f(s)+\rho(\wh{w}^\infty(s))\big)\,ds >M\bigg\}\wedge T, 
\end{align*}and then passing $M\to\infty$, we obtain that 
\begin{align*}
	&\wh{\EE}\bigg[\int_0^T \psi(t)\bigg(\int_0^{t\wedge \tau_{\wh{w}^\infty}^M}e^{-\int_0^s(f(r)+\rho(\wh{w}^\infty(r)))\, dr}  | \Sigma_1(\wh{w}^\infty(s))-\Sigma_1^\infty(s) | _{\mathcal{L}_2(U,H)}^2\,ds\bigg)\,dt\bigg]\\&\leq 0,
\end{align*}and hence $\Sigma_1^\infty=\Sigma_1(\wh{w}^\infty)$, using the uniqueness of weak limit. Next, we choose $\varphi=\wh{w}^\infty -\ep \phi v$, for $\phi\in L^\infty(\Omega;L^\infty(0,T;\R)$, $v\in V$ and $\tau_{\wh{w}^\infty}^M$, for $M>0$. Then we divide the resultant by $\ep$ and passing $\ep\to0$, followed by the Hypothesis \ref{hyp} (H.1) (i), yield
\begin{align}\label{MM05}
	\nonumber
	&\wh{\EE}\bigg[\int_0^T \psi(t)\bigg(\int_0^{t\wedge \tau_{\wh{w}^\infty}^M}e^{-\int_0^s(f(r)+\rho(\wh{w}^\infty(r)))\, dr} \phi(s)\la A^\infty(s)-A(\wh{w}^\infty(s)),v(s)\ra \,ds\bigg)\,dt\bigg]\leq 0.
\end{align}Note that, here we are allowed to interchange the limit and the integral due to the Hypothesis \ref{hyp} and the definition of the stopping time $\tau_{\wh{w}^\infty}^M$. Due to the arbitrary choice of $\psi$ and $\phi$, we conclude that $A^\infty=A(\wh{w}^\infty)$ on the interval $[0,\tau_{\wh{w}^\infty}^M]$. Passing $M\to\infty$, we conclude that $A^\infty=A(\wh{w}^\infty)$.
\end{proof}
\end{step}
Hence, the proof of Theorem \ref{ther_main} is complete.
\end{proof}

	\begin{appendix}
	\renewcommand{\thesection}{\Alph{section}}
	\numberwithin{equation}{section}

\section{Extension of a probability  space}\label{extension}

To construct the Wiener process, it is necessary to introduce additional random variables on an auxiliary probability space and then extend the original probability space by means of this auxiliary space.
In what follows, we make precise the notion of an extension of a probability space (compare with \cite{Ikeda}).

\begin{definition}
	\label{Definition7.1.} We say a probability space $(\widetilde{\Omega}, \widetilde{\mathscr{F}}, \widetilde{\mathbb{P}} )$ with a filtration $\{\widetilde{\mathscr{F}}_{t}\}_{t\ge 0}$ is an extension of a probability space $(\Omega, \mathscr{F}, \mathbb{P})$ with a  filtration $\{\mathscr{F}_{t}\}$,  if there exists a mapping $\pi: \widetilde{\Omega} \longrightarrow \Omega$ which is $\widetilde{\mathscr{F}} / \mathscr{F}$-measurable such that\\
	(i) $\widetilde{\mathscr{F}}_{t} \supset \pi^{-1}\left(\mathscr{F}_{t}\right)$,\\
	(ii) $\mathbb{P}=\pi(\widetilde{\mathbb{P}})(:=\widetilde{\mathbb{P}} \circ \pi^{-1}), $ and\\
	(iii) for every $X(\omega) \in \mathscr{L}_{\infty}(\Omega, \mathscr{F}, \mathbb{P})$
	$$
	\widetilde{\EE}\left(\widetilde{X}(\widetilde{\omega}) \mid \widetilde{\mathscr{F}}_{t}\right)=\EE\left(X \mid \mathscr{F}_{t}\right)(\pi \widetilde{\omega}), \quad \widetilde{\PP} \text {-a.s., }
	$$
	where we set $\widetilde{X}(\widetilde{\omega})=X(\pi \widetilde{\omega})$ for $\widetilde{\omega} \in \widetilde{\Omega}$.
\end{definition}
\begin{definition}
	\label{Definition_extension}  Let $(\Omega, \mathscr{F}, \mathbb{P})$ be a probability space with  filtration $\{\mathscr{F}_{t}\}_{t\ge 0}$. Let $\left(\Omega^{\prime}, \mathscr{F}^{\prime}, \PP^{\prime}\right)$ be another probability space and set
	$$
	\widetilde{\Omega}=\Omega \times \Omega^{\prime}, \quad \widetilde{\mathscr{F}}=\mathscr{F} \times \mathscr{F}^{\prime}, \quad \widetilde{\mathbb{P}}=\mathbb{P} \times \mathbb{P}^{\prime},
	$$
	and
	$$
	\pi \widetilde{\omega}=\omega \quad \text { for } \quad \widetilde{\omega}=\left(\omega, \omega^{\prime}\right) \in \widetilde{\Omega}.
	$$
	If $\{\widetilde{\mathscr{F}}_{t}\}_{t\ge 0}$ is a filtration  on $(\widetilde{\Omega}, \widetilde{\mathscr{F}}, \widetilde{\mathbb{P}})$ such that $\mathscr{F}_{t} \times \mathscr{F}^{\prime} \supset \widetilde{\mathscr{F}}_{t} \supset \mathscr{F}_{t} \times\left\{\Omega^{\prime}, \emptyset\right\}$, then $(\widetilde{\Omega}, \widetilde{\mathscr{F}}, \widetilde{\mathbb{P}})$ with $\{\widetilde{\mathscr{F}}_{t}\}_{t\ge 0}$ is called a standard extension of $(\Omega, \mathscr{F}, \mathbb{P})$ with filtration $\{\mathscr{F}_t\}_{t\ge 0}$.
\end{definition}

{
\begin{remark}\label{Remcts}
	In this remark, we are justifying the continuity of the mapping $\mathscr{V}_\iota$ mentioned in the verification (\ref{02}) of page 30.
		\begin{enumerate}
			\item Instead of using a single deterministic map on fixed canonical space, the manuscript uses a standard extension $\mathfrak{A}_\mu=\mathfrak{A}_0\times \mathfrak{A}_1$ with an independent Wiener process and applies the measurable solution operator constructed from the shifted Haar projection $\mathrm{Proj}_{n_\iota}$ and $\mathcal{V}_{\mathfrak{A},W}^\iota$. The construction is carried out on the path space once $\mu$, or the law of the input coordinate process, is provided. This is because the definition of standard extensions provided in Appendix A provides a canonical, measurable method to add the Wiener coordinate while maintaining independence.
			
			\item The construction ensures that the law is invariant. Once the fixed point has been constructed on one space, the subsequent rules correspond when the same stepwise construction is applied to another space. When moving between $\mathfrak{A}$ and $\mathfrak{A}^\ast$ and demonstrating $\mathrm{Law}(w_{\iota_j,\infty})=\mathrm{Law}(w_{\iota_j,\infty}^\ast)$, this is employed explicitly. Therefore, $\mathscr{V}_{\iota}(\mu)$ depends entirely on $\mu$.
			
			\item	\emph{Continuity in Prokhorov metric}: We have 
			\begin{enumerate}
				\item 
				A measurable solution map produced on the product space using $\mathrm{Proj}_{n_\iota}$ and $\mathcal{V}_{\mathfrak{A},W}^{\iota}$; 
				\item Due to Assumption (\ref{iii}), there is uniform continuity on the bounded subset $\mathcal{X}_{\mathfrak{A}_\mu}(R)$. 
			\end{enumerate}
				\end{enumerate}
			As stated in (\ref{02}), continuity of $\mathscr{V}_{\iota}$ on laws in the Prokhorov metric results from meeting the requirements to express continuity from paths to laws via the continuous mapping principle. 
			\end{remark}
}
\section{$L^p$-spaces and the Haar system}\label{app:haar-system}\setcounter{equation}{0} 

In this subsection, we assume that $(Y,|\cdot|)$ is a separable Banach space. We assume for simplicity that $T=1$.
We recall some facts about the approximation properties of the Haar system and its shifted version. For the proofs and other details, we are referring to \cite[Appendix B]{FKEHMH}.

\subsection{Definition of the shifted Haar projection}

For each $n \in \mathbb{N}$, let $\pi_n = \{ s^n_0 = 0 < s^n_1 < \cdots < s^n_{2^n} = 1\}$ denote the dyadic partition of $[0,1]$ given by
\[
s^n_j = j 2^{-n}, \quad j = 0, \dots, 2^n.
\]
Each interval $(s^n_{j-1}, s^n_j]$ is called a \textit{dyadic interval}.

The classical (non-shifted) Haar projection $\mathfrak{h}_n: L^p([0,1];Y) \to L^p([0,1];Y)$ is defined by
\[
\mathfrak{h}_n(x) = \sum_{j=1}^{2^n} \mathbf{1}_{(s^n_{j-1}, s^n_j]} \otimes \iota_{j,n}(x),
\]
where
\[
\iota_{j,n}(x) := \frac{2^n}{T} \int_{s^n_{j-1}}^{s^n_j} x(r) \, dr
\]
denotes the mean value of \(x\) over $(s^n_{j-1}, s^n_j]$.

The \textit{shifted Haar projection} $\mathfrak{h}_n^s : L^p([0,1];Y) \to L^p([0,1];Y)$ is defined by
\[
\mathfrak{h}_n^s(x) = \sum_{j=0}^{2^n-1} \mathbf{1}_{(s^n_j, s^n_{j+1}]} \otimes \iota_{j,n}(x),
\]
where $\iota_{0,n}(x) := 0$ for convention.

\begin{proposition}[Convergence rate in fractional Sobolev spaces]\label{prop:haar-conv}
Let $Y$ be a Banach space, $1 < p < \infty$, and let $x \in \mathbb{W}^{\alpha}_p([0,1];Y)$ for some $\alpha \in (0,1]$. Then there exists a constant $C>0$, depending only on $p$ and $\alpha$, such that
\[
\| \mathfrak{h}_n^s(x) - x \|_{L^p([0,1];Y)} \leq C 2^{-n\alpha} \|x\|_{{W}^{\alpha}_p},
\]
where
\[
\|x\|_{{W}^{\alpha}_p}^p = \int_0^1 \int_0^1 \frac{ | x(t) - x(s) |_Y^p }{ |t-s|^{1+\alpha p} } \, ds \, dt
\]
denotes the Gagliardo seminorm.
\end{proposition}

\begin{proof}
	For the proof we are referring to \cite[Lemma B.2]{FKEHMH}
\end{proof}

\begin{proposition}[Uniform convergence on compact sets]\label{prop:haar-uniform}
Let $\mathbb{X}$ be a Banach space compactly embedded into $L^p([0,1];Y)$, $1<p<\infty$. Let $\mathcal{X} \subset \mathbb{X}$ be a bounded set. Then for every $\varepsilon>0$, there exists $N \in \mathbb{N}$ such that for all $n \geq N$,
\[
\sup_{x \in \mathcal{X}} \| \mathfrak{h}_n^s(x) - x \|_{L^p([0,1];Y)} <\varepsilon.
\]
\end{proposition}

\begin{proof}
Since $\mathbb{X}$ embeds compactly into $L^p([0,1];Y)$ and $\mathcal{X}$ is bounded in $\mathbb{X}$, the closure $\overline{\mathcal{X}}$ of $\mathcal{X}$ in $L^p$ is compact.

{
The operators $\mathfrak{h}_n^s$ are uniformly bounded, (considering the operator norm is bounded by $M$), and  $\mathfrak{h}_n^s(x) \to x$ strongly in $L^p([0,1];Y)$.  
Now, let $\varepsilon>0$ and set $\delta:=\frac{\varepsilon}{3(M+1)}$. By the compactness of $\overline{\mathcal{X}}$, there exists a finite covering of open balls with radii $\delta$ and centers ${x_1, \ldots,x_N}\subset \overline{\mathcal{X}}$. For every $x\in \mathcal{X}$, there is $i$ with $\|x-x_i\|_{L^p([0,1];Y)}<\delta$. For each $i$, $\mathfrak{h}_n^s x_i \to x_i$, so there exists $N\in\N$ such that for each $n\geq N$ and $i=1,\ldots ,N$, we have
\begin{align*}
	\| \mathfrak{h}_n^s(x_i) - x_i \|_{L^p([0,1];Y)} <\frac{\varepsilon}{3}.
\end{align*} Now, for any $x\in \mathcal{X}$ and choose $i$ with $\|x-x_i\|<\delta$. For $n\geq N$, we have
\begin{align*}
&	\|\mathfrak{h}_n^s(x)-x\|_{L^p([0,1];Y)} 
\\&	\leq 	\|\mathfrak{h}_n^s(x-x_i)\|_{L^p([0,1];Y)}+ \|\mathfrak{h}_n^s(x_i)-x_i\|_{L^p([0,1];Y)}+\|x_i-x\|_{L^p([0,1];Y)}
\\&  \leq M \|x-x_i\|_{L^p([0,1];Y)}+\frac{\varepsilon}{3}+\|x-x_i\|_{L^p([0,1];Y)} 
\\&\leq (M+1)\delta+\frac{\varepsilon}{3} = \frac{2\varepsilon}{3}+\frac{\varepsilon}{3} = \varepsilon. 
\end{align*}Finally, taking the supremum over $x\in \mathcal{X}$ produce 
\begin{align*}
	\sup_{x \in \mathcal{X}} \| \mathfrak{h}_n^s(x) - x \|_{L^p([0,1];Y)} < \varepsilon,
\end{align*}for all $n\geq \N$. Since $\varepsilon>0$ was arbitrary, thus the convergence is uniform over $\mathcal{X}$.}
\end{proof}

\section{L\'evy--Ciesielski construction}\label{LC:sec} In this section, first we recall the L\'evy-Ciesielski theorem for Brownian motion (see \cite[Theorem 6.1]{Rogers}), and then provide an extension to the infinite-dimensional Wiener process.
\begin{theorem}[Lévy--Ciesielski construction: Almost sure uniform convergence on one-dimensional space and identification as Brownian motion]
	Let $(\Omega,\mathcal F,\mathbb P)$ support an i.i.d.\ family $\{\xi_{n,k}: n\ge 0, \, {0\le k<2^n}\}$ of $\mathcal N(0,1)$ random variables.
	Let $\{h_{n,k}\}$ be the (orthonormal) system  on $L^2(0,1)$  consisting of Haar wavelets and let
	\[
	S_{n,k}(t):=\int_0^t h_{n,k}(s)\,ds, \quad t\in[0,1]
	\]
	be the associated Schauder functions. Define partial sums
	\[
	B_m(t):=\xi_{0,0}\,t\;+\;\sum_{n=0}^{m}\ \sum_{k=0}^{2^n-1}\ \xi_{n,k} S_{n,k}(t),\quad t\in[0,1].
	\]
	Then:
	\begin{enumerate}
		\item $B_m\to B$ \emph{uniformly $\PP$-almost surely} on $[0,1]$, hence $B$ has continuous paths and is defined on the same probability space $(\Omega,\mathcal F,\mathbb P)$.
		In particular, 
		\begin{align*}
			\P \Big(\lim_{m\to\infty}\sup_{t\in[0,1]}\big|B_n(t)-B(t)\big|>\ep\Big)=0, \ \ \text{ for all }\ \ \ep>0;
		\end{align*}

		\item $B$ is a centred Gaussian process with covariance
		\[
		\mathbb E[B(s)B(t)]=\min\{s,t\},\quad s,t\in[0,1].
		\]
		Consequently, $B$ is a standard Brownian motion on $[0,1]$.
	\end{enumerate}
\end{theorem}

\begin{proof} The proof is available in \cite[pp. 10--12]{Rogers}
\end{proof}
The proof of the following extension relies on \cite[Theorem 6.1]{Rogers}.
\begin{theorem}[Lévy--Ciesielski construction in infinite-dimension]\label{LCCthrm}
Let	{$\{\xi^{\,i}_{k,n}\}_{0\le k<2^n,n\in \N}^{i\in \mathbb{I}}$} be a sequence of independent standard normal random variables defined on a probability space 
$(\Omega,\mathcal F,\mathbb P)$.  Let $\{\psi_i\}_{i\in \mathbb{I}}$ be the eigenfunctions of the covariance operator $Q$ corresponding to the eigenvalues $\{\lambda_i\}_{i\in \mathbb{I}}$, and the family $\{\psi_i\}_{i\in \mathbb{I}}$  forms a complete orthonormal system in $U$.
	Let $\{h_{k,n}\}$ be the (orthonormal) Haar wavelet system on $L^2(0,1)$ and let
	\[
	S_{k,n}(t):=\int_0^t h_{k,n}(s)\,ds, \quad t\in[0,1]
	\]
	be the associated Schauder functions. Define partial sums
	\[
{	W_n(t):=\sum_{i\in \mathbb{I}}\sum_{m=0}^{n}\ \sum_{(k,m)\in S_{m}}Q^\frac{1}{2}\psi_i\ \xi^{\,i}_{k,m}S_{k,m}(t),\quad t\in[0,1].}
	\]
	where $S_m=\{(k,m): k \text{ odd, } k\leq 2^m\}.$
	Then, the following hold:
	\begin{enumerate}
		\item $W_n\to W$ \emph{uniformly almost surely} on $[0,1]$, hence $W$ has continuous paths and is defined on the same probability space $(\Omega,\mathcal F,\mathbb P)$. In particular,
			\begin{align*}
			\P \Big(\lim_{m\to\infty}\sup_{t\in[0,1]}\big|W_n(t)-W(t)\big|_U>\ep\Big)=0, \ \ \text{ for all }\ \ \ep>0;
		\end{align*}
		\item $W$ is a centred Gaussian process with covariance
		{	\[
	\mathbb E\big[\la W(t),\psi_i\ra \la W(s), \psi_j\ra ]=\lambda_i \delta_{ij}\min\{s,t\},\quad s,t\in[0,1],
		\]where $\delta_{ij}$ is the Kronecker delta.}
		
		Consequently, $W$ is a Wiener process on $[0,1]$.
	\end{enumerate}
\end{theorem}
\begin{proof}
	Let us assume a sequence of Gaussian random variables $\{\xi_{k,n}^i:i\in \mathbb{I},n\in\N, k \text{ odd,}\break k\leq 2^n\}$. 
	Now, define 
	\begin{equation*}
		h_{1,n}(t)=1,
			\end{equation*}and
			\begin{equation*}
	h_{k,n}(t)=
		\left\{\begin{aligned}
			&2^{(n-1)/2}, \quad &&(k-1)2^{-n}<t\leq k2^{-n},\\
			&-2^{(n-1)/2}, \quad  &&k2^{-n}<t\leq (k+1)2^{-n},\\
			&0, &&  \text{ otherwise},
		\end{aligned}
	\right.
	\end{equation*}for $n\in\N$, $0\leq k\leq 2^n$, $k$ odd. 

Let $S_n=\{(k,n): k \text{ odd, } k\leq 2^n\}$, and define $S=\bigcup\limits_{n\in\N}S_n$. Then, $\{h_{k,n}:(k,n)\in S\}$ is a complete orthonormal system in $L^2(0,1)$. Now, we define
	\[
S_{k,n}(t):=\int_0^t h_{k,n}(s)\,ds\quad (t\in[0,1]),
\] then $S_{k,n}$ satisfies the following properties:
\begin{enumerate}
	\item $|S_{k,n}|_\infty=\sup\limits_{t\in[0,1]}|S_{k,n}(t)|=2^{-(n+1)/2}$ (represents a tent-shaped function of height $2^{-(n+1)/2})$;
	\item for each fixed $n$ and $t$, at most one index $k$ has $S_{k,n}(t)\neq 0$, (supports are disjoint on dyadic intervals).
\end{enumerate}
Now, we consider the sequence of partial sums
\[
W_n(t):=\sum_{i\in\mathbb{I}}\sum_{m=0}^{n}\ \sum_{(k,m)\in S_{m}}Q^\frac{1}{2}\psi_i\ \xi_{k,m}^i S_{k,m}(t),\qquad t\in[0,1].
\]
Let us prove (i). For any positive constant $a_n$, 
\begin{align*}\nonumber
&	\P\Big(\sup_{t\in[0,1]} | W_n(t)-W_{n-1}(t) | _U>a_n\Big)
\\&\nonumber
\leq \P\Big(\sup_{t\in[0,1]} | W_n(t)-W_{n-1}(t) | _U^2>a_n^2\Big)
\\&\nonumber
\leq \P\Bigg(\sup_{t\in[0,1]}\Big(\sum_{i\in\mathbb{I}}Q^\frac{1}{2}\psi_i \xi_{k,n}^iS_{k,n}(t),\sum_{i'\in\mathbb{I}}Q^\frac{1}{2}\psi_{i'} \xi_{k,n}^{i'}S_{k,n}(t)\Big)>a_n^2\Bigg)
\\&\nonumber
\leq \P\Bigg(\sup_{t\in[0,1]}\sum_{i\in\mathbb{I}}\lambda_i  | \psi_i | _U^2 |\xi_{k,n}^i|^2|S_{k,n}(t)|^2>a_n^2\Bigg) 
\\&\nonumber
\leq \P\Big(\sum_{i\in\mathbb{I}}\lambda_i 2^{-(n+1)}\sup_{k}|\xi_{k,n}^i|^2>a_n^2\Big)
 \ \Big(\text{using orthonormality of }\psi_i, (i) \text{ for }S_{k,n}\Big)
\\&
 \sim  \P\left(\sup_{k}|\xi_{k,n}^i|^2>\frac{2^{n+1}a_n^2}{\sum\limits_{i\in\mathbb{I}}\lambda_i }\right)   \quad \Big(\text{using Claim \ref{claim6.3}}\Big)
\\&\nonumber
\leq \sqrt{\frac{C_1}{\pi} }2^{n/2-1}a_n^{-1}e^{-\frac{2^{(n+1)}a_n^2}{C_1}}, \qquad \Big(\text{where }C_1=\sum_{i\in\mathbb{I}}\lambda_i \Big).  \label{}
\end{align*}%
Let us choose the constant $a_n$ in such a way that 
{\begin{align}\label{FC}
\sqrt{\frac{C_1}{\pi}}	\sum_{n\in\N}  2^{n/2-1}a_n^{-1}e^{-\frac{2^{(n+1)}a_n^2}{C_1}}<\infty, \text{ and }
	\sum_{n\in\N}a_n<\infty.
\end{align}}
Let us choose $a_n= \big(\mathfrak{c}n2^{-(n+1)}\big)^{\frac{1}{2}}$, where $\mathfrak{c}>2C_1\ln 2$. Then, conditions change to
\begin{align*}
	\sqrt{\frac{C_1}{2\pi\mathfrak{c}}}\sum_{n\in\N}\frac{(2e^{-\mathfrak{c}/C_1})^n}{n^{\frac{1}{2}}}<\infty, \text{ and }\sqrt{\frac{\mathfrak{c}}{2}}\sum_{n\in\N} \frac{n^{\frac{1}{2}}}{2^{n/2}}<\infty.
\end{align*}
The condition \eqref{FC}$_1$ and the Borel-Cantelli lemma ensure that the convergence is almost sure. Therefore, 
\begin{align*}
	\sup_{t\in[0,1]}|W_n(t)-W_{n-1}(t)|_U\leq a_n, \text{ for all large enough }n;
\end{align*}the second condition \eqref{FC}$_2$ guarantees that $W_n$ converges uniformly almost surely to a limit function $W$, therefore, $W$ is continuous.

 From the above arguments, we obtained that $W_n$ converges uniformly to some continuous limit  $W$. Now, we need to show that the limit $W$ is a centred Gaussian process with covariance structure $\mathbb{E}\big[\langle W(t),u\rangle \langle W(s),v\rangle\big]=s\langle Q u,v\rangle,$ for $s\leq t$. Note that, $W_n$ is a centred Gaussian process, thus the vector form $\big(W_n(t_1), W_n(t_2),\dots,W_n(t_k)\big)$ is also a centred Gaussian process and converges almost surely to $\big(W(t_1), W(t_2),\dots,W(t_k)\big)$, which also has a zero-mean Gaussian law. Moreover, the limit of covariances of $W_n$ gives the covariance of $W$
 
Let us move to the proof of covariance. We know that 
\begin{align*}
&	\mathbb{E}\Big[\langle W_n(t),u\rangle \langle W_n(s),v\rangle \Big]\\&=\EE \Bigg[ \Big\langle\sum_{i\in\mathbb{I}}\sum_{m=0}^{n}\ \sum_{(k,m)\in S_{m}}\sqrt{\lambda_i }\psi_i \ \xi_{k,m} ^iS_{k,m}(t),u\Big\rangle\\&\quad \times \Big\langle \sum_{i'\in\mathbb{I}}\sum_{m=0}^{n}\ \sum_{(k,m)\in S_{m}}\sqrt{\lambda_{i'} }\psi_{i'} \ \xi_{k,m}^{i'} S_{k,m}(s),v\Big\rangle\Bigg]
\\&=
\sum_{i\in\mathbb{I}}\langle\sqrt{\lambda_i }\psi_i ,u\rangle\langle \sqrt{\lambda_i }\psi_i ,v\rangle\Bigg(
\sum_{m=0}^{n}\ \sum_{(k,m)\in S_{m}} S_{k,m}(s) S_{k,m}(t)\Bigg)
\\&=
\sum_{i\in\mathbb{I}}\langle\sqrt{Q}\psi_i ,u\rangle\langle \sqrt{Q}\psi_i ,v\rangle\Bigg(
\sum_{m=0}^{n}\ \sum_{(k,m)\in S_{m}} S_{k,m}(s) S_{k,m}(t)\Bigg)
\\&=
\langle Qu,v\rangle\Bigg(
\sum_{m=0}^{n}\ \sum_{(k,m)\in S_{m}} S_{k,m}(s) S_{k,m}(t)\Bigg)
\end{align*}where we used the independence of $\xi_{k,n}$. The right-hand side of the above expression converges to 
\begin{align*}
	\langle Qu,v\rangle
 \sum_{(k,m)\in S} S_{k,m}(s) S_{k,m}(t)= \int_0^1 \mathds{1}_{[0,s]} (r)\mathds{1}_{[0,t]}(r)\, d r=\min\{t,s\},
\end{align*}since $\displaystyle S_{k,m}(t)=\int_0^1\mathds{1}_{[0,t]}(r)h_{k,m}(r)\, dr$ is the Fourier coefficient of $h_{k,n}$ in the representation of $\mathds{1}_{[0,t]}$ in terms of complete orthonormal system $\{h_{k,n}:(k,n)\in S\}$.
\end{proof}

\begin{claim}\label{claim6.3}
	To show
	\begin{align*}
		\P\bigg(\sup_{k\in[0,2^n],k\text{ odd}}|\xi_{k,n}^i|^2> \frac{2^{(n+1)}a_n^2}{C_1}\bigg)
		\simeq
		\sqrt{\frac{C_1}{\pi} }2^{n/2-1}a_n^{-1}e^{-\frac{2^{(n+1)}a_n^2}{C_1}}.
	\end{align*}
\end{claim}
\begin{proof}
We know that $\xi_{k,n}^i\in \mathcal{N}(0,1)$, thus the square of a standard normal random variable $|\xi_{k,n}^i|^2$ follows a $\chi_1^2$ distribution whose CDF is given by 
\begin{align*}
	F_{\chi_1^2}(x) =\P\Big(|\xi_{k,n}^i|^2\leq x\Big)=\frac{1}{\sqrt{2\pi}}\int_0^xt^{-\frac{1}{2}} e^{-t/2}\,dt. 
\end{align*}Thus, 
\begin{align*}
	\P\Big(|\xi_{k,n}^i|^2>x\Big)=1-	F_{\chi_1^2}(x) .
\end{align*}
Using the union bound, we can approximate the probability of the supremum exceeding $\frac{2^{(n+1)}a_n^2}{C_1}$. For any permissible $k$, the probability that $ |\xi_{k,n}^i|^2$ exceeds $ \frac{2^{(n+1)}a_n^2}{C_1}$ is
\begin{align*}
	\P\Big(|\xi_{k,n}^i|^2> \frac{2^{(n+1)}a_n^2}{C_1}\Big)=1-F_{\chi_1^2}\Big(\frac{2^{(n+1)}a_n^2}{C_1}\Big).
\end{align*}Note that $\xi_{k,n}^i$ are independent, the probability that none of $|\xi_{k,n}^i|^2$ exceeds $\frac{2^{(n+1)}a_n^2}{C_1}$ is 
\begin{align*}
		\P\Big(\sup_{k\in[0,2^n],k\text{ odd}}|\xi_{k,n}^i|^2\leq \frac{2^{(n+1)}a_n^2}{C_1}\Big)=\prod_{k\in[0,2^n],k\text{ odd}}\P\Big(|\xi_{k,n}^i|^2\leq \frac{2^{(n+1)}a_n^2}{C_1}\Big).
\end{align*}Using the independence, we get 
\begin{align*}
	\P\Big(\sup_{k\in[0,2^n],k\text{ odd}}|\xi_{k,n}^i|^2\leq \frac{2^{(n+1)}a_n^2}{C_1}\Big)=
	\Big[F_{\chi_1^2}\Big(\frac{2^{(n+1)}a_n^2}{C_1}\Big)\Big]^{2^{n-1}},
\end{align*}since $k$ is odd, there are $2^{n-1}$ odd indices in the range $[0,2^n]$. Therefore,
\begin{align*}
	\P\Big(\sup_{k\in[0,2^n],k\text{ odd}}|\xi_{k,n}^i|^2> \frac{2^{(n+1)}a_n^2}{C_1}\Big)=1-
	\Big[F_{\chi_1^2}\Big(\frac{2^{(n+1)}a_n^2}{C_1}\Big)\Big]^{2^{n-1}}.
\end{align*}Note that, for large $n$, the term $2^{n-1}$ grows exponentially. If $F_{\chi_1^2}\Big(\frac{2^{(n+1)}a_n^2}{C_1}\Big)$ is close to $1$, we can approximate the complementary probability using the union bound, that is, 
 \begin{align*}
 	\P\Big(\sup_{k\in[0,2^n],k\text{ odd}}|\xi_{k,n}^i|^2>  \frac{2^{(n+1)}a_n^2}{C_1}\Big)&\eqsim 
 	2^{n-1}	\P\bigg(|\xi_{k,n}^i|^2> \frac{2^{(n+1)}a_n^2}{C_1}\bigg) 	\\&= 2^{n-1}\bigg[1-F_{\chi_1^2}\Big(\frac{2^{(n+1)}a_n^2}{C_1}\Big)\bigg]. 
 \end{align*}We know that, for large $x$, the tail of $\chi_1^2$ distribution can be approximated as 
\begin{align*}
	1-F_{\chi_1^2}(x)\sim \sqrt{\frac{2}{\pi x}}e^{-x/2}.
\end{align*}Using the above fact, we obtain 
 \begin{align*}
	\P\Big(\sup_{k\in[0,2^n],k\text{ odd}}|\xi_{k,n}^i|^2>  \frac{2^{(n+1)}a_n^2}{C_1}\Big)&\eqsim 
	\frac{2^{n-1}}{\sqrt{ \frac{\pi 2^{(n+1)}a_n^2}{2C_1}}}e^{-\frac{2^{(n+1)}a_n^2}{C_1}}\\&=
	\sqrt{\frac{C_1}{\pi} }2^{n/2-1}a_n^{-1}e^{-\frac{2^{(n+1)}a_n^2}{C_1}},
\end{align*}which completes the proof.
\end{proof}
\end{appendix}

\noindent
\textbf{Acknowledgement.} E.H. gratefully acknowledges the support of the Austrian Science Foundation, Project
number: P34681. A.K. acknowledges partial funding by the Austrian Science Fund (FWF), Project
Number: P34681 (\href{https://www.fwf.ac.at/en/research-radar/10.55776/P34681}{10.55776/P34681}), P32295 (\href{https://www.fwf.ac.at/en/research-radar/10.55776/P32295}{10.55776/P32295}) and \href{https://www.fwf.ac.at/forschungsradar/10.55776/ESP4373225}{10.55776/ESP4373225}.
The research of J.M.T. was partially supported by the European Union's Horizon Europe research and innovation programme under the Marie Sk\l{}odowska-Curie Actions Staff Exchanges (Grant agreement no.~101183168 -- LiBERA, Call: HORIZON-MSCA-2023-SE-01). The authors express
their gratitude to Dr. Max Sauerbrey (Max-Planck-Institute f\"ur Mathematik in den Naturwissenschaften, Germany), for the insight that inspired the manuscript. The
authors would like to express their sincere gratitude to the reviewer for his/her
insightful comments and suggestions.

\noindent
\textbf{Disclaimer.}
Funded by the European Union. Views and opinions expressed are however those of the authors only and do not necessarily reflect those of
the European Union or the European Education and Culture Executive Agency (EACEA). Neither the European Union nor EACEA can be held responsible for them.

\noindent
\textbf{Open access:} For the purpose of open access, the author(s) has applied a Creative Commons Attribution (CC BY) public copyright license to any Author Accepted Manuscript version arising from this submission.

\medskip\noindent	\textbf{Declarations:}

\noindent 	\textbf{Ethical Approval:}   Not applicable.

\noindent  \textbf{Conflict of interest statement: } On behalf of all authors, the corresponding author states that there is no conflict of interest.

\noindent 	\textbf{Authors' contributions: } All authors have contributed equally.

\noindent 	\textbf{Funding:}  Austrian Science Fund (FWF), Project Number: P34681 (\href{https://www.fwf.ac.at/en/research-radar/10.55776/P34681}{10.55776/P34681}) (E.H.).\\
Austrian Science Fund (FWF), Project Number: P34681 (\href{https://www.fwf.ac.at/en/research-radar/10.55776/P34681}{10.55776/P34681}), P32295 (\href{https://www.fwf.ac.at/en/research-radar/10.55776/P32295}{10.55776/P32295})  and \href{https://www.fwf.ac.at/forschungsradar/10.55776/ESP4373225}{10.55776/ESP4373225} (A.K.).\\
European Union's Horizon Europe research and innovation programme under the Marie Sk\l{}odowska-Curie Actions Staff Exchanges (Grant agreement no.~101183168 -- LiBERA, Call: HORIZON-MSCA-2023-SE-01) (J.M.T.).

\noindent 	\textbf{Availability of data and materials: } Not applicable.

\end{document}